\documentclass[11pt, letterpaper, oneside,reqno]{amsart}

\usepackage{latexsym, amsmath, amssymb, amsfonts, amscd,bm}
\usepackage{amsthm}
\usepackage{t1enc}
\usepackage[mathscr]{eucal}
\usepackage{indentfirst}
\usepackage{graphicx, pb-diagram}
\usepackage{fancyhdr}
\usepackage{fancybox}
\usepackage{enumerate}
\usepackage{color}
\usepackage[all]{xy}
\usepackage{hyperref}
\usepackage{tikz-cd}
\usepackage{mathtools}
\usepackage{blkarray}
\usepackage{url}
\usepackage{float}
\usepackage{tikz}
\usepackage{mdframed}

\theoremstyle{plain}
\newtheorem{thm}{Theorem}[section]

\newtheorem{prop}[thm]{Proposition}

\newtheorem{thmx}{Theorem}

\newtheorem{lemma}[thm]{Lemma}
\newtheorem{cor}[thm]{Corollary}

\newtheoremstyle{underline}
{}        
{}              
{}              
{}    
{\large}              
{:}             
{1mm}         
{{\underline{\thmname{#1}\thmnumber{ #2}}}}  

\theoremstyle{underline}
\newtheorem*{claim*}{Claim}

\theoremstyle{definition}
\newtheorem{defi}[thm]{Definition}

\newtheorem{ex}[thm]{Example}
\newtheorem{exs}[thm]{Examples}

\theoremstyle{remark}
\newtheorem{remark}[thm]{Remark}

\newcommand{\RR}{\ensuremath{\mathbb R}}

\newcommand{\cL}{\mathcal{L}}

\newcommand{\cF}{\mathcal{F}}

\newcommand{\cO}{\mathcal{O}}

\newcommand{\pd}[1]{{\partial_{#1}}} 

\definecolor{forest}{rgb}{0,0.5,0}

\begin{document}
	
	\title{Deformations of Lagrangian submanifolds in log-symplectic manifolds}
	
	\author{Stephane Geudens}
	\address{
	\begin{tabular}[t]{@{}l@{}}
		KU Leuven, Department of Mathematics, Celestijnenlaan 200B, BE-3001 Leuven, Belgium\\
		\hspace{-0.5cm}Current address: Max-Planck-Institut f\"{u}r Mathematik, Vivatsgasse 7, 53111 Bonn, Germany 
		\end{tabular}%
	}
	\email{stephane\_geudens@hotmail.com}

	\author{Marco Zambon}
	\address{KU Leuven, Department of Mathematics, Celestijnenlaan 200B, BE-3001 Leuven, Belgium}
	\email{marco.zambon@kuleuven.be}  
	

	\begin{abstract}
		This paper is devoted to deformations of Lagrangian submanifolds contained in the singular locus of a log-symplectic manifold. We prove a normal form result for the log-symplectic structure around such a Lagrangian, which we use to extract algebraic and geometric information about the Lagrangian deformations. We show that the deformation problem is governed by a DGLA, we discuss whether the Lagrangian admits deformations not contained in the singular locus, and we give precise criteria for unobstructedness of first order deformations. We also address equivalences of deformations, showing that the gauge equivalence relation of the DGLA corresponds with the geometric notion of equivalence by Hamiltonian isotopies. We discuss the corresponding moduli space, and we prove a rigidity statement for the more flexible equivalence relation by Poisson isotopies.
	\end{abstract}

	\maketitle
	
\begin{center}
{\it Dedicated to Olga Radko}
\end{center}		
	
	\setcounter{tocdepth}{1} 
	\tableofcontents
	
	\section*{Introduction}
	Symplectic manifolds are a key concept in modern geometry and physics.
	A fundamental role in symplectic geometry is played by the
	distinguished class of Lagrangian submanifolds, as emphasized in Weinstein's  symplectic creed \cite{weinstein1981}: \emph{Everything is a Lagrangian submanifold}. 
	
	The deformation theory of Lagrangian submanifolds is well-behaved: as a consequence of Weinstein's Lagrangian neighborhood theorem \cite{Weinstein}, deformations of a Lagrangian submanifold $L$ correspond with small closed one-forms on $L$, and the moduli space under equivalence by Hamiltonian isotopies can be identified with the first de Rham cohomology group $H^1(L)$.
	
	Poisson manifolds are intimately related with symplectic geometry.  
	The non-degenerate Poisson manifolds are exactly the symplectic ones. If one relaxes the non-degeneracy condition, replacing it with a transverse vanishing condition, one obtains a larger class of Poisson manifolds, called log-symplectic manifolds:
	they are symplectic outside of their singular locus, which is a codimension-one submanifold.
	Their first appearance occurs in the work of Tsygan-Nest \cite{tsygan-nest}.
	The study of their geometry was initiated by Radko \cite{Radko}, who classified two-dimensional log-symplectic manifolds (nowadays called \emph{Radko surfaces}). Since the systematic study of their geometry in arbitrary dimension by Guillemin-Miranda-Pires \cite{miranda2}, log-symplectic manifolds  have
	attracted lots of attention. One reason for this is that, despite the presence of  singularities, they behave like symplectic manifolds in many respects. For instance, M\u{a}rcu\c{t}-Osorno Torres \cite{defslog} showed that,
	on a   compact manifold $M$,
	the space of  log-symplectic structures $\mathcal{C}^1$-close to a given one   (modulo  $\mathcal{C}^1$-small diffeomorphisms) is smooth and finite dimensional, parametrized by the second $b$-cohomology of $M$.
	
	This work originated from the following question: \emph{in   log-symplectic geometry, is the deformation theory of Lagrangian submanifolds as nicely behaved  as in   symplectic geometry?}
	
	\noindent
	For Lagrangian submanifolds $L$ transverse to the singular locus of the log-symplectic manifold,
	the answer is easily seen to be positive, as shown by Kirchhoff-Lukat \cite{CharlotteThesis}:
	a neighborhood of $L$ is equivalent to the $b$-cotangent bundle of $L$, and the Lagrangian deformations of $L$ (modulo Hamiltonian isotopy) are parametrized by the first $b$-cohomology group of $L$. In particular, the moduli space of Lagrangian deformations is smooth and finite dimensional for compact Lagrangians $L$. 
	
	This paper focuses on the opposite extreme: we assume that the Lagrangian submanifold $L^n$ is \emph{contained in the singular locus $Z$} of an orientable log-symplectic manifold $M^{2n}$. 
	Note that the $b$-calculus developed by Melrose \cite{Melrose}, which is one of the main tools in log-symplectic geometry, does not apply in our setting, due to the complete lack of transversality to $Z$.
	
	The main geometric questions we address are:
	\begin{itemize}
		\item[1)] Can $L\subset Z$ be deformed smoothly to Lagrangian submanifolds not contained in $Z$?
		\item[2)] Can a first order deformation of $L$ be extended to a  smooth  path of Lagrangian deformations?
		\item[3)] Is the moduli space of Lagrangian deformations  --  under the equivalence by Hamiltonian isotopies  -- smooth at $L$?
	\end{itemize}
	
	For ``many'' Lagrangian submanifolds $L$, the answer to 1) is positive, ensuring that the deformation problem we consider does not boil down to the case of symplectic geometry. The answer to 3) is typically negative, in contrast to the symplectic case. 
	The answer to 2) is striking, and   displays a behaviour that comes close to the symplectic case:   first order deformations are generally obstructed, but if an obvious quadratic obstruction vanishes, then they can be extended to a smooth path of deformations.
	
	\vspace{0.1cm}
	{\bf Summary of results.}
	As in many deformation problems in geometry, the
	first step consists in providing a normal form for the log-symplectic structure in a neighborhood of the Lagrangian $L$. Notice that as $L$ is contained in the singular locus, it carries a  codimension-one foliation $\cF_L$.
	Our normal form around $L$ (Cor. \ref{normal}) is constructed in two steps: we combine a normal form statement around Lagrangian submanifolds transverse to the  symplectic leaves of an arbitrary Poisson manifold (Prop. \ref{normalform}) with the normal form around the singular locus $Z$ of a log-symplectic manifold $(M,\Pi)$ due to Guillemin-Miranda-Pires \cite{miranda2}, \cite{Osorno}. Since the latter involves the modular class of $(M,\Pi)$, we also need
	to express the first Poisson cohomology of a neighborhood of $L$ in  the singular locus $Z$ in terms of $L$ alone (Cor. \ref{cor:isocoho}).	
	The modular class  is then encoded by two objects attached to $L$: 
	\begin{itemize}
		\item[a)] A class in $H^{1}(\mathcal{F}_{L})$, the first foliated de Rham cohomology.\\
		We fix a representative  $\gamma\in\Omega_{cl}^{1}(\mathcal{F}_{L})$.
		\item[b)] An element of $\mathfrak{X}(L)^{\mathcal{F}_{L}}/\Gamma(T\mathcal{F}_{L})\cong H^0(\mathcal{F}_{L})$.\\
		We fix a representative  $X\in \mathfrak{X}(L)^{\mathcal{F}_{L}}$, a vector field on $L$ that preserves the foliation and is nowhere tangent to it.
	\end{itemize}
	
	\begin{thmx}\label{thm:a}
		The log-symplectic structure in a tubular neighborhood of $L$  is isomorphic to
		\[
		\big(U\subset T^{*}\mathcal{F}_{L}\times\mathbb{R},\;\;\left(V_{vert}+V_{lift}\right)\wedge t\partial_{t}+\Pi_{can}\big).
		\]
		Here {$U$ is a neighborhood of the zero section $L$}, $\Pi_{can}$ is the canonical Poisson structure on the cotangent bundle $T^{*}\mathcal{F}_{L}$ of the foliation $\mathcal{F}_{L}$, and $t$ denotes the coordinate on $\RR$.
		Further, $V_{vert}$ is the vertical fiberwise constant vector field on $T^{*}\mathcal{F}_{L}$ which  corresponds to $\gamma\in\Gamma(T^{*}\mathcal{F}_{L})$ under the natural identification, and $V_{lift}$ is the cotangent lift of $X$.
	\end{thmx}
	
	The above normal form theorem gives an explicit model in which the Lagrangian deformations of $L$ can be investigated. We can characterize algebraically the Lagrangian deformations of $L$, as follows (Thm. \ref{equations}, Cor. \ref{DGLA}):
	\begin{thmx}\label{thm:b}
		Lagrangian deformations $\mathcal{C}^1$-close to $L$ are exactly the graphs of  sections $(\alpha,f)$ of the vector bundle $T^{*}\mathcal{F}_{L}\times\mathbb{R}\to L$ satisfying the quadratic equation \begin{equation*}
		\begin{cases}
		d_{\mathcal{F}_{L}}\alpha=0\\
		d_{\mathcal{F}_{L}}f+f(\gamma-\pounds_{X}\alpha)=0,
		\end{cases}
		\end{equation*}
		where $d_{\mathcal{F}_{L}}$ denotes the foliated de Rham differential and $\gamma, X$ are as above.
		
		Further, this equation is the Maurer-Cartan equation of a DGLA.
	\end{thmx}
	The differential graded Lie algebra mentioned above is the one introduced in greater generality by Cattaneo-Felder \cite{CaFeCo2}, and to ensure that it captures the  Lagrangian deformations we need to check that the Poisson structure of Thm. \ref{thm:a} is fiberwise entire.  
	\bigskip
	
	In turn, Thm. \ref{thm:b} has several geometric consequences. Before explaining them, we discuss briefly two of the tools we use. First, when $L$ is compact and connected, the following dichotomy about the foliation $\cF_L$ is well-known  \cite[Theorem 9.3.13]{conlon}:  either it is the foliation associated to a fibration $L\to S^1$, or all leaves are dense. This allows us to prove several statements in the compact case by considering the two cases separately.
	
	Second, the linear part of the above Maurer-Cartan equation  reads
	\begin{equation}\label{eq:introlinMC}
	d_{\mathcal{F}_{L}}\alpha=0,\;\;\;\;\;\;
	d^{\gamma}_{\mathcal{F}_{L}}f=0
	\end{equation}
	where $d^{\gamma}_{\mathcal{F}_{L}}f=d_{\mathcal{F}_{L}}f+f\gamma$ denotes the foliated de Rham differential \emph{twisted} by $\gamma$. The cohomology associated to $d^{\gamma}_{\mathcal{F}_{L}}$
	is the foliated Morse-Novikov cohomology $H^{\bullet}_{\gamma}(\mathcal{F}_{L})$. We will compute it in degree $0$ for  compact $L$ ({see Theorem}  \ref{H}). The ordinary (untwisted)   foliated cohomology  will be denoted by  $H^{\bullet}(\mathcal{F}_{L})$.

	\bigskip
	If the modular vector field can be chosen to be tangent to $L$, i.e. $[\gamma]=0\in H^{1}(\mathcal{F}_{L})$, then it is easy to see that $L$ can be deformed smoothly to Lagrangian submanifolds outside of the singular locus $Z$. At the opposite end of the spectrum we have (Cor. \ref{constrained}, Prop. \ref{denserestr}):
	\begin{thmx} 
		Assume $L$ is compact and connected.
		
		i) Suppose  $\mathcal{F}_{L}$ is the fiber foliation of a fiber bundle $p:L\rightarrow S^{1}$. If  for every leaf $B$ of $\cF_L$ we have $[\gamma|_B]\neq 0\in H^1(B)$,  
		then $\mathcal{C}^{1}$-small deformations of $L$ stay inside $Z$.  
		
		ii)  
		Suppose  $\mathcal{F}_{L}$ has dense leaves, and that $H^{1}(\mathcal{F}_{L})$ is finite dimensional. If $\gamma\in\Omega^{1}_{cl}(\mathcal{F}_{L})$ is not exact, then 
		$\mathcal{C}^{\infty}$-small deformations of $L$ stay inside $Z$.
	\end{thmx}
	
	The finite dimensionality assumption in ii) above is necessary: we show this 
	exhibiting an example, in which $L$ is the 2-torus and $\cF_{L}$ a Kronecker foliation for which the slope  $\lambda\in\mathbb{R}\setminus\mathbb{Q}$ is a Liouville number. 
	The proof of these statements relies on some functional analysis and Fourier analysis.
	\bigskip
	
	A first order deformation is a solution of eq. \eqref{eq:introlinMC}, the linear part of the  Maurer-Cartan equation.
	The deformation problem is obstructed in general: there are first order deformations which do not extend to a (formal or smooth) path of Lagrangian deformations. 
	This is detected by the classical Kuranishi criterium: given a first order deformation $(\alpha,f)$, where $\alpha\in \Omega^1(\mathcal{F}_{L})$ and $f\in C^{\infty}(L)$, the class $Kr\big([(\alpha,f)]\big)$ might not vanish.
	This class lives in the {first foliated Morse-Novikov cohomology group $H^{1}_{\gamma}(\mathcal{F}_{L})$}. For a general deformation problem, the  Kuranishi criterium is a necessary -- but not sufficient  -- condition to extend a first order deformation to a formal curve of deformations. In the case at hand however, we have the following striking result
	(Prop. \ref{prop:krunob}, Cor. \ref{cor:Krsimple}):
	
	\begin{thmx} Assume $L$ is compact and connected. The following are equivalent:
		\begin{itemize}
			\item 	 A first order deformation $(\alpha,f)$ of $L$ is \emph{smoothly} unobstructed,
			\item  $Kr\big([(\alpha,f)]\big)=0$,
			\item $\alpha$ extends to a closed one-form on $ L\setminus\mathcal{Z}_{f}$,
			the complement of the zero locus of $f$.
		\end{itemize}
	\end{thmx}
	Notice that the third condition  is independent of the data $(X,\gamma)$ encoding the modular vector field. 
	\bigskip
	
	Finally, we address moduli spaces.
	From a geometric point of view, it is natural to identify two $\mathcal{C}^1$-small Lagrangian deformations of $L$ if they are related by a Hamiltonian isotopy of the ambient log-symplectic manifold $(M,\Pi)$. We show that this is exactly the equivalence relation that the  DGLA of Thm. \ref{thm:b} induces on  Maurer-Cartan elements (Prop. \ref{prop:ham}). Thus by  eq. \eqref{eq:introlinMC}, the resulting moduli space $\mathcal{M} ^{Ham}$ has formal tangent space at $[L]$ given by
	$$
	T_{[L]}\mathcal{M} ^{Ham}=H^{1}(\mathcal{F}_{L})\oplus H^{0}_{\gamma}(\mathcal{F}_{L}).$$
	For most choices of $L$, this is an infinite dimensional vector space, while the formal tangent space to $\mathcal{M} ^{Ham}$ at Lagrangians contained in $M\setminus Z$ is  finite dimensional (at least if $L$ is compact). Hence, for most choices of $L$, the moduli space is not smooth at $[L]$. We also exhibit some choices of $L$ at which the moduli space is smooth, see \S\ref{subsubsec:hammod}.
	
	The same phenomenon occurs for the moduli space $\mathcal{M} ^{Poiss}$ obtained replacing Hamiltonian isotopies by Poisson isotopies (Prop. \ref{prop:poisequiv}). 
	When $L$ is compact with dense leaves, we show that  $L$ being infinitesimally rigid under  Poisson isotopies  (i.e. $T_{[L]}\mathcal{M} ^{Poiss}=0$) implies that  $L$ is rigid in the following sense: any sufficiently $\mathcal{C}^{\infty}$-small deformation of $L$ is related to $L$ by a Poisson diffeomorphism isotopic to the identity (Prop. \ref{rigidity}).
	
	\bigskip
\textbf{Organization of the paper.} 	
In \S\ref{sec:lagr} and \S\ref{sec:Poisvf} we provide the geometric background and prove the normal form given in  Theorem \ref{thm:a}. In
\S\ref{sec:alg} and \S\ref{sec:geom} we address the deformations of Lagrangian submanifolds in log-symplectic manifolds, exhibiting the underlying algebraic structure and drawing several geometric consequences. 
We refer to the introductory text of the individual sections for more details.

	\bigskip
	\textbf{Acknowledgements.}  
	We thank Ioan M\u{a}rcu\c{t} for his valuable input during many useful discussions. In particular, we thank him for directing  us to \cite[Theorem 9.3.13]{conlon} and \cite{Osorno}, and for suggesting the generalization in Remark \ref{generalization} and the proof of Lemma \ref{continuous}. 
	
	We acknowledge partial support by the long term structural funding -- Methusalem grant of the Flemish Government, the FWO under EOS project G0H4518N, and the FWO research project G083118N (Belgium). {S.G. also thanks the Max Planck Institute for Mathematics in Bonn for its hospitality and financial support.}

	\section{Lagrangian submanifolds in Poisson geometry}\label{sec:lagr}
	In this section, we first recall some concepts in Poisson geometry and we introduce the notion of Lagrangian submanifold. Then we prove a normal form for Poisson structures around Lagrangian submanifolds intersecting the symplectic leaves transversely (Prop. \ref{normalform}), which can be seen as an extension of Weinstein's Lagrangian neighborhood theorem from symplectic geometry. 	Our main motivation is the study of Lagrangian submanifolds contained in the singular locus of a log-symplectic manifold. In \S\ref{subsec:logsym}-\S\ref{subsec:lagrlogsym}
 we use the aforementioned result to find local and semilocal normal forms around them (Prop. \ref{coordinates} and Cor. \ref{normal}). 
	
	\subsection{Poisson structures}
	\begin{defi}
		A \emph{Poisson structure} on a manifold $M$ is a bivector field $\Pi\in\Gamma(\wedge^{2}TM)$ satisfying $[\Pi,\Pi]=0$, where $[\cdot,\cdot]$ is the Schouten-Nijenhuis bracket of multivector fields.
	\end{defi}
	
	The Schouten-Nijenhuis bracket on $\Gamma\left(\wedge^{\bullet}TM\right)$ is a natural extension of the Lie bracket of vector fields, which turns $\Gamma(\wedge^{\bullet}TM)[1]$ into a graded Lie algebra \cite[Section 1.8]{Normalforms}. 
	
	The bivector field $\Pi$ induces a bundle map $\Pi^{\sharp}:T^{*}M\rightarrow TM$, given by contraction of $\Pi$ with covectors. The rank of $\Pi$ at a point $p\in M$ is defined to be the rank of the linear map $\Pi^{\sharp}_{p}:T_{p}^{*}M\rightarrow T_{p}M$. A Poisson structure is called regular if its rank is the same at all points. Poisson structures $\Pi\in\Gamma(\wedge^{2}TM)$ of full rank correspond with symplectic structures $\omega\in\Gamma(\wedge^{2}T^{*}M)$ via $\omega\leftrightarrow-\Pi^{-1}$. In general, a Poisson manifold $(M,\Pi)$ comes with an integrable singular distribution $\text{Im}(\Pi^{\sharp})$. Each leaf $\mathcal{O}$ of the associated foliation has an induced symplectic structure, given by $\omega_{\mathcal{O}}=-\left(\Pi|_{\mathcal{O}}\right)^{-1}$.
	
	A map $\Phi:(M,\Pi_{M})\rightarrow (N,\Pi_{N})$ between Poisson manifolds is a Poisson map if $\Pi_{M}$ and $\Pi_{N}$ are $\Phi$-related, i.e. $\left(\wedge^{2}d_{p}\Phi\right)\left(\Pi_{M}\right)_{p}=\left(\Pi_{N}\right)_{\Phi(p)}$ for all $p\in M$. 
	A vector field $X$ on a Poisson manifold $(M,\Pi)$ is called Poisson if its flow consists of Poisson diffeomorphisms, or equivalently, if $\pounds_{X}\Pi=0$. Each function $f\in C^{\infty}(M)$ determines a Poisson vector field $X_{f}:=\Pi^{\sharp}(df)$, called the Hamiltonian vector field of $f$. The characteristic distribution $\text{Im}(\Pi^{\sharp})$ of a Poisson manifold $(M,\Pi)$ is generated by its Hamiltonian vector fields.
	
	Thanks to the graded Jacobi identity of the Schouten-Nijenhuis bracket $[\cdot,\cdot]$, the operator $[\Pi,\cdot]:\Gamma\left(\wedge^{\bullet}TM\right)\rightarrow\Gamma\left(\wedge^{\bullet+1}TM\right)$ squares to zero. The cohomology of the resulting cochain complex $\left(\Gamma\left(\wedge^{\bullet}TM\right),[\Pi,\cdot]\right)$ is the Poisson cohomology of $(M,\Pi)$, which we denote by $H^{\bullet}_{\Pi}(M)$. The cohomology groups in low degrees have geometric interpretations, see for instance \cite[Section 2.1]{Normalforms}. We will only encounter the first cohomology group $H^{1}_{\Pi}(M)$, which is the quotient of the space of Poisson vector fields by the space of Hamiltonian vector fields.
	
	The modular class of $(M,\Pi)$ is a distinguished element in $H^{1}_{\Pi}(M)$ which will play a key role in this paper. It is defined as follows: upon choosing a volume form $\mu\in\Omega^{top}(M)$, there is a unique vector field $V_{mod}^{\mu}\in\mathfrak{X}(M)$ such that for all $f\in C^{\infty}(M)$, one has
	\[
	\pounds_{X_{f}}\mu=V_{mod}^{\mu}(f)\mu.
	\] 
	The vector field $V_{mod}^{\mu}$ is called the modular vector field associated with $\mu$. One can check that $V_{mod}^{\mu}$ is a Poisson vector field, and that choosing a different volume form $\mu'=g\mu$ changes the modular vector field $V_{mod}^{\mu}$ by a Hamiltonian vector field:
	\begin{equation}\label{changevgl}
	V_{mod}^{\mu'}=V_{mod}^{\mu}-X_{\ln|g|}.
	\end{equation}
	So the Poisson cohomology class $[V_{mod}^{\mu}]\in H^{1}_{\Pi}(M)$ is intrinsically defined; it is called the modular class of $(M,\Pi)$. A Poisson manifold is called unimodular if its modular class vanishes. If $M$ is not orientable, one can still define the modular class using densities instead of volume forms. In this paper, we will only work with modular vector fields on orientable manifolds. For more on the modular class, see \cite{modular}.
	
	\vspace{0.2cm}
	
	We also recall some useful notions from contravariant geometry \cite{realisations}. The general idea behind contravariant geometry on Poisson manifolds $(M,\Pi)$ is to replace the tangent bundle $TM$ by the cotangent bundle $T^{*}M$, using the bundle map $\Pi^{\sharp}:T^{*}M\rightarrow TM$.
	
	\begin{defi}\label{spray}
		Given a Poisson manifold $(M,\Pi)$, a Poisson spray $\chi\in\mathfrak{X}(T^{*}M)$ is a vector field on $T^{*}M$ that satisfies the following properties:
		\begin{enumerate}[i)]
			\item $p_{*}\chi(\xi)=\Pi^{\sharp}(\xi)$ for all $\xi\in T^{*}M$,
			\item $m_{t}^{*}\chi=t\chi$ for all $t>0$,
		\end{enumerate}
		where $p:T^{*}M\rightarrow M$ is the projection and $m_{t}:T^{*}M\rightarrow T^{*}M$ is multiplication by $t$. 
	\end{defi}
	
	Property $ii)$ above implies that $\chi$ vanishes on $M$, so that there exists a neighborhood $U\subset T^{*}M$ of $M$ where the flow $\phi_{\chi}$ of $\chi$ is defined up to time $1$. The contravariant exponential map of $\chi$ is defined as
	\[
	\exp_{\chi}:U\subset T^{*}M\rightarrow M:\xi\mapsto p\circ\phi_{\chi}^{1}(\xi). 
	\]
	The properties of the Poisson spray imply that $\exp_{\chi}$ fixes $M$ and that its derivative at points $x\in M$ is given by
	\[
	d_{x}\exp_{\chi}:T_{x}M\oplus T_{x}^{*}M:\rightarrow T_{x}M:(v,\xi)\mapsto v+\Pi_{x}^{\sharp}(\xi).
	\]
	By property $i)$, $\exp_{\chi}$ maps the fiber $U\cap T_x^{*}M$ into the symplectic leaf through $x$.
	Poisson sprays exist for any Poisson manifold $(M,\Pi)$. They proved to be useful in the construction  of symplectic realizations \cite{realisations} and normal forms \cite{poissontransversals}, for instance.

	\subsection{Lagrangian submanifolds of Poisson manifolds}
	\leavevmode
	\vspace{0.1cm}
	
	\subsubsection{\underline{Lagrangian submanifolds}}
	We now introduce Lagrangian submanifolds, which are the main objects of study in this paper. We will use the following definition \cite{vaisman}, \cite{grabowski}.
	\begin{defi}\label{lagrdef}
		A submanifold $L$ of a Poisson manifold $(M,\Pi)$ will be called \emph{Lagrangian} if the following equivalent conditions hold at all points $p\in L$:
		\begin{enumerate}[i)]
			\item $T_{p}L\cap T_{p}\mathcal{O}$ is a Lagrangian subspace of the symplectic vector space $\left(T_{p}\mathcal{O},\left(\omega_{\mathcal{O}}\right)_{p}\right)$.
			\item $\Pi^{\sharp}_{p}\left(T_{p}L^{0}\right)=T_{p}L\cap T_{p}\mathcal{O}$, where $T_{p}L^{0}\subset T_{p}^{*}M$ denotes the annihilator of $T_{p}L$.
		\end{enumerate}
		Here $\left(\mathcal{O},\omega_{\mathcal{O}}\right)$ denotes the symplectic leaf through the point $p$. 
	\end{defi}
	
	In case $(M,\Pi)$ is symplectic, this definition reduces to the usual notion of Lagrangian submanifold in symplectic geometry. Another special case of interest is when the manifold $L$ has clean intersection with the leaves of $(M,\Pi)$; then $L$ is Lagrangian in $M$ exactly when its intersection with each leaf is Lagrangian inside the leaf, in the sense of symplectic geometry. 
	
	Coisotropic submanifolds of a Poisson manifold $(M,\Pi)$ are defined similarly, replacing ``Lagrangian" by ``coisotropic" in i) and replacing equality by the inclusion $\subset$ in ii). While coisotropic submanifolds have received lots of attention, Lagrangian submanifolds only appear rarely in the context of Poisson geometry. In this regard, there seems to be no standard definition for Lagrangian submanifolds $L\subset (M,\Pi)$. Another definition that appears in the literature uses the condition $\Pi^{\sharp}(TL^{0})=TL$ (e.g. \cite{dediego}). Notice that the latter definition is more restrictive than our Definition \ref{lagrdef}, since it imposes that connected components of $L$ are contained in symplectic leaves and are Lagrangian therein.

	\begin{exs}
		\begin{enumerate}[a)]
			\item The symplectic foliation associated with the Lie-Poisson structure on $\mathfrak{so}_{3}^{*}\cong\mathbb{R}^{3}$ consists of concentric spheres of radius $r\geq 0$. So a plane in $\mathfrak{so}_{3}^{*}$ is Lagrangian exactly when it passes through the origin.
			\item Let $(M,\Pi)$ be a regular Poisson manifold of rank $2k$, and let $\Phi:(M,\Pi)\rightarrow (N,0)$ be a proper surjective Poisson submersion of maximal rank, i.e. $\dim N=\dim M-k$. Assuming that the fibers of $\Phi$ are connected, they are Lagrangian tori contained in the symplectic leaves of $(M,\Pi)$ \cite[Theorem 2.6]{liouville}. 
			
			\item {It is well-known that the graph of a Poisson map $\Phi:(M,\Pi_M)\rightarrow(N,\Pi_N)$ is coisotropic in the product $(M\times N,\Pi_M-\Pi_N)$. If $\Phi$ is additionally an immersion\footnote{More generally, if $\Phi$ restricts to an immersion on each leaf of $(M,\Pi_M)$, then its graph is Lagrangian.}, then   its graph is in fact a Lagrangian submanifold of $(M\times N,\Pi_M-\Pi_N)$.}
			
			\item Let $G$ be a Lie group acting on a Poisson manifold $(M,\Pi)$ with equivariant moment map $J:M\rightarrow\mathfrak{g}^{*}$. Assume the action is {locally} free on $J^{-1}(0)$. Then $J^{-1}(0)\subset (M,\Pi)$ is coisotropic and transverse to the symplectic leaves \cite[Lemma 3.8]{GZ}. If the leaves it meets have dimension equal to $2\dim\mathfrak{g}$, then $J^{-1}(0)$ is Lagrangian.
		\end{enumerate}
	\end{exs}

	\subsubsection{\underline{Normal forms}}
	We will  prove a normal form theorem around Lagrangian submanifolds $L\subset (M,\Pi)$ that are transverse to the symplectic leaves, extending Weinstein's Lagrangian neighborhood theorem \cite{Weinstein} from symplectic geometry. This is done in Proposition \ref{normalform} below. The following lemma reduces the problem to Lagrangian submanifolds of regular Poisson manifolds.
	
	\begin{lemma}\label{lem:regPoisnearby}
		Let $(M,\Pi)$ be a Poisson manifold, and $L\subset (M,\Pi)$ a Lagrangian submanifold transverse to the symplectic leaves. Then there exists a neighborhood $U$ of $L$ such that $\Pi|_{U}$ is regular.
	\end{lemma}
	\begin{proof}
		The conditions that $L$ be Lagrangian and transverse to the leaves of $(M,\Pi)$ determine the dimension of the leaves that $L$ meets. Indeed, if $p\in L$ and $\mathcal{O}$ is the leaf through $p$, then
		\begin{align*}
		\dim(T_{p}L)+\dim(T_{p}\mathcal{O})&=\dim(T_{p}L+T_{p}\mathcal{O})+\dim(T_{p}L\cap T_{p}\mathcal{O})\\
		&=\dim(T_{p}M)+\frac{1}{2}\dim(T_{p}\mathcal{O}),
		\end{align*}
		so that $\dim(\mathcal{O})=2(\dim(M)-\dim(L))$. It now suffices to show that there is an open neighborhood $U$ of $L$ that is contained in the saturation of $L$ (i.e. the union of the leaves that intersect $L$).
		
		To construct such a neighborhood, fix a Poisson spray $\chi\in\mathfrak{X}(T^{*}M)$. Let $E:=\Pi^{\sharp}(TL^{0})$, which is a vector bundle of rank $\dim(M)-\dim(L)$ because of the transversality requirement. Choosing a complement to $E$ in $TM|_{L}$, we get a direct sum decomposition
		\begin{equation}\label{directsum}
		T^{*}M|_{L}=E^{*}\oplus E^{0}.
		\end{equation}
		We claim that the contravariant exponential map
		\begin{equation*}
		\exp_{\chi}:E^{*}\rightarrow M
		\end{equation*}
		maps a neighborhood $V\subset E^{*}$ of $L$ diffeomorphically onto a neighborhood $U\subset M$ of $L$. By property $i)$ in Definition \ref{spray}, this neighborhood $U$ is then automatically contained in the saturation of $L$. To prove the claim, it suffices to show injectivity of the derivative of $\exp_{\chi}$ along the zero section
		\begin{equation}\label{dexp}
		d_{x}\exp_{\chi}:T_{x}L\oplus E_{x}^{*}\rightarrow T_{x}M:(v,\xi)\mapsto v+\Pi_{x}^{\sharp}(\xi).
		\end{equation}
		To do so, note that if $\Pi_{x}^{\sharp}(\xi)=-v\in T_{x}L$, then $\xi\in\left(\Pi_{x}^{\sharp}\right)^{-1}(T_{x}L)=E_{x}^{0}$. But also $\xi\in E_{x}^{*}$, so that $\xi=0$ because of the direct sum \eqref{directsum}. This then implies that also $v=0$, which proves injectivity of the map \eqref{dexp}. This finishes the proof.
	\end{proof}
	
	So in the following, we may assume that $L$ is Lagrangian in a regular Poisson manifold $(M,\Pi)$. In the next lemma, we put the foliation of $(M,\Pi)$ in normal form around $L$, and we construct the local model for the Poisson structure $\Pi$.

	\begin{lemma}\label{diff}
		Let $(M,\Pi)$ be a regular Poisson manifold with associated symplectic foliation $(\mathcal{F},\omega)$. Let $L\subset (M,\Pi)$ be a Lagrangian submanifold transverse to the leaves of $\mathcal{F}$, and denote by $\mathcal{F}_{L}$ the induced foliation on $L$. We then have the following:
		\begin{enumerate}[a)]
			\item There is a foliated diffeomorphism $\phi$ between a neighborhood of $L$ in $(M,\mathcal{F})$ and a neighborhood of $L$ in $(T^{*}\mathcal{F}_{L},p^{*}\mathcal{F}_{L})$, with $\phi|_{L}=\text{Id}$. Here $T^{*}\mathcal{F}_{L}$ denotes the union of the cotangent bundles of the leaves of $\mathcal{F}_{L}$, and
			$p^{*}\mathcal{F}_{L}$ is the pullback foliation of $\mathcal{F}_{L}$ by the bundle projection $p:T^{*}\mathcal{F}_{L}\rightarrow L$. 
			\item There is a canonical Poisson structure $\Pi_{can}$ on the total space $T^{*}\mathcal{F}_{L}$ which gives rise to the foliation $p^{*}\mathcal{F}_{L}$.
		\end{enumerate}
	\end{lemma}
	
	\begin{proof}
		\begin{enumerate}[a)]
			\item By definition, $T\mathcal{F}_{L}$ is a Lagrangian subbundle of the symplectic vector bundle $(T\mathcal{F}|_{L},\omega|_{L})$. Let $V$ be a Lagrangian complement, i.e. $T\mathcal{F}|_{L}=T\mathcal{F}_{L}\oplus V$. The leafwise symplectic form $\omega$ gives an isomorphism of vector bundles
			\begin{equation}\label{1}
			-\omega^{\flat}:V\rightarrow T^{*}\mathcal{F}_{L}.
			\end{equation}
			Next, by choosing a fiber metric $g$ on the vector bundle $T\mathcal{F}$, we obtain a foliated exponential map $\exp_{\mathcal{F}}:U\subset T\mathcal{F}\rightarrow M$ \cite[Example 3.3.9]{CandelConlonI}. For each leaf $\mathcal{O}$ of $\mathcal{F}$, we have that $\exp_{\mathcal{F}}:U\cap T\mathcal{O}\rightarrow\mathcal{O}$ is the usual exponential map of $(\mathcal{O},g|_{T\mathcal{O}})$. Since $V\subset T\mathcal{F}|_{L}$ is a complement to $TL$ in $TM|_{L}$, the map $\exp_{\mathcal{F}}$ gives a local diffeomorphism between neighborhoods of $L$
			\begin{equation}\label{2}
			\exp_{\mathcal{F}}:V\rightarrow M.
			\end{equation}
			Composing \eqref{1} and \eqref{2} now gives a local diffeomorphism that matches the leaves of $\mathcal{F}$ with those of $p^{*}\mathcal{F}_{L}$. Clearly, this map restricts to the identity on $L$.
			\item We claim that the canonical Poisson structure $\Pi_{T^{*}L}$ on $T^{*}L$ pushes forward under the restriction map $r:T^{*}L\rightarrow T^{*}\mathcal{F}_{L}$, and that $\Pi_{can}:=r_{*}\left(\Pi_{T^{*}L}\right)$ satisfies the requirement. This is readily checked in coordinates. Take a foliated chart $(x_{1},\ldots,x_{k},x_{k+1}\ldots,x_{n})$ on $L$ such that plaques of $\mathcal{F}_{L}$ are level sets of $(x_{k+1},\ldots,x_{n})$, and let $(y_{1},\ldots,y_{n})$ be the associated fiber coordinates on $T^{*}L$. Then the restriction map $r:T^{*}L\rightarrow T^{*}\mathcal{F}_{L}$ is just the projection onto the first $n+k$ coordinates, which implies that $\Pi_{T^{*}L}=\sum_{i=1}^{n}\partial_{x_{i}}\wedge\partial_{y_{i}}$ pushes forward to a Poisson structure 
			\[
			r_{*}\left(\Pi_{T^{*}L}\right)=\sum_{i=1}^{k}\partial_{x_{i}}\wedge\partial_{y_{i}}.
			\]
			Clearly, the Poisson manifold $\left(T^{*}\mathcal{F}_{L},\Pi_{can}\right)$ decomposes into symplectic leaves as follows:
			\begin{equation}\label{decomposition}
			(T^{*}\mathcal{F}_{L},\Pi_{can})=\coprod_{\mathcal{O}\in\mathcal{F}_{L}}\left(T^{*}\mathcal{O},\omega_{T^{*}\mathcal{O}}\right),
			\end{equation}
			where $\omega_{T^{*}\mathcal{O}}$ denotes the canonical symplectic form on $T^{*}\mathcal{O}$. This finishes the proof.
		\end{enumerate}
	\end{proof}
	
	We can now show that $(M,\Pi)$ and $\left(T^{*}\mathcal{F}_{L},\Pi_{can}\right)$ are Poisson diffeomorphic near $L$. If $\phi:(M,\mathcal{F})\rightarrow(T^{*}\mathcal{F}_{L},p^{*}\mathcal{F}_{L})$ denotes the diffeomorphism constructed in Lemma \ref{diff} (defined on a neighborhood of $L$), then we have that
	\begin{equation}\label{agree}
	\left(\phi_{*}\Pi\right)|_{L}=\Pi_{can}|_{L}.
	\end{equation}
	This can be checked by direct computation, but instead we refer to the proof of Weinstein's Lagrangian neighborhood theorem in \cite{Weinstein}, as we are just applying Weinstein's construction leaf by leaf. In some detail, we consider the restriction $\phi:\left(\mathcal{S},\omega_{\mathcal{S}}\right)\rightarrow\left(T^{*}(L\cap\mathcal{S}),\omega_{T^{*}(L\cap\mathcal{S})}\right)$ for each leaf $\mathcal{S}\in\mathcal{F}$, and the usual argument of the Lagrangian neighborhood theorem shows that $\phi^{*}\omega_{T^{*}(L\cap\mathcal{S})}$ and $\omega_{\mathcal{S}}$ agree along $L\cap\mathcal{S}$. This immediately implies the equality \eqref{agree}.
	
	Having established \eqref{agree}, we need an appropriate version of Moser's theorem in order to construct a Poisson diffeomorphism between neighborhoods of $L$ in $(M,\Pi)$ and $\left(T^{*}\mathcal{F}_{L},\Pi_{can}\right)$. This in turn requires a foliated version of the relative Poincar\'{e} lemma. Both statements already appeared in the literature; we state them here for the reader's convenience.
	
	\begin{lemma}\cite[Proposition 3.3]{poincare}\label{poin}
		Let $(N,\mathcal{F})$ be a foliated manifold, and let $p:M\rightarrow N$ be a vector bundle over $N$. Denote by $\mathcal{F}':=p^{*}(\mathcal{F})$ the pullback foliation of $\mathcal{F}$. Suppose that $\alpha\in\Gamma\left(\wedge^{k}T^{*}\mathcal{F}'\right)$ is a closed foliated $k$-form whose pullback $i^{*}\alpha$ to $(N,\mathcal{F})$ vanishes. Then there exists a foliated $(k-1)$-form $\beta\in\Gamma\left(\wedge^{k-1}T^{*}\mathcal{F}'\right)$ such that $d_{\mathcal{F}'}\beta=\alpha$ and $\beta|_{N}=0$.
	\end{lemma}
	
	\begin{lemma}\cite[Lemma 5]{cm}\label{moser}
		Let $(M,\mathcal{F},\omega)$ be a symplectic foliation. Consider a foliated $1$-form $\alpha\in\Omega^{1}(\mathcal{F})$ satisfying $\alpha|_{N}=\left(d_{\mathcal{F}}\alpha\right)|_{N}=0$ for some submanifold $N\subset M$. Then $\omega+d_{\mathcal{F}}\alpha$ is non-degenerate in a neighborhood $U$ of $N$, and the resulting symplectic foliation $\left(U,\mathcal{F}|_{U},\omega|_{U}+\left(d_{\mathcal{F}}\alpha\right)|_{U}\right)$ is isomorphic around $N$ to $(M,\mathcal{F},\omega)$ by a foliated diffeomorphism that is the identity on $N$.
	\end{lemma}
	
	Altogether, we obtain the following normal form around Lagrangian submanifolds transverse to the symplectic leaves of a Poisson manifold.
	
	\begin{prop}[Local model around Lagrangians transverse to symplectic leaves]\label{normalform}
		Given a Poisson manifold $(M,\Pi)$, let $L\subset (M,\Pi)$ be a Lagrangian submanifold transverse to the symplectic leaves. Denote by $\mathcal{F}_{L}$ the induced foliation on $L$. Then a neighborhood of $L$ in $(M,\Pi)$ is Poisson diffeomorphic with a neighborhood of $L$ in $\left(T^{*}\mathcal{F}_{L},\Pi_{can}\right)$, through a diffeomorphism that restricts to the identity on $L$. 
	\end{prop}
	\begin{proof}
		By Lemma \ref{lem:regPoisnearby}, we can assume that $(M,\Pi)$ is regular, with underlying foliation $\mathcal{F}$. By Lemma \ref{diff} and \eqref{agree}, there exists a foliated diffeomorphism between neighborhoods of $L$, $\phi:U\subset(M,\mathcal{F})\rightarrow V\subset(T^{*}\mathcal{F}_{L},p^{*}\mathcal{F}_{L})$, satisfying
		\[
		\left(\phi_{*}\Pi\right)|_{L}=\Pi_{can}|_{L}\hspace{0.5cm}\text{and}\hspace{0.5cm}\phi|_{L}=\text{Id}.
		\]
		Denote by $\omega,\tilde{\omega}\in\Omega^{2}(p^{*}\mathcal{F}_{L}|_{V})$ the leafwise symplectic forms on $V\subset T^{*}\mathcal{F}_{L}$ corresponding with the Poisson structures $\Pi_{can}$ and $\phi_{*}\Pi$, respectively. Since $\tilde{\omega}-\omega$ is closed and the restriction $(\tilde{\omega}-\omega)|_{L}$ vanishes, we can apply Lemma \ref{poin}: shrinking $V$ if necessary, we obtain that $\tilde{\omega}-\omega=d_{p^{*}\mathcal{F}_{L}}\beta$ for some $\beta\in\Omega^{1}(p^{*}\mathcal{F}_{L}|_{V})$ satisfying $\beta|_{L}=0$.  Lemma \ref{moser} gives an isomorphism of symplectic foliations $\psi:\left(V,p^{*}\mathcal{F}_{L}|_{V},\tilde{\omega}|_{V}\right)\rightarrow\left(\psi(V),p^{*}\mathcal{F}_{L}|_{\psi(V)},\omega|_{\psi(V)}\right)$ such that $\psi|_{L}=\text{Id}$, again shrinking $V$ if necessary. The map $\psi\circ\phi:\left(U,\Pi|_{U}\right)\rightarrow\left({\psi(V)},\Pi_{can}|_{{\psi(V)}}\right)$ now satisfies the criteria.
	\end{proof}

	\begin{remark}
		One can also obtain Proposition \ref{normalform} by applying some more general results that appeared in \cite{CZ}. There one shows the following:
		\begin{itemize}
			\item \cite[Theorem 8.1]{CZ} Let $(M,D)$ be a smooth Dirac manifold. If $D\cap TM$ has constant rank, then $(M,D)$ can be embedded coisotropically into a Poisson manifold $(P,\Pi)$. Explicitly, denote $E:=D\cap TM$ and define $P$ to be the total space of the vector bundle $\pi:E^{*}\rightarrow M$. Choosing a complement to $E$ inside $TM$ gives an embedding $i:E^{*}\hookrightarrow T^{*}M$. Then the Dirac structure $e^{i^{*}\omega_{T^{*}M}}\left(\pi^{*}D\right)$, obtained by pulling back $D$ along $\pi$ and applying the gauge transformation by $i^{*}\omega_{T^{*}M}$, defines a Poisson structure $\Pi$ on a neighborhood of $M$ in $E^{*}$. It has the desired properties: $M\subset(P,\Pi)$ is coisotropic and the Dirac structure $D_{\Pi}$ pulls back to $D$ on $M$.
			\item \cite[Proposition 9.4]{CZ} Suppose we are given a Dirac manifold $(M,D)$ for which $D\cap TM$ has constant rank $k$, and let $(P_{1},\Pi_{1})$ and $(P_{2},\Pi_{2})$ be Poisson manifolds of dimension $\dim(M)+k$ in which $(M,D)$ embeds coisotropically. Assume moreover that the presymplectic leaves of $(M,D)$ have constant dimension. Then $\left(P_{1},\Pi_{1}\right)$ and ${\left(P_{2},\Pi_{2}\right)}$ are Poisson diffeomorphic around $M$.
		\end{itemize}
		In our situation, we have a Lagrangian submanifold $i:L\hookrightarrow(M,\Pi)$ transverse to the symplectic leaves of $(M,\Pi)$, so the pullback $i^{*}{D_{\Pi}}$ is a smooth Dirac structure on $L$. Moreover, $i^{*}{D_{\Pi}}\cap TL$ has constant rank since it is given by $\Pi^{\sharp}(TL^{0})=T\mathcal{F}_{L}$. The procedure in described in the first bullet point above then yields exactly the local model $\left(T^{*}\mathcal{F}_{L},\Pi_{can}\right)$. 
		
		Now $(L,i^{*}{D_{\Pi}})$ is embedded coisotropically in $(M,\Pi)$ and in $\left(T^{*}\mathcal{F}_{L},\Pi_{can}\right)$, both of which have dimension equal to $\dim(L)+rk(T\mathcal{F}_{L})$. The presymplectic leaves of $(L,i^{*}{D_{\Pi}})$ have constant dimension, since they are just the leaves of $\mathcal{F}_{L}$. Applying the second bullet point above then shows that $(M,\Pi)$ and $\left(T^{*}\mathcal{F}_{L},\Pi_{can}\right)$ are Poisson diffeomorphic around $L$.
	\end{remark}
	
	\begin{figure}[h]
		\begin{center}
			\begin{tikzpicture}[scale=1]
			\draw[thick][fill=gray!15!white] (-5.5,-1) -- (-4.5,1) -- (2.5,1) --   (1.5,-1) -- cycle; 
			\node[below,right] at (2,0) {$L$};
			
			\draw[purple,thick] (0,-1) -- (1,1);
			\node[below,left] at (0.75,0.5) {\textcolor{purple}{$\cF_L$}};
			\draw (0,-2.5) -- (0,1.5)--(1,3.5) -- (1,-0.5) --cycle;
			\node[below,right] at (1.1,3) {$T^*\cF_L$};
			
			\draw[purple,thick] (-2,-1) -- (-1,1);
			\node[below,left] at (-2+0.75,0.5) {\textcolor{purple}{$\cF_L$}};
			\draw[-] (-2,-2.5) -- (-2,1.5)--(-1,3.5) -- (-1,-0.5) --cycle;
			
			\draw[purple,thick] (-4,-1) -- (-3,1);
			\node[below,left] at (-4+0.75,0.5) {\textcolor{purple}{$\cF_L$}};
			\draw[-] (-4,-2.5) -- (-4,1.5)--(-3,3.5) -- (-3,-0.5) --cycle;   
			\end{tikzpicture}
		\end{center}
		\caption{The foliation $\cF_L$ and vector bundle $T^*\cF_L$.}
	\end{figure}
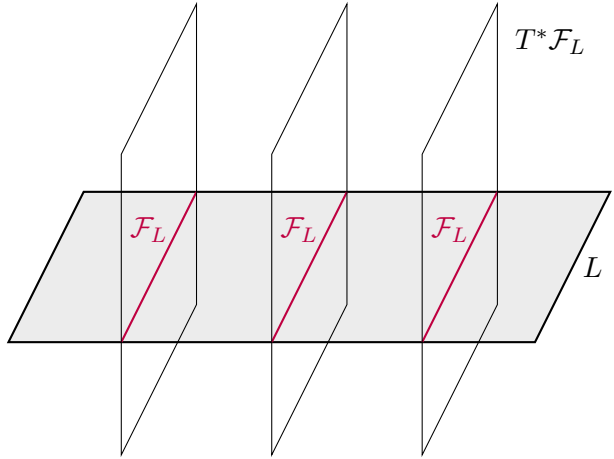
	
	Proposition \ref{normalform} implies that $\mathcal{C}^{1}$-small deformations of a Lagrangian $L\subset (M,\Pi)$ transverse to the leaves correspond with Lagrangian sections of $\left(T^{*}\mathcal{F}_{L},\Pi_{can}\right)$. Thanks to the decomposition \eqref{decomposition}, these can be studied using well-known results from symplectic geometry about Lagrangian sections in cotangent bundles. We obtain that deformations of $L\subset (M,\Pi)$ are classified by the first foliated cohomology group $H^{1}(\mathcal{F}_{L})$.
	
	\begin{cor}\label{moduliZ}
		Given a Poisson manifold $(M,\Pi)$, let $L\subset (M,\Pi)$ be a Lagrangian submanifold transverse to the symplectic leaves. Denote by $\mathcal{F}_{L}$ the induced foliation on $L$. 
		\begin{itemize}
			\item The graph of $\alpha\in\Gamma\left(T^{*}\mathcal{F}_{L}\right)$ is Lagrangian in $\left(T^{*}\mathcal{F}_{L},\Pi_{can}\right)$ exactly when $d_{\mathcal{F}_{L}}\alpha=0$.
			\item The graphs of closed foliated one-forms $\alpha,\beta\in\Gamma\left(T^{*}\mathcal{F}_{L}\right)$ are related by a Hamiltonian diffeomorphism exactly when $[\alpha]=[\beta]$ in $H^{1}(\mathcal{F}_{L})$.
		\end{itemize} 
	\end{cor}
	
	\subsection{Log-symplectic structures}\label{subsec:logsym}
	\leavevmode
	\vspace{0.1cm}
	
	The rest of this paper is devoted to a specific class of Poisson structures, called log-symplectic structures. These are generically symplectic, except at some singularities where the bivector drops rank in a controlled way.
	
	\begin{defi}
		A Poisson structure $\Pi$ on a manifold $M^{2n}$ is called \emph{log-symplectic} if $\wedge^{n}\Pi$ is transverse to the zero section of the line bundle $\wedge^{2n}TM$.
	\end{defi}
	
	A log-symplectic structure $\Pi$ is symplectic everywhere, except at points lying in the set $Z:=\left(\wedge^{n}\Pi\right)^{-1}(0)$, called the singular locus of $(M,\Pi)$. If $Z$ is nonempty, then it is a smooth hypersurface by the transversality condition. In that case, $Z$ is a Poisson submanifold of $(M,\Pi)$ with an induced Poisson structure that is regular of corank-one.
	
	The geometry of the singular locus $(Z,\Pi|_{Z})$ has some nice features. The foliation of $\Pi|_{Z}$ is unimodular, i.e. defined by a closed one-form $\theta\in\Omega^{1}(Z)$, and the leafwise symplectic form extends to a closed two-form $\omega\in\Omega^{2}(Z)$. The pair $(\theta,\omega)$ defines a cosymplectic structure on $Z$. The existence of such a pair is equivalent with the existence of a Poisson vector field on $Z$ that is transverse to the leaves of $\Pi|_{Z}$ \cite{miranda1}. One can obtain such a vector field by restricting a modular vector field on $(M,\Pi)$ to $Z$ \cite{miranda2}.
	
	\begin{ex}\label{ex:R2nlog}
		The standard example of a log-symplectic manifold is $\mathbb{R}^{2n}$ with coordinates $(x_{1},y_{1},\ldots,x_{n},y_{n})$ and Poisson structure $\Pi=\partial_{x_{1}}\wedge y_{1}\partial_{y_{1}}+\sum_{i=2}^{n}\partial_{x_{i}}\wedge\partial_{y_{i}}$. It follows from Weinstein's splitting theorem that any log-symplectic structure looks like this near a point in its singular locus. In this example, the vector field $\partial_{x_{1}}$ is the modular vector field corresponding with the volume form $dx_{1}\wedge dy_{1}\wedge\cdots\wedge dx_n\wedge dy_n$. It is indeed transverse to the symplectic leaves of $Z=\{y_{1}=0\}$, which are the level sets of $x_{1}$.
	\end{ex}
	
	The importance of modular vector fields is apparent in the following normal form result, which describes the log-symplectic structure in a neighborhood of its singular locus \cite{miranda2}, \cite[Prop. 4.1.2]{Osorno}.
	
	\begin{prop}[Local form around singular locus]\label{normZ}
		Let $\Pi$ be a log-symplectic structure on an orientable manifold $M$, with singular locus $(Z,\Pi|_{Z})$. Let $V_{mod}\in\mathfrak{X}(M)$ be a modular vector field on $M$. Then there is a tubular neighborhood $U\subset Z\times\mathbb{R}$ of $Z$, in which $Z$ corresponds to $t=0$, such that
		\[
		\Pi|_{U}=V_{mod}|_{Z}\wedge t\partial_{t}+\Pi|_{Z}.
		\]
	\end{prop}
	
	Log-symplectic structures can alternatively be viewed as symplectic forms on a suitable Lie algebroid. To any $b$-manifold $(M,Z)$ consisting of a manifold $M$ and a hypersurface $Z\subset M$, one can associate a Lie algebroid ${}^{b}TM$ whose sections are the vector fields on $M$ that are tangent to $Z$. Lie algebroid $2$-forms $\omega\in\Gamma\left(\wedge^{2}\left(^{b}T^{*}M\right)\right)$ that are closed and non-degenerate are called $b$-symplectic forms. Having a log-symplectic structure $\Pi$ on $M$ with singular locus $Z$ is equivalent to having a $b$-symplectic form on $(M,Z)$ \cite{miranda2}. This point of view allows one to study log-symplectic structures using symplectic techniques.
	
	\subsection{Lagrangian submanifolds of log-symplectic manifolds}\label{subsec:lagrlogsym}
	\leavevmode
	\vspace{0.1cm}
	
	We now focus on Lagrangian submanifolds $L$ of log-symplectic manifolds $(M,Z,\Pi)$. Lagrangians transverse to the degeneracy locus $Z$ can be treated using the $b$-geometry point of view, which reduces their study to symplectic geometry. Indeed, the submanifold $L$ is naturally a $b$-manifold $(L,L\cap Z)$ and the condition that $L$ be Lagrangian (in the sense of Def. \ref{lagrdef}) is equivalent with the requirements 
	\[
	\begin{cases}
	{}^{b}i^{*}\omega=0\\
	\dim(L)=\frac{1}{2}\dim(M)
	\end{cases},
	\]
	where $\omega$ is the $b$-symplectic form defined by $\Pi$ and $i:(L,L\cap Z)\hookrightarrow (M,Z)$ is the inclusion. In \cite{CharlotteThesis}, one shows that a neighborhood of $L$ in $(M,\omega)$ is $b$-symplectomorphic with a neighborhood of $L$ in its $b$-cotangent bundle ${}^{b}T^{*}L$, endowed with the canonical $b$-symplectic form. As a consequence, the moduli space of Lagrangian deformations of $L$ under Hamiltonian equivalence can be identified with the first $b$-cohomology group ${}^{b}H^{1}(L)$. All of this is in complete analogy with what happens in symplectic geometry.

	\vspace{0.2cm}
	We will consider Lagrangians at the other extreme, i.e. those that are contained in the singular locus of a log-symplectic manifold $(M^{2n},Z,\Pi)$. If $L$ is such a Lagrangian and $\mathcal{O}^{2n-2}$ is the leaf through $p\in L$, then we have
	\[
	\dim(T_{p}L)=\dim(T_{p}L+T_{p}\mathcal{O})-n+1,
	\]
	where $2n-2\leq\dim(T_{p}L+T_{p}\mathcal{O})\leq 2n-1$. So either $\dim(L)=n-1$ and connected components of $L$ lie inside symplectic leaves, or $\dim(L)=n$ and $L$ is transverse to the leaves in $Z$. In the rest of this note, we will deal with Lagrangians of the second kind: 
	\begin{center}\emph{middle dimensional Lagrangian submanifolds contained in the singular locus.}\end{center}
	
	\begin{remark}
		More generally, instead of middle dimensional Lagrangian submanifolds, one could consider middle dimensional coisotropic submanifolds $C\subset(M,Z,\Pi)$.
		Although these two notions coincide for submanifolds transverse to the degeneracy locus $Z$, they are not equivalent in general -- in particular, they are not equivalent in the setup we consider.

		An example of middle dimensional coisotropic $C$ contained in $Z$ which is not  Lagrangian, is the following.
		Take $M=\RR^4$ and $\Pi=\partial_{x_1}\wedge y_1\partial_{y_1}+\partial_{x_2}\wedge\partial_{y_2}$, take $C$ given by the constraints $x_1-y_2^3=0$ and  $y_1=0$. It is coisotropic because the Poisson bracket of these constraints is $y_1$, thus again a constraint. It is not Lagrangian because $T_pC=T_p\mathcal{O}$
		at points $p$ of $C$ where $y_2$ vanishes, where $\mathcal{O}$ denotes the (2-dimensional) symplectic leaf through $p$.
	\end{remark}

	\begin{ex}\label{ex}
		In the local model $\left(\mathbb{R}^{2n},x_{1},y_{1},\ldots,x_{n},y_{n}\right)$ with its standard log-symplectic structure $\Pi=\partial_{x_{1}}\wedge y_{1}\partial_{y_{1}}+\sum_{i=2}^{n}\partial_{x_{i}}\wedge\partial_{y_{i}}$, the submanifold $L=\{y_{1}=\cdots=y_{n}=0\}$ is Lagrangian of middle dimension, contained in the singular locus.
	\end{ex}
	
	Example \ref{ex} is in fact the local model for any Lagrangian $L^{n}\subset Z\subset (M^{2n},\Pi)$.

	\begin{prop}[Local form around a point]\label{coordinates}
		Let $(M^{2n},Z,\Pi)$ be a log-symplectic manifold and let $L^{n}\subset Z$ be a Lagrangian submanifold. Around any point $p\in L$, there exist coordinates $(x_{1},y_{1},\ldots,x_{n},y_{n})$ such that
		\[
		\begin{cases}
		Z=\{y_{1}=0\}\\
		\Pi=\partial_{x_{1}}\wedge y_{1}\partial_{y_{1}}+\sum_{i=2}^{n}\partial_{x_{i}}\wedge\partial_{y_{i}}\\
		L=\{y_{1}=\cdots=y_{n}=0\}
		\end{cases}.
		\]
	\end{prop}
	\begin{proof}
		Applying Prop. \ref{normalform} and  Prop. \ref{normZ}  locally around $p$ shows that there exists a coordinate system $(U; x_{1},t, x_{2},y_{2}\ldots,x_{n},y_{n})$ such that 
		\begin{equation}\label{pi}
		\Pi|_{U}=V_{mod}|_{U\cap Z}\wedge t\partial_{t}+\sum_{i=2}^{n}\partial_{x_{i}}\wedge\partial_{y_{i}}.
		\end{equation}
		Here $V_{mod}$ is a locally defined modular vector field,  $L=\{t=y_2=\cdots=y_{n}=0\}$ and $Z=\{t=0\}$. If we write $V_{mod}|_{U\cap Z}$ in coordinates as
		\[
		V_{mod}|_{U\cap Z}=z(x,y)\partial_{x_{1}}+ \sum_{i=2}^{n}g_{i}(x,y)\partial_{x_{i}}+\sum_{i=2}^{n}h_{i}(x,y)\partial_{y_{i}},
		\]
		then requiring that $V_{mod}|_{U\cap Z}$ is Poisson yields that $z(x,y)$ only depends on $x_{1}$. Now $V_{mod}|_{U\cap Z}-z(x_{1})\partial_{x_{1}}$ is a Poisson vector field tangent to the leaves, so it is locally Hamiltonian. This implies that, changing to a different modular vector field, we may assume 
		\[
		V_{mod}|_{U\cap Z}=z(x_{1})\partial_{x_{1}}.
		\]
		Note here that $z(x_{1})$ is nowhere zero since $V_{mod}|_{U\cap Z}$ is transverse to the leaves. This allows us to define a new coordinate $\xi$ by
		\[
		\xi:=\int\frac{1}{z(x_{1})}dx_{1}.
		\]
		In the new coordinate system $(\xi,t,x_{2},y_{2},\ldots,x_{n},y_{n})$, the expression \eqref{pi} becomes
		\[
		\Pi=\partial_{\xi}\wedge t\partial_{t}+\sum_{i=2}^{n}\partial_{x_{i}}\wedge\partial_{y_{i}},
		\]
		so these coordinates satisfy the criteria.
	\end{proof}

{
We display two classes of middle dimensional  Lagrangian submanifolds contained in the singular locus.
\begin{ex}[Mapping tori]
Assume the singular locus $Z^{2n-1}$ is compact and has a compact leaf $S$. Then $Z$
is the mapping torus of a symplectomorphism $\phi\colon S\to S$, i.e.
$$Z=( [0,1]\times S)/\sim ,$$
where the equivalence relation is given by $(0,x)\sim (1,\phi(x))$ for all $x\in S$  \cite[Theorem 19]{miranda1}.
Let $\ell\subset S$ be a Lagrangian submanifold such that $\phi(\ell)=\ell$. Then $$L:=( [0,1]\times \ell)/\sim, $$
the mapping torus of $\ell$,
is an $n$-dimensional Lagrangian submanifold in $Z$. 
\end{ex}	
}

Now let $(X,Z)$ be a $b$-manifold (i.e. $Z$ is a codimension-one submanifold of $X$), 
and let $N\subset X$ be a $b$-submanifold (i.e. $N$ is transverse to $Z$). Then we have ${}^bTN\subset  {}^bTX$, and the annihilator $({}^bTN)^{\circ}$ is a Lagrangian submanifold of the log-symplectic manifold ${}^bT^*X$ transverse to the singular locus. We now adapt this construction to obtain Lagrangian submanifolds of ${}^bT^*X$ \emph{contained} in the singular locus. The construction depends on additional choices,
and is somewhat reminiscent of the following in symplectic geometry: for a submanifold $N$ of an arbitrary manifold and a closed  $\alpha\in \Omega^1(N)$, the affine subbundle  $\{\xi: \xi|_{TN} =\alpha\}$ over $N$ is a Lagrangian submanifold of the cotangent bundle. 
 In Proposition \ref{lem:splitting} we provide a construction depending on Lie algebroid splittings, and in Corollary \ref{prop:Lagrexplicit} we make the construction more explicit. We first need to recall some well-known facts in the following remark.

\begin{remark}\label{rem:wellknown}
\begin{itemize}
\item[i)] Let $A\to Z$ be a Lie algebroid, and $B\to N$ a Lie subalgebroid supported on a submanifold $N$ of $Z$. Then the annihilator $B^{\circ}$ is a coisotropic submanifold of the Poisson manifold $A^*$, where the latter is endowed with its natural linear Poisson structure \cite[Prop. 4.6]{Xu}. 
\item[ii)] Let  $(X,Z)$ be a $b$-manifold. The $b$-cotangent bundle ${}^bT^*X$ has a canonical 
log-symplectic structure $\Pi$ \cite[Ex. 9]{miranda2}, which agrees with the linear Poisson structure associated to the Lie algebroid structure  $({}^bTX,[\cdot,\cdot],\rho)$. 
(This construction generalizes Example \ref{ex:R2nlog}.)

There is a short exact sequence of Lie algebroids over $Z$:
\begin{equation}\label{eq:sesL}
0\rightarrow\mathbb{L}\hookrightarrow{}^{b}TX|_Z\overset{\rho}{\rightarrow}TZ\rightarrow 0.
\end{equation} 
The kernel $\mathbb{L}$ has a canonical non-vanishing section $w$ (given by $x_1\pd{x_1}$ for any local coordinate system on $X$ for which $x_1|_Z=0$), thus we have a canonical isomorphism of line bundles
$\mathbb{L}\cong \RR\times Z$ \cite[\S3.1]{miranda2}.
The dual short exact sequence reads
\begin{equation}\label{eq:sesL*}
0\rightarrow
T^*Z
\overset{\rho^*}{\hookrightarrow}
{}^{b}T^*X|_Z
\rightarrow
\mathbb{L}^*
\rightarrow 0.
\end{equation} 
It is known that ${}^{b}T^*X|_Z$ is the singular locus of the log-symplectic structure $\Pi$, hence it has a corank-1  Poisson structure. It turns out that its symplectic leaves are exactly the preimages of the constant sections of $\mathbb{L}^*$ under the map in \eqref{eq:sesL*}, i.e. the level sets of the function $\langle w,\cdot \rangle$ on ${}^{b}T^*X|_Z$, as can be seen easily in coordinates.
\end{itemize}
\end{remark}

\begin{prop}\label{lem:splitting}
Let  $(X^n,Z^{n-1})$ be a $b$-manifold, and $N\subset Z$ a submanifold. Let  $$\sigma\colon TN\to {}^bTX|_Z$$ be a \emph{Lie algebroid morphism} such that\footnote{When $N=Z$, this means that $\sigma$ is a Lie algebroid splitting of the sequence \eqref{eq:sesL}.} $\rho\circ \sigma=Id_{TN}$.

Then $L$, defined as the annihilator of $\sigma(TN)$ in the vector bundle ${}^bT^*X|_N$, is an $n$-dimensional Lagrangian submanifold of ${}^bT^*X$ contained in the singular locus.
\end{prop}	

\begin{proof}
By assumption,  $\sigma(TN)$ is a Lie subalgebroid of ${}^bTX|_Z$ supported on $N$. Its annihilator
$L$ maps onto $\mathbb{L}^*|_N$ under the map in the short exact sequence \eqref{eq:sesL*}. Hence $L$ intersects transversely the $2n-2$-dimensional symplectic leaves of ${}^bT^*X|_Z$, by Remark \ref{rem:wellknown} ii).
Notice that $L$ is an $n$-dimensional coisotropic submanifold of ${}^bT^*X|_Z$ by Remark \ref{rem:wellknown} i). Hence the intersection of $L$ with any symplectic leaf is a   $(n-1)$-dimensional
coisotropic submanifold of the symplectic leaf, i.e. it is Lagrangian there. 
Hence $L$ is Lagrangian.
\end{proof}	

\begin{cor}\label{prop:Lagrexplicit}
Let  $(X^n,Z^{n-1})$ be a $b$-manifold, and assume that the normal bundle $TX|_Z/TZ$ is trivial. Let $N\subset Z$ be a submanifold. Let 
\begin{itemize}
\item $f\in C^{\infty}(U)$ be a defining function\footnote{Existence of such $f$ is equivalent with $TX|_Z/TZ$ being trivial. Given $f$, note that $df|_{Z}$ trivializes $(TX|_Z/TZ)^{*}$, hence also $TX|_Z/TZ$ is trivial. Conversely, take an exponential map $\exp:TX|_Z/TZ\cong Z\times\mathbb{R}\rightarrow M$ on a neighborhood of $Z$, set $pr:Z\times\mathbb{R}\rightarrow\mathbb{R}$ to be the projection and define $f:=pr\circ \exp^{-1}$.} for $Z$ (here $U$ is a tubular neighborhood of $Z$),
\item $\alpha\in \Omega^1(N)$ be a \emph{closed} 1-form on $N$.
\end{itemize}
Then 
$$TN^{\circ}\oplus \RR\left(\frac{df}{f}|_N- \alpha'\right)$$
is an $n$-dimensional Lagrangian submanifold of ${}^bT^*X$ contained in the singular locus.
 Here 
 \begin{itemize}
\item $TN^{\circ}$ is the annihilator of $TN$ in $T^*Z|_N$,
\item  $\alpha'\in \Gamma(T^*Z|_N)$ is any extension\footnote{The construction clearly does not depend on the extension.} of $\alpha$,
\end{itemize}
and we view both as lying in ${}^{b}T^*X|_Z$ by using the map $\rho^*$ in \eqref{eq:sesL*}.
\end{cor}	
\begin{proof}
The line subbundle $\RR\frac{df}{f}|_Z$ of ${}^bT^*X|_Z$  maps surjectively onto $\mathbb{L}^*$ in the sequence \eqref{eq:sesL*}, hence the kernel of $\frac{df}{f}|_Z$ is the image of a 
vector bundle splitting
$\tau\colon TZ\to {}^{b}TX|_Z$ of the anchor map $\rho$ in the sequence \eqref{eq:sesL}. Further $\tau$ is a Lie algebroid morphism (this  can be  be seen choosing coordinates on $X$ adapted to $Z$; it also  follows from the fact that 
$\frac{df}{f}|_Z$ is closed and thus its kernel is involutive).

  Now let $\sigma\colon TN\to {}^bTX|_Z$ be a vector bundle map such that $\rho\circ \sigma=Id_{TN}$. Then, in view of the short exact sequence \eqref{eq:sesL},  the difference $\sigma-\tau|_{TN}$ is a vector bundle map $TN\to \mathbb{L}|_N\cong \RR\times N$. Hence $$\sigma=\tau|_{TN}+\alpha  w$$ for some $\alpha\in \Omega^1(N)$, where $w$ denotes the canonical section of $\mathbb{L}$ as in Remark \ref{rem:wellknown} ii). Conversely, given any  $\alpha\in \Omega^1(N)$, the above formula defines a    vector bundle map $\sigma$ satisfying $\rho\circ \sigma=Id_{TN}$.
  
It turns out that $\sigma$ is a Lie algebroid morphism (i.e. $\sigma(TN)$ is a Lie subalgebroid)
if{f} $\alpha$ is closed. Indeed, for all vector fields $X,Y
\in \mathfrak{X}(N)$, after choosing extensions $\tilde{X},  \tilde{Y}\in \mathfrak{X}(Z)$ and an extension $\tilde{\alpha}\in \Omega^1(Z)$ of $\alpha$, we have:
$$[\tau(\tilde{X})+\tilde{\alpha}(\tilde{X}) w,\tau(\tilde{Y})+\tilde{\alpha}(\tilde{Y}) w]|_N=\tau([X,Y])+ \alpha([X,Y]) w$$
if{f} $(d\alpha)(X,Y)=0$, as one computes
using the fact that $\tau$ is bracket-preserving.

As in Proposition \ref{lem:splitting}, denote by
 $L$ the annihilator of $\sigma(TN)$ in ${}^bT^*X|_N$. Then
\begin{equation}\label{eq:cL}
  L=TN^{\circ}\oplus \RR\left(\frac{df}{f}|_N- \alpha'\right).
\end{equation}
The inclusion ``$\supset$'' holds for the following reasons: $TN^{\circ}$ annihilates 
$\sigma(TN)$,
and  for all vectors $X\in TN$ we have 
\[
\left\langle \frac{df}{f}|_N- \alpha'\,,\, \sigma(X) \right\rangle=\alpha(X)-\alpha(X)=0.
\]
Since both vector subbundles have the same rank, we obtain the equality \eqref{eq:cL}, and Proposition \ref{lem:splitting} concludes the proof.
\end{proof}

\begin{ex}
i) The case $N=\{p\}$, for $p$ a  point of $Z$: in this case the construction of Corollary \ref{prop:Lagrexplicit} does not depend on any additional choice, and yields as Lagrangian submanifold the fiber  ${}^bT_p^*X$.

ii) The case $N=Z$: for any defining function $f$ for $Z$ and  closed  1-form $\alpha$ on $Z$, the line bundle 
\begin{equation}\label{eq:specialLagr}
\RR \left(\frac{df}{f}|_Z-\alpha\right)\to Z
\end{equation}
is a Lagrangian submanifold. We remark that such $\frac{df}{f}|_Z-\alpha$ are exactly the closed (w.r.t. the Lie algebroid differential) sections of ${}^bT^*X|_Z$ which on $w$ evaluate to $1$.

 Notice that any other defining function for $Z$ is of the form $fg$ for some function $g$ without zeros, and that $\frac{d(fg)}{fg}|_Z$ is given by $\frac{df}{f}|_Z$ plus an exact term $d(log(|g|_Z|)$. Thus, when $H^1(Z)\neq 0$, not all Lagrangian submanifolds of the form \eqref{eq:specialLagr} can be realized as  $\RR (\frac{dh}{h}|_Z)$ for some defining function $h$.

  \end{ex}

	Given a Lagrangian $L^{n}$ contained in the singular locus $Z$ of a log-symplectic manifold $(M^{2n},\Pi)$, Prop. \ref{normalform} describes a neighborhood of $L$ in $(Z,\Pi|_{Z})$ and Prop. \ref{normZ} describes a neighborhood of $Z$ itn $(M^{2n},\Pi)$. Combining the two propositions, we get the following normal form around $L^{n}\subset (M^{2n},\Pi)$.

	\begin{cor}[Local form around a Lagrangian in the singular locus]\label{normal}
		Let $(M^{2n},Z,\Pi)$ be an orientable log-symplectic manifold, and make a choice of modular vector field  $V_{mod}$ on $M$.
		Let  $L^{n}\subset Z$ be a Lagrangian submanifold, and denote by $\mathcal{F}_{L}$ the induced\footnote{The leaves of the  codimension-one foliation $\mathcal{F}_{L}$ are the connected components of the intersections	of $L$ with the symplectic leaves of $Z$.}
		 foliation on $L$. Then a neighborhood of $L$ in $(M,\Pi)$ can be identified with a neighborhood of $L$ in the vector bundle $T^{*}\mathcal{F}_{L}\times\mathbb{R}\rightarrow L$, endowed with the log-symplectic structure
		\begin{equation}\label{freedom}
		\widetilde{\Pi}:=V\wedge t\partial_{t}+\Pi_{can}.
		\end{equation}
		Here $t$ is a coordinate on $\mathbb{R}$, and $V$ is the image of $V_{mod}|_{Z}$ under the Poisson diffeomorphism $(Z,\Pi|_{Z})\rightarrow(T^{*}\mathcal{F}_{L},\Pi_{can})$ between neighborhoods of $L$ constructed in Prop. \ref{normalform}.
	\end{cor}
	
	\begin{remark}\label{rem:freedom}
		The vector field $V$ in \eqref{freedom} is only defined on a neighborhood {$W$} of $L$ in $T^{*}\mathcal{F}_{L}$. Note that there is some freedom in the formula \eqref{freedom}, in the sense that there we can replace $V$ by any Poisson vector field representing the Poisson cohomology class $[V]$. 
		
		To see this, take any representative $V-X_f$ of $[V]$, for some function $f$ defined on {$W$}. Note that $V$ is a modular vector field of $\widetilde{\Pi}$, with respect to the volume form $\Omega$ on {$W\times\RR$} that is uniquely determined by requiring that $\langle\Omega,\wedge^{n}\widetilde{\Pi} \rangle=t$.  If $\tilde{f}$ is an extension of $f$ to {$W\times\RR$}, then also $V-X_{\tilde{f}}$ is a modular vector field of $\widetilde{\Pi}$, with respect to the volume form $e^{\tilde{f}}\Omega$ on {$W\times\RR$}. Proposition \ref{normZ} now implies that replacing $V$ by $V-X_{f}$ in \eqref{freedom} gives a log-symplectic structure that is Poisson diffeomorphic to $\widetilde{\Pi}$ in a neighborhood of {$W\subset W\times\RR$}.
	\end{remark}

	An arbitrary representative of the modular class has little to do with the Lagrangian $L$; we will remedy this in the next section.
	One could hope to find a representative of $[V]$ that is tangent to $L$. This amounts to finding a modular vector field, defined on a neighborhood of $L$ in $(M,\Pi)$, that is tangent to $L$. This can always be done locally near a point, as a consequence of Prop. \ref{coordinates} (namely, the vector field $\partial_{x_{1}}$ in the statement of the proposition is modular and tangent to $L$). Globally however, this may fail, as we now show.
	
	\begin{ex}[No modular vector field is tangent]
		Consider the manifold $\mathbb{R}\times S^{1}\times\mathbb{T}^{2}$ with coordinates $(t,\tau,\theta_{1},\theta_{2})$ and log-symplectic structure
		\[
		\Pi=\left(\partial_{\tau}+\partial_{\theta_{2}}\right)\wedge t\partial_{t}+\partial_{\theta_{1}}\wedge\partial_{\theta_{2}}.
		\]
		The submanifold $L:=\{t=\theta_{2}=0\}\cong S^{1}\times S^{1}$ is Lagrangian inside the singular locus $Z=S^{1}\times\mathbb{T}^{2}$. Note that $\partial_{\tau}+\partial_{\theta_{2}}$ is a modular vector field for $\Pi$ (associated with the volume form $d\theta_{1}\wedge d\theta_{2}\wedge dt\wedge d\tau$). If there  existed a modular vector field tangent to $L$ (defined near $L$), then its restriction to $Z$ would look like
		\[
		\partial_{\tau}+\partial_{\theta_{2}}+\left(\partial_{\theta_{1}}\wedge\partial_{\theta_{2}}\right)^{\sharp}(df)=\partial_{\tau}+\left(1+\frac{\partial f}{\partial\theta_{1}}\right)\partial_{\theta_{2}}-\frac{\partial f}{\partial\theta_{2}}\partial_{\theta_{1}}
		\]
		for some $f\in C^{\infty}(Z)$,
		where
		\begin{equation}\label{res}
		\left.\left(1+\frac{\partial f}{\partial\theta_{1}}\right)\right|_{\theta_{2}=0}=0.
		\end{equation}
		But then, fixing any value of $\tau$ and denoting by $i:\{\tau\}\times S^{1}\times\{0\}\hookrightarrow Z$ the inclusion, we get
		\[
		0=\int_{\{\tau\}\times S^{1}\times\{0\}}i^{*}df=\int_{\{\tau\}\times S^{1}\times\{0\}}\frac{\partial f}{\partial\theta_{1}}d\theta_{1}=-\int_{\{\tau\}\times S^{1}\times\{0\}} d\theta_{1}=-2\pi,
		\]
		using Stokes' theorem, and \eqref{res} 
		in the third equality. So there is no modular vector field tangent to $L$.
	\end{ex}
	
	\section{Poisson vector fields on the cotangent bundle of a foliation}\label{sec:Poisvf}
	Let $L$ be a manifold and $\mathcal{F}_{L}$ a foliation on $L$. Denote by $\Pi_{can}$ the canonical Poisson structure on $T^{*}\mathcal{F}_{L}$ (as in $b)$ of Lemma \ref{diff}). This section treats Poisson vector fields on $\left(T^{*}\mathcal{F}_{L},\Pi_{can}\right)$. We show that every class in the first Poisson cohomology group of $\left(T^{*}\mathcal{F}_{L},\Pi_{can}\right)$
admits a convenient representative (Thm. \ref{thm:modvfL}),
and use this to compute explicitly the 
 first Poisson cohomology group (Cor. \ref{cor:isocoho}).
At the beginning of \S\ref{sec:alg}, we apply these results to
the modular vector field of a log-symplectic manifold, and  
find a convenient representative of the class $[V]$ in \eqref{freedom}.
	
	\subsection{Convenient representatives}
	\leavevmode
	\vspace{0.1cm}
	
	We denote by
	$$\mathfrak{X}(L)^{\mathcal{F}_{L}}:=\big\{W\in \mathfrak{X}(L): [W,\Gamma(T\mathcal{F}_{L})]\subset \Gamma(T\mathcal{F}_{L})\big\}$$
	the Lie subalgebra of vector fields on $L$ whose flow preserves the foliation $\mathcal{F}_{L}$.
	
	\begin{lemma}\label{lem:lift}
		Let $W\in \mathfrak{X}(L)^{\mathcal{F}_{L}}$ and let $r:\left(T^{*}L,\Pi_{T^{*}L}\right)\rightarrow(T^{*}\mathcal{F}_{L},\Pi_{can})$ denote the restriction. We then have the following:
		\begin{enumerate}[(i)]
			\item The cotangent lift of $W$ pushes forward via $r\colon T^{*}L\rightarrow T^{*}\mathcal{F}_{L}$ to a Poisson vector field on $T^{*}\mathcal{F}_{L}$, which we denote by $\widetilde{W}$.
			
			\item When $W$ lies in $\Gamma(T\mathcal{F}_{L})$, the vector field $\widetilde{W}$ is Hamiltonian.
		\end{enumerate}
	\end{lemma}
	\begin{proof}
		We denote by  $p_{T^{*}\mathcal{F}_{L}}\colon  T^*\mathcal{F}_{L}\to L$ and $p_{T^{*}L}:T^{*}L\to L$ the vector bundle projections.
		\begin{enumerate}[(i)]
			\item Let $W_{T^*L}\in \mathfrak{X}(T^*L)$ denote the cotangent lift of $W$. To show that it pushes forward via $r$, we need to show that its action on functions preserves $r^*(C^{\infty}(T^*\mathcal{F}_{L}))$. It suffices to consider fiberwise constant and fiberwise linear functions on $T^*\mathcal{F}_{L}$. The fiberwise constant ones are of the form $p_{T^{*}\mathcal{F}_{L}}^*g$ for $g\in C^{\infty}(L)$.
			Since $p_{T^*L}=p_{T^{*}\mathcal{F}_{L}} \circ r$, we have 
			\[
			W_{T^*L}(r^*(p_{T^{*}\mathcal{F}_{L}}^*g))=W_{T^*L}\left(p_{T^*L}^{*}g\right)=p_{T^*L}^*(W(g))=r^*(p_{T^{*}\mathcal{F}_{L}}^*(W(g))).
			\]
			Next, fiberwise linear functions on $T^{*}\mathcal{F}_{L}$ look like $h_{X}:T^{*}\mathcal{F}_{L}\rightarrow\mathbb{R}:(p,\alpha)\mapsto \langle \alpha,X(p)\rangle$ for $X\in\Gamma(T\mathcal{F}_{L})$. Clearly, one has a commutative diagram
			\[
			\begin{tikzcd}[column sep=large, row sep=large]
			C^{\infty}_{lin}(T^{*}\mathcal{F}_{L}) \arrow[hook]{r}{r^{*}} & C^{\infty}_{lin}(T^{*}L)\\
			\Gamma(T\mathcal{F}_{L})\arrow[hook]{r}{i}\arrow{u}{h_{\bullet}}&\Gamma(TL)\arrow{u}{h_{\bullet}}
			\end{tikzcd}.
			\]
			Recall that for the standard symplectic structure on $T^*L$, the Poisson bracket satisfies $\{h_{X},h_{Y}\}=-h_{[X,Y]}$, for $X,Y\in\Gamma(TL)$. Moreover, the cotangent lift $W_{T^{*}L}$ is minus the Hamiltonian vector field of $h_{W}$ (see e.g. \cite[\S2]{crampin}). So for $X\in\Gamma(T\mathcal{F}_{L})$ we get
			\[
			W_{T^{*}L}(r^{*}h_{X})=W_{T^{*}L}\left(h_{i(X)}\right)=-X_{h_{W}}\left(h_{i(X)}\right)=-\{h_{W},h_{i(X)}\}=h_{[W,i(X)]}.
			\]
			The vector field $[W,i(X)]$ lies in $\Gamma(T\mathcal{F}_{L})$ by assumption, so that $W_{T^{*}L}(r^{*}h_{X})$ lies in $r^*(C_{lin}^{\infty}(T^*\mathcal{F}_{L}))$. This shows that $W_{T^*L}$ pushes forward under $r$.
			
			\noindent
			The vector field $\widetilde{W}$ is Poisson since the cotangent lift $W_{T^*L}$ is a symplectic vector field and $r$ is a Poisson map.
			
			\item If $W$ lies in $\Gamma(T\mathcal{F}_{L})$, then we have  
			\[
			\widetilde{W}=r_{*}\big(i(W)\big)_{T^{*}L}=r_{*}\left(-X_{h_{i(W)}}\right)=-r_{*}\left(X_{r^{*}h_{W}}\right)=-X_{h_{W}},
			\]
			where in the second equality we used the above comment about Hamiltonian vector fields, and in the third the commutativity of the  diagram.\qedhere
		\end{enumerate}
	\end{proof}
	
	The rest of this section is devoted to the following theorem, which provides convenient representatives for   first Poisson cohomology classes, and its consequences.
	
	\begin{thm}\label{thm:modvfL}
		Let $(L,\mathcal{F}_{L})$ be a foliated manifold. Consider the standard Poisson structure $\Pi_{can}$ on the total space of the vector bundle $p\colon  T^*\mathcal{F}_{L}\to L$. Fix a class in $H^1_{\Pi_{can}}(T^*\mathcal{F}_{L})$. Then there exists a representative $Y\in \mathfrak{X}(T^*\mathcal{F}_{L})$ such that
		\begin{enumerate}[(i)]
			\item $Y$ is $p$-projectable and $p_*Y\in \mathfrak{X}(L)^{\mathcal{F}_{L}}$,
			\item the vector field\footnote{The lift $\widetilde{p_*Y}$ was defined in Lemma \ref{lem:lift}.}  $Y-\widetilde{p_*Y}$ is vertical and constant on each fiber of $p$, and  $Y-\widetilde{p_*Y}$ is closed when viewed as\footnote{I.e., when viewed as a section of $p:T^*\mathcal{F}_{L}\to L$.} a foliated 1-form on $(L,\mathcal{F}_{L})$.
		\end{enumerate}
	\end{thm}
	Notice that given a class in $H^1_{\Pi_{can}}(T^*\mathcal{F}_{L})$, a representative $Y$ as in Theorem \ref{thm:modvfL} is by no means unique: adding to $Y$ a Hamiltonian vector field of the form $\widetilde{W}+X_{p^{*}g}$ for $W\in\Gamma(T\mathcal{F}_{L})$ and $g\in C^{\infty}(L)$ gives a representative of the same class that still satisfies the requirements of Theorem \ref{thm:modvfL} (see Corollary \ref{cor:isocoho} below).
	\begin{ex}
		Consider the plane $L=\RR^2$  with coordinates $x,y$, and the foliation $\mathcal{F}_{L}$ given by the lines $\{x=const\}$. Then $T^*\mathcal{F}_{L}$ is $\RR^3$ with coordinates $x,y,z$, with vector bundle projection $p=(x,y)\colon \RR^3\to \RR^2$ and Poisson structure $\Pi_{can}= \pd{y}\wedge\pd{z}$.
		An arbitrary Poisson vector field has the form $$U=f(x)\pd{x}+g\pd{y}+k\pd{z},$$ where $g,k\in C^{\infty}(\RR^3)$ satisfy $\pd{y}g=-\pd{z}k$.  
		This vector field is not $p$-projectable in general, because $g$ might depend on $z$. However, defining $h(x,y,z):=\int_0^z g(x,y,t)dt$, we obtain a function on $\RR^3$ such that 
		$$Y:=U+X_h=f(x)\pd{x}+(k+\pd{y}h)\pd{z}$$ is $p$-projectable. Notice that 
		$p_*Y=f(x)\pd{x}$ lies in $\mathfrak{X}(L)^{\mathcal{F}_{L}}$.
		Moreover, since the partial
		derivative
		$\partial_{z}(k+\pd{y}h)$ vanishes, the vertical vector field $V:=(k+\pd{y}h)\pd{z}$ is indeed   
		constant on each fiber of $p$. Regarding $V$ as a foliated 1-form on $(L,\mathcal{F}_{L})$ yields $(k+\pd{y}h)dy$, which is closed due to dimension reasons. 
	\end{ex}
	
	To prove Theorem \ref{thm:modvfL}, we need a few general statements about cotangent bundles.
	
	\begin{lemma}\label{lem:cotg}
		Let $N$ be a manifold. Consider its cotangent bundle $T^*N$ with the standard symplectic form $\omega$ and bundle projection $p_{T^*N}$.
		\begin{enumerate}[(i)]
			\item Let $Y\in \mathfrak{X}(T^*N)$ be a
			symplectic vector field\footnote{Part $(i)$ of the lemma holds more generally whenever the pullback of $\iota_{Y}\omega$ to each fiber of $p_{T^{*}N}$ is closed.}. Then there is $h\in C^{\infty}(T^*N)$ such that $Y+X_h$ is a vertical vector field.  
			
			\item Let $V\in \mathfrak{X}(T^*N)$ be a vertical symplectic vector field. Then $V$ must be constant on each fiber. It is closed when viewed as an element of $\Gamma(T^*N)=\Omega^1(N)$.
		\end{enumerate}
		\begin{proof}
			\begin{enumerate}[(i)]
				\item Consider the foliation $\cF_{fiber}$ of $T^*N$ by fibers of the projection $p_{T^*N}$. Denote by $i_{fiber}$ the inclusion of its tangent distribution into the tangent bundle of $T^*N$. Since the 1-form $\iota_Y\omega\in \Omega^1(T^*N)$ is closed, its pullback $i_{fiber}^{*}(\iota_Y\omega)$ is closed as a foliated $1$-form. It is foliated exact, as the leaves of $\cF_{fiber}$ are just fibers of a vector bundle  (choosing the primitives on each fiber to vanish on the zero section, they assemble to a smooth function on $T^{*}N$, c.f. Lemma \ref{poin}).
				So $i_{fiber}^*(\iota_Y\omega)$ equals $d_{\cF_{fiber}}h$ for some  $h\in C^{\infty}(T^*N)$, which implies that 
				$\iota_Y\omega-dh\in \Omega^1(T^*N)$ pulls back to zero under $i_{fiber}$. As the fibers are Lagrangian, this means that $\omega^{-1}(\iota_Y\omega-dh)=Y+X_h$ 			is a vertical vector field on $T^*N$.
				
				\item The 1-form $\iota_V\omega$ is closed because $V$ is a symplectic vector field. For any vertical vector field $W$ we have $\iota_W\iota_V\omega=0$ and $\cL_W(\iota_V\omega)=0$, so 
				$\iota_V\omega=-p_{T^*N}^*\alpha$ for a unique, closed $\alpha\in \Omega^1(N)$.
				Writing in local coordinates $\alpha=\sum_i f_i(q)dq_i$, in the corresponding canonical coordinates on $T^*N$ we have $$V=-(\omega^{-1})(p_{T^*N}^*\alpha)=\sum_i f_i(q)\pd{p_i},$$ 
				showing that $V$ is constant along the fibers.  This formula also shows that $V$, regarded as an element of $\Gamma(T^*N)=\Omega^1(N)$, is precisely the closed 1-form $\alpha$.
			\end{enumerate}
		\end{proof}
	\end{lemma}
	
	\begin{proof}[Proof of Thm. \ref{thm:modvfL}]
		Let $U\in \mathfrak{X}(T^*\mathcal{F}_{L})$ be any representative of the given class in $H^1_{\Pi_{can}}(T^*\mathcal{F}_{L})$. Let $U_0\in \mathfrak{X}(L)$ be given by $(U_0)(x):=(d_xp)(U(x))$ at each point $x\in L$. So $(U_{0})(x)$ is just the $T_xL$-component of $U(x)$ w.r.t. the canonical splitting $T_x(T^*\mathcal{F}_{L})=T_xL\oplus T_x^*\mathcal{F}_{L}$.
		
		We first show that $U_0\in \mathfrak{X}(L)^{\mathcal{F}_{L}}$. Since this is a local statement, it suffices to consider
		open subsets of $L$ whose quotient by the restriction of $\mathcal{F}_{L}$ is a smooth manifold and show that the restriction of $U_0$ projects to a vector field on the leaf space.
		By abuse of notation, we denote such an open subset by $L$. Since the  leaves of the symplectic foliation $\cF_{sympl}$ of  $T^*\mathcal{F}_{L}$ are the preimages under $p$ of the leaves of $\mathcal{F}_{L}$, there is a canonical
		diffeomorphism of leaf spaces
		$$T^*\mathcal{F}_{L}/ \cF_{sympl}\cong L/\mathcal{F}_{L},$$
		induced  by the vector bundle projection $p\colon  T^*\mathcal{F}_{L}\to L$ (or equivalently, by the inclusion of the zero section).
		Since $U$ is a Poisson vector field, it projects under $T^*\mathcal{F}_{L} \to T^*\mathcal{F}_{L}/ \cF_{sympl}$ to some vector field $U_{quot}$. Restricting to points of the zero section $L$, we see that $U_0$
		is projectable  under $L\to L/\mathcal{F}_{L}$ (to the same vector field $U_{quot}$). 
		
		By Lemma \ref{lem:lift}, $U_0$ lifts to a Poisson vector field  $\widetilde{U_0}$ on $T^*\mathcal{F}_{L}$.
		The Poisson vector field $U-\widetilde{U_0}$ is tangent to the symplectic foliation on $T^*\mathcal{F}_{L}$. Indeed, since the statement is a local one, we can again work on suitable open subsets of $L$ and use that both $U$ and $\widetilde{U_0}$ are projectable to the same vector field under $T^*\mathcal{F}_{L} \to T^*\mathcal{F}_{L}/ \cF_{sympl}$.
		
		We now apply Lemma \ref{lem:cotg} $i)$ smoothly to the leaves of $\mathcal{F}_{sympl}$ and the vector field $U-\widetilde{U_0}$. More precisely, by the proof of  Lemma \ref{lem:cotg} $i)$, if $\omega$ denotes the leafwise symplectic form on $T^{*}\mathcal{F}_{L}$, then we find a function $h\in C^{\infty}(T^{*}\mathcal{F}_{L})$ such that the pullback of
		$
		\iota_{U-\widetilde{U_0}}\omega-d_{\cF_{sympl}}h
		$
		to the fibers of $p$ is zero. It follows that
		\[
		-\Pi_{can}^{\sharp}\left(\iota_{U-\widetilde{U_0}}\omega-d_{\cF_{sympl}}h\right)=U-\widetilde{U_0}+X_{h}
		\]
		is vertical, i.e. tangent to the $p$-fibers.
		This has two consequences. First, we can apply Lemma \ref{lem:cotg} ii) to conclude that this vector field is constant on each fiber, and is  closed when viewed as a foliated 1-form on $(L,\mathcal{F}_{L})$. 
		Second, $U+X_h$ is $p$-projectable
		and it projects to the same vector field as $\widetilde{U_0}$, namely $U_0\in \mathfrak{X}(L)$. Hence $Y:=U+X_h$ is a representative of the class  $H^1_{\Pi_{can}}(T^*\mathcal{F}_{L})$ with the required properties.
	\end{proof}
	
	\subsection{The first Poisson cohomology}
	\leavevmode
	\vspace{0.1cm}
	
	Using Theorem \ref{thm:modvfL}, we can compute the first Poisson cohomology group of $\left(T^{*}\mathcal{F}_{L},\Pi_{can}\right)$. In the following, $H^{\bullet}(\mathcal{F}_{L})$ denotes the cohomology of the foliated differential forms along the leaves of $\mathcal{F}_{L}$.
	
	\begin{cor}\label{cor:isocoho}
		Let $(L,\mathcal{F}_{L})$ be a foliated manifold and denote by $\Pi_{can}$ the standard Poisson structure on the total space of the vector bundle $p\colon  T^*\mathcal{F}_{L}\to L$.
		There is a linear isomorphism
		\begin{align}\label{iso}
		\Phi\colon H^1_{\Pi_{can}}(T^*\mathcal{F}_{L})&\to \mathfrak{X}(L)^{\mathcal{F}_{L}}/\Gamma(T\mathcal{F}_{L}) \times H^1(\mathcal{F}_{L})\nonumber\\
		[Y]&\mapsto \;\;\;\;\big([p_*Y],[Y-\widetilde{p_*Y}]\big),
		\end{align}
		where the representative $Y$ satisfies the properties in Thm. \ref{thm:modvfL}. 
	\end{cor}
	
	Notice that  $\mathfrak{X}(L)^{\mathcal{F}_{L}}/\Gamma(T\mathcal{F}_{L})$ agrees with the space of vector fields on $L/\mathcal{F}_{L}$, whenever the latter quotient  is smooth. {An alternative description is the following: $\mathfrak{X}(L)^{\mathcal{F}_{L}}/\Gamma(T\mathcal{F}_{L})$ is naturally isomorphic with $H^{0}(\mathcal{F}_{L},\nu)$, the zero-th foliated cohomology of $\mathcal{F}_{L}$ with coefficients in the normal bundle $\nu=TL/T\mathcal{F}_{L}$ equipped with the Bott connection.}
	
	\begin{proof}
		We first show that the map $\Phi$ is well-defined. For this, due to Thm. \ref{thm:modvfL}, we only need to show that the above assignment is independent of the choice of representative. Equivalently, since the expression in \eqref{iso} depends linearly on $Y$, we have to show that if $Y$ is a Hamiltonian vector field on $T^*\mathcal{F}_{L}$ satisfying the properties in Thm. \ref{thm:modvfL}, then 
		$p_*Y$ lies in $\Gamma(T\mathcal{F}_{L})$ and $Y-\widetilde{p_*Y}$ is exact when viewed as a foliated 1-form on $(L,\mathcal{F}_{L})$.
		
		Being Hamiltonian, $Y$ is tangent to the symplectic foliation of $T^*\mathcal{F}_{L}$, so $p_*Y$ is tangent to the foliation $\mathcal{F}_{L}$. Hence $\widetilde{p_*Y}$ is a Hamiltonian vector field, by Lemma \ref{lem:lift}. 
		Being the difference of two Hamiltonian vector fields, the vertical and fiberwise constant vector field $V:=Y-\widetilde{p_*Y}$ is Hamiltonian. Denote by $F\in C^{\infty}( T^*\mathcal{F}_{L})$ a Hamiltonian function for $V$, so that for  each leaf $\cO$ of $\mathcal{F}_{L}$ we have
		$\iota_V\omega_{T^*\cO}=-d(F|_{T^*\cO})$.
		Regarding the vertical constant vector field $V$ as a foliated 1-form yields $\alpha\in \Omega^1(\mathcal{F}_{L})$, determined by 
		\begin{equation}\label{eq:Valpha}
		\iota_V\omega_{T^*\cO}=-p_{T^*\cO}^*(\alpha|_{\cO}),
		\end{equation}
		see the proof of Lemma \ref{lem:cotg} $(ii)$.
		In particular, $F$ is constant along the fibers of $p\colon  T^*\mathcal{F}_{L}\to L$, i.e. $F=p^*(F|_L)$.
		Thus $\alpha=d_{\mathcal{F}_{L}}F$, showing that it is foliated exact.
		
		We show that $\Phi$ is surjective. Let $W\in \mathfrak{X}(L)^{\mathcal{F}_{L}}$.
		Then its lift $\widetilde{W}$ is a Poisson vector field on $T^*\mathcal{F}_{L}$, by Lemma \ref{lem:lift} $i)$.
		Let $\alpha\in \Omega^1(\mathcal{F}_{L})$ be a closed foliated 1-form. Denote by $V$  the corresponding vertical fiberwise constant vector field on  $T^*\mathcal{F}_{L}$. 
		Then $V$ is a Poisson vector field, because it is tangent to the symplectic leaves of $T^*\mathcal{F}_{L}$ and its restriction to each symplectic leaf is a symplectic vector field, by eq. \eqref{eq:Valpha}.
		Hence $\widetilde{W}+V$ is a Poisson vector field on $T^*\mathcal{F}_{L}$. By construction it satisfies the properties of Thm. \ref{thm:modvfL}, and its Poisson cohomology class maps under $\Phi$ to $([W],[\alpha])$.
		
		We show that $\Phi$ is injective. Let $Y$ be a Poisson vector field on $T^*\mathcal{F}_{L}$ satisfying the properties in Thm. \ref{thm:modvfL},   so that $p_*Y$ lies in $\Gamma(T\mathcal{F}_{L})$ and $V:=Y-\widetilde{p_*Y}$ is exact when viewed as a foliated 1-form on $(L,\mathcal{F}_{L})$.
		By Lemma \ref{lem:lift} $ii)$, $\widetilde{p_*Y}$ is a Hamiltonian vector field.
		Let $\alpha=d_{\mathcal{F}_{L}}f\in \Omega^1(\mathcal{F}_{L})$ be the exact foliated 1-form corresponding to $V$, where $f\in C^{\infty}(L)$. Then eq. \eqref{eq:Valpha} implies that $V=\Pi_{can}^{\sharp}(p^*(df))$, showing that $V$ is  a Hamiltonian vector field. Hence $Y=\widetilde{p_*Y}+V$ is Hamiltonian, so $[Y]=0$.
	\end{proof}

	We discuss the isomorphism \eqref{iso} in two particular cases.
	
	\begin{ex}
		\begin{enumerate}[i)]
			\item 	Suppose $\mathcal{F}_{L}$ is the foliation of $L$ by points. Then $T^*\mathcal{F}_{L}$ is just $L$ with the zero Poisson structure, and the map $\Phi$ is just the identity on $\mathfrak{X}(L)$.  
			\item 	On the other extreme, suppose $\mathcal{F}_{L}$ is the one-leaf foliation of $L$. Then $T^*\mathcal{F}_{L}$ is the cotangent bundle $T^*L$ with its standard symplectic form, and  	
			$$\Phi\colon  H^1_{\Pi_{can}}(T^*L)\to  H^1(L).$$ Since $\Phi$ is an   isomorphism, every class in $H^1_{\Pi_{can}}(T^*L)$ admits a representative $V$ which is a vertical fiberwise constant vector field (c.f. Lemma \ref{lem:cotg}). The image of this class under $\Phi$ is $[\alpha]\in H^1(L)$, where $\alpha$ is just $V$ regarded as a 1-form. The inverse map $\Phi^{-1}$ reads $[\alpha]\mapsto -[\omega^{-1} (p^*\alpha)]$, by eq. \eqref{eq:Valpha}, i.e. it is the composition of the natural isomorphism $p^*\colon H^1(L)\to H^1(T^*L)$ and the isomorphism $H^1(T^*L)\cong  H^1_{\Pi_{can}}(T^*L)$ from de Rham to Poisson cohomology carried on every symplectic manifold.
		\end{enumerate}
	\end{ex}

	\begin{remark}\label{rem:codim1PoisCoho}
		In case the foliation $\mathcal{F}_{L}$ on $L$ is of codimension-one,  we can compare our Corollary \ref{cor:isocoho} with some results that appeared in \cite{Osorno}.
		\begin{enumerate}[i)]
			\item In \cite[Prop. 1.4.7]{Osorno}, one proves the following: if $(M,\Pi)$ is a corank-one Poisson manifold and $(\mathcal{F},\omega)$ denotes its symplectic foliation, then there is a long exact sequence
			\begin{equation}\label{sequence}
			\cdots\rightarrow H^{k-2}(\mathcal{F},\nu)\overset{\mathfrak{d}}{\rightarrow} H^{k}(\mathcal{F})\overset{\Pi}{\longrightarrow} H^{k}_{\Pi}(M)\rightarrow H^{k-1}(\mathcal{F},\nu)\overset{\mathfrak{d}}{\rightarrow} H^{k+1}(\mathcal{F})\rightarrow\cdots
			\end{equation}
			Here $\nu:=TM/T\mathcal{F}$ denotes the normal bundle of the foliation, and $H^{\bullet}(\mathcal{F},\nu)$ is the cohomology of the complex $\big(\Gamma(\wedge^{\bullet}T^{*}\mathcal{F}\otimes\nu),d_{\nabla}\big),$ where the differential $d_{\nabla}$ is induced by the Bott connection
			\[
			\nabla:\Gamma(T\mathcal{F})\times\Gamma(\nu)\rightarrow\Gamma(\nu):\nabla_{X}\overline{N}=\overline{[X,N]}.
			\]
			The connecting map $\mathfrak{d}$ is, up to sign, given by the cup product with the leafwise variation $var_{\omega}\in H^{2}(\mathcal{F},\nu^{*})$ of $\omega$ \cite[Def. 1.2.14]{Osorno}, which vanishes when $\omega$ extends to a globally defined closed $2$-form on $M$.
			
			Specializing to our situation, assume $(L,\mathcal{F}_{L})$ is a codimension-one foliation. Then $(T^{*}\mathcal{F}_{L},\Pi_{can})$ is a corank-one Poisson manifold with symplectic foliation $(\mathcal{F}_{sympl},\omega)$. The leafwise symplectic form $\omega\in\Gamma(\wedge^{2}T^{*}\mathcal{F}_{sympl})$ extends to a closed $2$-form on $T^{*}\mathcal{F}_{L}$. Indeed, a closed extension of $\omega$ is given by $q^{*}\omega_{T^{*}L}$, where $q:T^{*}\mathcal{F}_{L}\rightarrow T^{*}L$ is any splitting of the restriction map $r:T^{*}L\rightarrow T^{*}\mathcal{F}_{L}$ and $\omega_{T^{*}L}$ is the canonical symplectic form on $T^{*}L$. So the connecting map $\mathfrak{d}$ in \eqref{sequence} is zero, which implies in particular that
			\begin{equation}\label{picoh}
			H_{\Pi_{can}}^{1}\left(T^{*}\mathcal{F}_{L}\right)\cong H^{1}(\mathcal{F}_{sympl})\oplus H^{0}(\mathcal{F}_{sympl},\nu).
			\end{equation}
			This is equivalent with our isomorphism in Corollary \ref{cor:isocoho}. Firstly, $H^{1}(\mathcal{F}_{sympl})\cong H^{1}(\mathcal{F}_{L})$ by homotopy invariance. Secondly, 
			as $H^{0}(\mathcal{F}_{sympl},\nu)=\mathfrak{X}(T^{*}\mathcal{F}_{L})^{\mathcal{F}_{sympl}}/\Gamma(T\mathcal{F}_{sympl})$,
			we have an isomorphism
			\[
			\mathfrak{X}(L)^{\mathcal{F}_{L}}/\Gamma(T\mathcal{F}_{L})\rightarrow H^{0}(\mathcal{F}_{sympl},\nu):[X]\mapsto \overline{\widetilde{X}},
			\]
			where $\widetilde{X}$ is the lift of $X$ as defined in Lemma \ref{lem:lift}. To see that this map is well-defined, just note that $\widetilde{X}\in\mathfrak{X}(T^{*}\mathcal{F}_{L})^{\mathcal{F}_{sympl}}$, being a Poisson vector field. Injectivity is clear, for if $\widetilde{X}$ is tangent to $\mathcal{F}_{sympl}$, then its projection $p_{*}\widetilde{X}=X$ is tangent to $\mathcal{F}_{L}$. As for surjectivity, if $\overline{U}\in H^{0}(\mathcal{F}_{sympl},\nu)$ then $\overline{U}=\overline{\widetilde{U_0}}$ as in the proof of Theorem \ref{thm:modvfL}, where $U_{0}\in\mathfrak{X}(L)^{\mathcal{F}_{L}}$. So the isomorphism \eqref{picoh} is equivalent with the one from Corollary \ref{cor:isocoho}:
			\begin{equation}\label{is}
			H_{\Pi_{can}}^{1}\left(T^{*}\mathcal{F}_{L}\right)\cong H^{1}(\mathcal{F}_{L})\oplus \mathfrak{X}(L)^{\mathcal{F}_{L}}/\Gamma(T\mathcal{F}_{L}).
			\end{equation}
			
			\item In case $(L,\mathcal{F}_{L})$ is a unimodular codimension-one foliation, then we can further simplify the isomorphism \eqref{is}.	Indeed, if $\theta\in\Omega^{1}(L)$ is a closed defining one-form for $\mathcal{F}_{L}$, then we get an isomorphism
			\[
			\mathfrak{X}(L)^{\mathcal{F}_{L}}/\Gamma(T\mathcal{F}_{L})\rightarrow H^{0}(\mathcal{F}_{L}):[V]\mapsto\theta(V).
			\]
			An alternative argument, building on i) above, is the following. Since also $\mathcal{F}_{sympl}$ is unimodular, the representation of $T\mathcal{F}_{sympl}$ on $\nu$ given by the Bott connection is isomorphic with the trivial representation of $T\mathcal{F}_{sympl}$ on the trivial $\RR$-bundle  $T^{*}\mathcal{F}_{L}\times\mathbb{R}$ (see \cite[Lemma 1.5.15]{Osorno}). So in \eqref{picoh}, we get $H^{0}(\mathcal{F}_{sympl},\nu)\cong H^{0}(\mathcal{F}_{sympl})\cong H^{0}(\mathcal{F}_{L})$.
		\end{enumerate}
	\end{remark}
	
	We will now upgrade Corollary \ref{cor:isocoho} to an isomorphism of Lie algebras. Note that the Lie bracket on $\mathfrak{X}(L)$ restricts to $\mathfrak{X}(L)^{\mathcal{F}_{L}}$ thanks to the Jacobi identity. Since $\Gamma(T\mathcal{F}_{L})$ is a Lie algebra ideal of $\left(\mathfrak{X}(L)^{\mathcal{F}_{L}},[\cdot,\cdot]\right)$, the Lie bracket passes to the quotient $\mathfrak{X}(L)^{\mathcal{F}_{L}}/\Gamma(T\mathcal{F}_{L})$. 
	We get a representation of this Lie algebra on the vector space $H^{1}(\mathcal{F}_{L})$, namely
	\begin{equation}\label{action}
	\rho:\frac{\mathfrak{X}(L)^{\mathcal{F}_{L}}}{\Gamma(T\mathcal{F}_{L})}\rightarrow \text{End}\left(H^{1}(\mathcal{F}_{L})\right):[X]\mapsto \pounds_{X}\boldsymbol{\cdot}.
	\end{equation}
	Here the Lie derivative
	\begin{equation}\label{lieder}
	\pounds_{X}\alpha:=\left.\frac{d}{dt}\right|_{t=0}\phi_{t}^{*}\alpha
	\end{equation}
	of $\alpha\in\Omega^{1}(\mathcal{F}_{L})$ along $X$ makes sense since the flow $\phi_{t}$ of $X$ preserves the foliation $\mathcal{F}_{L}$. Clearly, the map \eqref{action} is well-defined: for any $X\in\mathfrak{X}(L)^{\mathcal{F}_{L}}$, the Lie derivative $\pounds_{X}$ acts on $H^{1}(\mathcal{F}_{L})$ since it commutes with the foliated differential $d_{\mathcal{F}_{L}}$. Moreover, if $X\in\Gamma(T\mathcal{F}_{L})$, then $\pounds_{X}$ acts trivially in cohomology thanks to Cartan's magic formula. The fact that $\rho$ is a Lie algebra morphism is simply the identity $\pounds_{[X,Y]}=\pounds_{X}\circ\pounds_{Y}-\pounds_{Y}\circ\pounds_{X}$ for $X,Y\in\mathfrak{X}(L)^{\mathcal{F}_{L}}$.
	
	\begin{prop}\label{isolie}
		Let $L$ be a manifold and $\mathcal{F}_{L}$ a foliation on $L$. Let $\Pi_{can}$ denote the standard Poisson structure on the total space of the vector bundle $p:T^{*}\mathcal{F}_{L}\rightarrow L$. The map $\Phi$ constructed in Corollary \ref{cor:isocoho} becomes an isomorphism of Lie algebras
		\[
		\Phi:\left(H_{\Pi_{can}}^{1}(T^{*}\mathcal{F}_{L}),[\cdot,\cdot]\right)\rightarrow \left(\mathfrak{X}(L)^{\mathcal{F}_{L}}/\Gamma(T\mathcal{F}_{L})\ltimes_{\rho}H^{1}(\mathcal{F}_{L}),[\cdot,\cdot]_{\rho}\right),
		\]	
		where $[\cdot,\cdot]$ is the usual the Lie bracket of vector fields and $[\cdot,\cdot]_{\rho}$ is the semidirect product bracket induced by the Lie algebra representation $\rho$ defined in \eqref{action}.
	\end{prop}

	To prove Prop. \ref{isolie}, it is convenient to rewrite the action \eqref{action} in terms of vertical fiberwise constant vector fields on $T^{*}\mathcal{F}_{L}$ instead of foliated one-forms on $(L,\mathcal{F}_{L})$.
	To do so, we use the correspondence
	\begin{equation}\label{cor}
	\big(\Omega^{\bullet}(\mathcal{F}_{L}),d_{\mathcal{F}_{L}}\big)\rightarrow\big(\mathfrak{X}^{\bullet}_{vert.const.}(T^{*}\mathcal{F}_{L}),[\Pi_{can},\boldsymbol{\cdot}]\big):\alpha\mapsto\left(\wedge^{\bullet}\Pi_{can}^{\sharp}\right)(p^{*}\alpha),
	\end{equation}
	which is an isomorphism of cochain complexes up to a global sign, i.e. it matches $d_{\mathcal{F}_{L}}$ with $-[\Pi_{can},\boldsymbol{\cdot}]$ (see for instance \cite[Lemma 2.1.3]{Normalforms}). 
	\begin{lemma} \label{rewrite}
		For every $X\in \mathfrak{X}(L)^{\mathcal{F}_{L}}$, the correspondence \eqref{cor}
		matches $\pounds_{X}$ and $\big[\widetilde{X},\boldsymbol{\cdot}\big]$, where $\widetilde{X}$ is the lift  as described in Lemma \ref{lem:lift}.
	\end{lemma}
	\begin{proof}
		For every foliated differential form $\alpha\in\Omega^{k}(\mathcal{F}_{L})$ we have to show that
		\begin{equation}\label{toprove}
		(\wedge^{k}\Pi_{can}^{\sharp})\left(p^{*}\left(\pounds_{X}\alpha\right)\right)=\big[\widetilde{X},(\wedge^{k}\Pi_{can}^{\sharp})(p^{*}\alpha)\big].
		\end{equation}
		The left hand side of this equality, using \eqref{lieder}, reads
		$$\left.\frac{d}{dt}\right|_{t=0}(\wedge^{k}\Pi_{can}^{\sharp})\left((\phi_{t}\circ p)^{*}\alpha\right).$$
		Since $p_{*}\widetilde{X}=X$, we have that $\phi_{t}\circ p= p\circ\psi_{t}$, where $\psi_{t}$ denotes the flow of $\widetilde{X}$. So 
		\begin{align*}
		(\wedge^{k}\Pi_{can}^{\sharp})\left(p^{*}\left(\pounds_{X}\alpha\right)\right)&=\left.\frac{d}{dt}\right|_{t=0}(\wedge^{k}\Pi_{can}^{\sharp})\left(\psi_{t}^{*}(p^{*}\alpha)\right)\\
		&=\left.\frac{d}{dt}\right|_{t=0}(\psi_{-t})_{*}\left((\wedge^{k}\Pi_{can}^{\sharp})\left(p^{*}\alpha\right)\right)\\
		&=\left[\widetilde{X},(\wedge^{k}\Pi_{can}^{\sharp})(p^{*}\alpha)\right],
		\end{align*}
		using in the second equality that $\psi_{t}$ is a Poisson diffeomorphism of $\left(T^{*}\mathcal{F}_{L},\Pi_{can}\right)$. So the equality \eqref{toprove} holds, and this proves the lemma.
	\end{proof}

	\begin{proof}[Proof of Prop. \ref{isolie}]
		To avoid confusion with too many brackets, we will denote equivalence classes by underlining the representatives. Fix classes $\underline{Y},\underline{Z}\in H_{\Pi_{can}}^{1}(T^{*}\mathcal{F}_{L})$ and assume that the representatives $Y,Z$ satisfy the properties in Theorem \ref{thm:modvfL}. Then also their Lie bracket $[Y,Z]$ satisfies these properties: it is $p$-projectable, $p_{*}[Y,Z]\in\mathfrak{X}(L)^{\mathcal{F}_{L}}$ and
		\begin{align}\label{br}
		[Y,Z]-\widetilde{p_{*}[Y,Z]}&=[Y,Z]-[\widetilde{p_{*}Y},\widetilde{p_{*}Z}]\nonumber\\
		&=\left[\widetilde{p_{*}Y}+\left(Y-\widetilde{p_{*}Y}\right),\widetilde{p_{*}Z}+\left(Z-\widetilde{p_{*}Z}\right)\right]-[\widetilde{p_{*}Y},\widetilde{p_{*}Z}]\nonumber\\
		&=\left[\widetilde{p_{*}Y},Z-\widetilde{p_{*}Z}\right]+\left[Y-\widetilde{p_{*}Y},\widetilde{p_{*}Z}\right],
		\end{align}
		using in the last equation that $Y-\widetilde{p_{*}Y}$ and $Z-\widetilde{p_{*}Z}$ are vertical and fiberwise constant. Lemma \ref{rewrite} shows in particular that both terms in \eqref{br} are vertical fiberwise constant Poisson vector fields, hence the same holds for $[Y,Z]-\widetilde{p_{*}[Y,Z]}$. So $[Y,Z]$ meets the criteria of Theorem \ref{thm:modvfL}. We can therefore proceed as follows:
		\begin{align*}
		\Phi\left([\underline{Y},\underline{Z}]\right)&=\Phi\left(\underline{[Y,Z]}\right)\\
		&=\left(\underline{p_{*}[Y,Z]},\underline{[Y,Z]-\widetilde{p_{*}[Y,Z]}}\right)\\
		&=\left(\underline{[p_{*}Y,p_{*}Z]},\underline{\left[\widetilde{p_{*}Y},Z-\widetilde{p_{*}Z}\right]}+\underline{\left[Y-\widetilde{p_{*}Y},\widetilde{p_{*}Z}\right]}\right)\\
		&=\left[\left(\underline{p_{*}Y},\underline{Y-\widetilde{p_{*}Y}}\right),\left(\underline{p_{*}Z},\underline{Z-\widetilde{p_{*}Z}}\right)\right]_{\rho}\\
		&=\left[\Phi(\underline{Y}),\Phi(\underline{Z})\right]_{\rho},
		\end{align*}
		using the equation \eqref{br} in the third equality and Lemma \ref{rewrite} in the fourth equality.
	\end{proof}

	\section{Deformations of Lagrangian submanifolds in log-symplectic manifolds: algebraic aspects}\label{sec:alg}
	
	In this section, we address the algebra behind the deformation problem of a Lagrangian submanifold $L^{n}$ contained in the singular locus of a log-symplectic manifold $(M^{2n},Z,\Pi)$. In \S\ref{subsec:MC}-\S\ref{corres} we show that the deformation problem is governed by a DGLA, and we discuss the corresponding Maurer-Cartan equation (Thm. \ref{equations} and Cor. \ref{DGLA}).
	We also compute the cohomology of the DGLA in degree one,
by calculating the	
	zeroth foliated Morse-Novikov cohomology in \S\ref{subsec:Nov} (Thm. \ref{H}).  
This result will be used in the next section to extract geometric information about the deformations.

\bigskip	
	To set up the stage, we revisit Corollary \ref{normal}, which states that a neighborhood of a Lagrangian submanifold $L^{n}$ contained in the singular locus of an orientable log-symplectic manifold $(M^{2n},Z,\Pi)$ can be identified with a neighborhood of the zero section in $T^{*}\mathcal{F}_{L}\times\mathbb{R}$, endowed with the log-symplectic structure
	\begin{equation}\label{pitilde}
	\widetilde{\Pi}:=V\wedge t\partial_{t}+\Pi_{can}.
	\end{equation}
	Here $V$ is defined on a neighborhood   of $L$ in $T^{*}\mathcal{F}_{L}$, and only its Poisson cohomology class $[V]$ is fixed, see Remark \ref{rem:freedom}.
	We can use Theorem \ref{thm:modvfL} to choose a convenient representative $V$ that satisfies
	\[
	V=V_{vert}+V_{lift},
	\]
	where $V_{lift}:=\widetilde{p_{*}V}$ is the cotangent lift of $p_{*}V\in\mathfrak{X}(L)^{\mathcal{F}_{L}}$ and $V_{vert}:=V-\widetilde{p_{*}V}$ is vertical, fiberwise constant and closed as a section of $p:T^{*}\mathcal{F}_{L}\rightarrow L$. Indeed, although Theorem \ref{thm:modvfL} is phrased for Poisson vector fields defined on all of $T^{*}\mathcal{F}_{L}$, it is clear that the proof still works if those vector fields are only defined on a neighborhood of $L$ in $T^{*}\mathcal{F}_{L}$ whose intersection with each fiber is {convex}. We summarize the setup for the rest of this paper:
	
	\vspace{0.2cm}
	\noindent
	\begin{mdframed}
		Given a Lagrangian submanifold $L^{n}$ contained in the singular locus $Z$ of an orientable log-symplectic manifold $(M^{2n},Z,\Pi)$, denote by $\mathcal{F}_{L}$ the induced foliation on $L$. Fix a tubular neighborhood embedding $\psi:(Z,\Pi|_{Z})\rightarrow(T^{*}\mathcal{F}_{L},\Pi_{can})$ of neighborhoods of $L$, as in Prop. \ref{normalform}. Denote by $[V]$ the image of $[V_{mod}|_{Z}]$ under this map, and assume that $V$ is a representative that satisfies the assumptions of Thm. \ref{thm:modvfL}. The local model around $L$ is then 
		\[
		\big(U\subset T^{*}\mathcal{F}_{L}\times\mathbb{R},\widetilde{\Pi}:=\left(V_{vert}+V_{lift}\right)\wedge t\partial_{t}+\Pi_{can}\big),
		\]
		{where $U$ is a neighborhood of the zero section $L$}.
		We denote by $$\gamma\in\Omega_{cl}^{1}(\mathcal{F}_{L})$$ the closed one-form defined by considering $V_{vert}$ as a section of $p:T^{*}\mathcal{F}_{L}\rightarrow L$, and we also write $$X:=p_{*}V\in\mathfrak{X}(L)^{\mathcal{F}_{L}}.$$
	\end{mdframed}

{
\begin{remark}
The Poisson cohomology class $[V]$ is completely determined by $[V_{mod}|_{Z}]$, i.e. different choices of tubular neighborhood embeddings $\psi:(Z,\Pi|_{Z})\rightarrow(T^{*}\mathcal{F}_{L},\Pi_{can})$ of neighborhoods of $L$ yield cohomologous Poisson vector fields $V$. This is a consequence of the fact that any two tubular neighborhood embeddings are isotopic; a proof can be made using the concrete isotopy constructed in the proof of \cite[Theorem 5.3]{hirsch}. As a consequence, the Poisson cohomology class $[V_{mod}|_{Z}]$ is faithfully encoded by $\big([\gamma],[X]\big)\in H^{1}(\mathcal{F}_{L})\times\mathfrak{X}(L)^{\mathcal{F}_{L}}/\Gamma(T\mathcal{F}_{L})$. 
\end{remark}
}

\bigskip	
{Slightly abusing notation, we will often denote the local model by $\big(T^{*}\mathcal{F}_{L}\times\mathbb{R},\widetilde{\Pi}\big)$ although it is only defined on the neighborhood $U$. Throughout the rest of the paper, $U$ denotes this fixed neighborhood. We only make reference to it when strictly necessary.}

	\vspace{0.1cm}
	
	\subsection{The Maurer-Cartan equation}\label{subsec:MC}
	\leavevmode
	\vspace{0.1cm}

	\vspace{0.3cm}
	Studying $\mathcal{C}^{1}$-small deformations of $L$ now amounts to studying Lagrangian sections in $\big(T^{*}\mathcal{F}_{L}\times\mathbb{R},\widetilde{\Pi}\big)$,
	the vector bundle over $L$ given by the Whitney sum of $T^{*}\mathcal{F}_{L}$ and the trivial $\RR$-bundle.
	By the following little lemma, it is equivalent to look at coisotropic sections.
	
	\begin{lemma}\label{lagcois}
		The graph of a section $(\alpha,f)\in\Gamma(T^{*}\mathcal{F}_{L}\times\mathbb{R})$ is coisotropic iff  it is Lagrangian.
	\end{lemma}
	\begin{proof}
		We only have to check the forward implication at points $(\alpha(q),0)$ inside the singular locus $T^{*}\mathcal{F}_{L}\times\{0\}$. The symplectic leaf of $\big(T^{*}\mathcal{F}_{L}\times\mathbb{R},\widetilde{\Pi}\big)$ through $(\alpha(q),0)$ is given by $p^{-1}(\mathcal{O})\times\{0\}$, where $\mathcal{O}$ is the leaf of $\mathcal{F}_{L}$ through $q$. By assumption, the subspace
		\begin{equation}\label{subspace}
		T_{(\alpha(q),0)}\text{Graph}(\alpha,f)\cap T_{(\alpha(q),0)}\big(p^{-1}(\mathcal{O})\times\{0\}\big)=\left\{(d_{q}\alpha)(v): v\in T_{q}\mathcal{O}\ \text{and}\ (d_{q}f)(v)=0\right\}
		\end{equation}
		is coisotropic in $T_{(\alpha(q),0)}\big(p^{-1}(\mathcal{O})\times\{0\}\big)$, so it is at least $(n-1)$-dimensional. But clearly the right hand side of \eqref{subspace} is at most $(n-1)$-dimensional, which shows that the subspace \eqref{subspace} is Lagrangian in $T_{(\alpha(q),0)}\big(p^{-1}(\mathcal{O})\times\{0\}\big)$. 
	\end{proof}
	
	We now derive the equations that cut out coisotropic sections in $\big(T^{*}\mathcal{F}_{L}\times\mathbb{R},\widetilde{\Pi}\big)$.
	
	\begin{thm}\label{equations}
		The graph of a section $(\alpha,f)\in\Gamma(T^{*}\mathcal{F}_{L}\times\mathbb{R})$ is coisotropic in $\big(T^{*}\mathcal{F}_{L}\times\mathbb{R},\widetilde{\Pi}\big)$ exactly when
		\begin{equation}\label{eqns}
		\boxed{
			\begin{cases}
			d_{\mathcal{F}_{L}}\alpha=0\\
			d_{\mathcal{F}_{L}}f+f(\gamma-\pounds_{X}\alpha)=0
			\end{cases}
		}
		\end{equation}
		
	\end{thm}
	
	Recall that any vector bundle $E\rightarrow L$ carries natural maps $\wedge^{\bullet}P_{E}:\mathfrak{X}^{\bullet}(E)\rightarrow \Gamma(\wedge^{\bullet}E)$, given by restriction to $L$ composed with the vertical projection $\Gamma\left(\wedge^{\bullet}TE|_{L}\right)\rightarrow\Gamma(\wedge^{\bullet}E)$. In particular, for $E:=T^{*}\mathcal{F}_{L}\times\mathbb{R}$ we have the following map in degree two:
	\[
	\wedge^{2}P_{E}:\mathfrak{X}^{2}(E)\rightarrow\Gamma(\wedge^{2}T^{*}\mathcal{F}_{L})\oplus\Gamma(T^{*}\mathcal{F}_{L}).
	\]
	It is clear that $L$ is coisotropic with respect to $\widetilde{\Pi}$ if and only if $\wedge^{2}P_{E}(\widetilde{\Pi})=0$. Below, we denote the bundle projections by $pr_{E}:E\rightarrow L$ and  $pr_{T^{*}\mathcal{F}_{L}}:T^{*}\mathcal{F}_{L}\rightarrow L$ respectively.
	
	\begin{proof}[Proof of Thm. \ref{equations}]
		A section $(\alpha,f)\in\Gamma(E)$ gives rise to a diffeomorphism 
		\[
		\phi^{(-\alpha,-f)}:E\rightarrow E:(p,\xi,t)\mapsto \big(p,\xi-\alpha(p),t-f(p)\big)
		\]
		which maps the graph of $(\alpha,f)$ to the zero section $L\subset E$. So it suffices to single out the sections $(\alpha,f)$ such that $L$ is coisotropic with respect to $\phi^{(-\alpha,-f)}_{*}\widetilde{\Pi}$. This amounts to asking that
		\begin{equation}\label{proj}
		0=\wedge^{2}P_{E}\left(\phi^{(-\alpha,-f)}_{*}\big(\left(V_{vert}+V_{lift}\right)\wedge t\partial_{t}+\Pi_{can}\big)\right).
		\end{equation}
		We now simplify the expression \eqref{proj} in two steps, identifying throughout vertical fiberwise constant vector fields on $T^{*}\mathcal{F}_{L}$ with foliated one-forms on $(L,\mathcal{F}_{L})$ via the bijection \eqref{cor}.
		\begin{enumerate}[i)]
			\item \underline{Claim 1:} $\wedge^{2}P_{E}\left(\phi^{(-\alpha,-f)}_{*}\big(\left(V_{vert}+V_{lift}\right)\wedge t\partial_{t}\big)\right)=\left(0,f\gamma-f\pounds_{X}\alpha\right).$ \hfill (*)
			
			\vspace{0.1cm}
			\noindent
			First, we note that
			\[
			P_{E}\left(\phi_{*}^{(-\alpha,-f)}V_{vert}\right)=P_{T^{*}\mathcal{F}_{L}}(V_{vert})=V_{vert}
			\]
			and
			\[
			P_{E}\left(\phi_{*}^{(-\alpha,-f)}t\partial_{t}\right)=P_{E}\big((t+pr_{E}^{*}f)\partial_{t}\big)=(pr_{E}^{*}f)\partial_{t},
			\]
			which yields the first term on the right in (*). Secondly, we have
			\begin{align*}
			\wedge^{2}P_{E}\left(\phi_{*}^{(-\alpha,-f)}\left(V_{lift}\wedge t\partial_{t}\right)\right)&=P_{E}\left(\phi_{*}^{(-\alpha,-f)}V_{lift}\right)\wedge(pr_{E}^{*}f)\partial_{t}\\
			&=P_{T^{*}\mathcal{F}_{L}}\left(\phi^{-\alpha}_{*}V_{lift}\right)\wedge(pr_{E}^{*}f)\partial_{t},
			\end{align*}
			so Claim 1 follows if we show that
			\begin{equation}\label{toshow}
			P_{T^{*}\mathcal{F}_{L}}\left(\phi^{-\alpha}_{*}V_{lift}\right)=-\pounds_{X}\alpha.
			\end{equation}
			To do so, we compute
			\begin{align}\label{cc}
			P_{T^{*}\mathcal{F}_{L}}\left(\phi^{-\alpha}_{*}V_{lift}\right)&=\left.\left(\phi^{-\alpha}_{*}V_{lift}\right)\right|_{L}-\left(pr_{T^{*}\mathcal{F}_{L}}\right)_{*}(\phi^{-\alpha}_{*}V_{lift})\nonumber\\
			&=\left.\left(\phi^{-\alpha}_{*}V_{lift}\right)\right|_{L}-\left(pr_{T^{*}\mathcal{F}_{L}}\right)_{*}(V_{lift})\nonumber\\
			&=\left.\left(\phi^{-\alpha}_{*}V_{lift}\right)\right|_{L}-V_{lift}|_{L}\nonumber\\
			&=\int_{0}^{1}\left.\frac{d}{dt}\left(\phi^{-t\alpha}_{*}V_{lift}\right)\right|_{L}dt,
			\end{align}
			using that $pr_{T^{*}\mathcal{F}_{L}}\circ\phi^{-\alpha}=pr_{T^{*}\mathcal{F}_{L}}$. Now, note that
			\begin{align}\label{use}
			\frac{d}{dt}\left(\phi^{-t\alpha}_{*}V_{lift}\right)&=\left.\frac{d}{ds}\right|_{s=0}\phi^{-t\alpha}_{*}\left(\phi^{-s\alpha}_{*}V_{lift}\right)\nonumber\\
			&=\phi^{-t\alpha}_{*}\left(\left.\frac{d}{ds}\right|_{s=0}\phi^{-s\alpha}_{*}V_{lift}\right)\nonumber\\
			&=\phi^{-t\alpha}_{*}\left([\alpha,V_{lift}]\right)\nonumber\\
			&=[\alpha,V_{lift}]\nonumber\\
			&=-[V_{lift},\alpha],
			\end{align}
			In the fourth equality, we used that $[\alpha,V_{lift}]$ is vertical and fiberwise constant, which follows from Lemma \ref{rewrite}. Therefore, the equality \eqref{cc} becomes
			\begin{align*}
			P_{T^{*}\mathcal{F}_{L}}\left(\phi^{-\alpha}_{*}V_{lift}\right)&=-\int_{0}^{1}\left.[V_{lift},\alpha]\right|_{L}dt=-\left.[V_{lift},\alpha]\right|_{L},
			\end{align*}
			which is exactly \eqref{toshow} under the identification \eqref{cor}. 
			This proves Claim 1.
			
			\vspace{0.1cm}
			
			\item \underline{Claim 2:} $\wedge^{2}P_{E}\left(\phi^{(-\alpha,-f)}_{*}\Pi_{can}\right)=\left(d_{\mathcal{F}_{L}}\alpha,d_{\mathcal{F}_{L}}f\right)$.\hfill (**)
			
			\vspace{0.1cm}
			\noindent
			Since $L\subset\left(T^{*}\mathcal{F}_{L},\Pi_{can}\right)$ is Lagrangian, we have $\wedge^{2}P_{E}(\Pi_{can})=\wedge^{2}P_{T^{*}\mathcal{F}_{L}}\Pi_{can}=0$, so the left hand side of (**) is equal to
			\begin{equation}\label{lhs}
			\wedge^{2}P_{E}\left(\phi^{(-\alpha,-f)}_{*}\Pi_{can}-\Pi_{can}\right).
			\end{equation}
			We can decompose
			\begin{equation}\label{dec}
			\phi^{(-\alpha,-f)}_{*}\Pi_{can}-\Pi_{can}=A_{t}\wedge\partial_{t}+B_{t}
			\end{equation}
			for some $A_{t}\in\mathfrak{X}(T^{*}\mathcal{F}_{L})$ and $B_{t}\in\mathfrak{X}^{2}(T^{*}\mathcal{F}_{L})$ depending smoothly on $t$. We find $A_{t}$ by contracting with $dt$:
			\begin{align}\label{At}
			A_{t}&=-\left(\phi^{(-\alpha,-f)}_{*}\Pi_{can}-\Pi_{can}\right)^{\sharp}(dt)\nonumber\\
			&=-\left[\phi^{(-\alpha,-f)}_{*}\circ\Pi_{can}^{\sharp}\circ\left(\phi^{(-\alpha,-f)}\right)^{*}\right](dt)\nonumber\\
			&=-\phi^{(-\alpha,-f)}_{*}\left(\Pi_{can}^{\sharp}\left(d(t-pr_{E}^{*}f)\right)\right)\nonumber\\
			&=\phi^{(-\alpha,-f)}_{*}\left(X_{pr_{T^{*}\mathcal{F}_{L}}^{*}f}\right)\nonumber\\
			&=X_{pr_{T^{*}\mathcal{F}_{L}}^{*}f},
			\end{align}
			using that $X_{pr_{T^{*}\mathcal{F}_{L}}^{*}f}$ is vertical and fiberwise constant. Next, since
			\begin{align*}
			\pounds_{\partial_{t}}\left(\phi^{(-\alpha,-f)}_{*}\Pi_{can}\right)&=\left.\frac{d}{dt}\right|_{t=0}\phi_{*}^{(0,-t)}\left(\phi^{(-\alpha,-f)}_{*}\Pi_{can}\right)\\
			&=\phi^{(-\alpha,-f)}_{*}\left(\left.\frac{d}{dt}\right|_{t=0}\phi_{*}^{(0,-t)}\Pi_{can}\right)\\
			&=\phi^{(-\alpha,-f)}_{*}(\pounds_{\partial_{t}}\Pi_{can})\\
			&=0,
			\end{align*}
			it follows that 
			\[
			\pounds_{\partial_{t}}B_{t}=\pounds_{\partial_{t}}\left(\phi^{(-\alpha,-f)}_{*}\Pi_{can}-\Pi_{can}-X_{pr_{T^{*}\mathcal{F}_{L}}^{*}f}\wedge\partial_{t}\right)=0,
			\]
			i.e. $B_{t}=B$ is independent of $t$. So $B$ is equal to its pushforward under the projection $T^{*}\mathcal{F}_{L}\times\mathbb{R}\rightarrow T^{*}\mathcal{F}_{L}$, which yields
			\begin{align}\label{B}
			B&=\phi^{-\alpha}_{*}\Pi_{can}-\Pi_{can}\nonumber\\
			&=\int_{0}^{1}\frac{d}{dt}\left(\phi^{-t\alpha}_{*}\Pi_{can}\right)dt\nonumber\\
			&=\int_{0}^{1}\pounds_{\alpha}\Pi_{can}dt\nonumber\\
			&=\pounds_{\alpha}\Pi_{can}.
			\end{align}
			Here the third equality follows from a computation similar to the one that led to \eqref{use}. Inserting \eqref{At} and \eqref{B} into \eqref{dec} gives
			\[
			\phi^{(-\alpha,-f)}_{*}\Pi_{can}-\Pi_{can}=X_{pr_{T^{*}\mathcal{F}_{L}}^{*}f}\wedge\partial_{t}+\pounds_{\alpha}\Pi_{can}.
			\]
			Applying the identification \eqref{cor} now yields the conclusion of Claim 2:
			\[
			\wedge^{2}P_{E}\left(\phi^{(-\alpha,-f)}_{*}\Pi_{can}-\Pi_{can}\right)=\wedge^{2}P_{E}\left(X_{pr_{T^{*}\mathcal{F}_{L}}^{*}f}\wedge\partial_{t}+\pounds_{\alpha}\Pi_{can}\right)=(d_{\mathcal{F}_{L}}\alpha,d_{\mathcal{F}_{L}}f).
			\]
		\end{enumerate}
		Combining Claim 1 and Claim 2, we see that the requirement \eqref{proj} is equivalent with the equations \eqref{eqns} in the statement of the theorem.
	\end{proof}
	
	{Recall that in symplectic geometry, the space of Lagrangian submanifolds close to a given one is contractible, hence path connected. Using Theorem \ref{equations}, we now show that the same statement is true in our setting.}
	
	\begin{cor}\label{connected}
		{The set of Lagrangian sections of $\big(T^{*}\mathcal{F}_{L}\times\mathbb{R},\widetilde{\Pi}\big)$ is contractible when endowed with the compact-open topology.}
	\end{cor}

\begin{proof}
{
Let us denote for short
\[
\text{Def}(L):=\left\{(\alpha,f)\in\Gamma(T^{*}\mathcal{F}_{L}\times\mathbb{R}):\ \text{graph}(\alpha,f)\ \text{is Lagrangian in}\ \big(T^{*}\mathcal{F}_{L}\times\mathbb{R},\widetilde{\Pi}\big)\right\}.
\]
We will construct a deformation retraction $F$ of $\text{Def}(L)$ onto the zero section. To this end, fix a smooth function $\Phi\in C^{\infty}(\RR)$ satisfying
\[
\begin{cases}
\Phi(s)=1\ \hspace{1.2cm}\forall\ s\leq 0\\
0<\Phi(s)<1\ \hspace{0.5cm}\forall\ 0<s<\frac{1}{2}\\
\Phi(s)=0\ \hspace{1.2cm}\forall\ s\geq\frac{1}{2}
\end{cases},
\]
and define $\Psi\in C^{\infty}(\RR)$ by setting \[\Psi(s):=\Phi\left(s-\frac{1}{2}\right).\]
Note that $\Psi\equiv 1$ on the support of $\Phi$, so that in particular $\Phi\cdot\Psi=\Phi$. We now define the map $F$ as follows:
\[
F:\text{Def}(L)\times[0,1]\rightarrow\text{Def}(L):\big((\alpha,f),s\big)\mapsto\big(\Psi(s)\alpha,\Phi(s)f\big).
\]
This map is well-defined: if $(\alpha,f)$ solves the equations \eqref{eqns}, then also $\big(\Psi(s)\alpha,\Phi(s)f\big)$ solves these equations because $\Phi\cdot\Psi=\Phi$. Clearly, the map $F$ is continuous when $\text{Def}(L)$ is endowed with the compact-open topology. Next, for $(\alpha,f)\in\text{Def}(L)$ we have $F\big((\alpha,f),0\big)=(\alpha,f)$ and $F\big((\alpha,f),1\big)=(0,0)$. Since clearly $F\big((0,0),s\big)=(0,0)$ for all $s\in[0,1]$, we conclude that $F$ is indeed a deformation retraction of $\text{Def}(L)$ onto $L$.}

\end{proof}

	\begin{remark}\label{rem}
		We comment on the Maurer-Cartan equation \eqref{eqns}.
		\begin{enumerate}[i)]
			\item Twisting the foliated de Rham differential with a closed element $\eta\in\Omega^{1}(\mathcal{F}_{L})$ gives a differential
			\begin{equation}\label{m-n}
			d^{\eta}_{\mathcal{F}_{L}}:\Omega^{k}(\mathcal{F}_{L})\rightarrow\Omega^{k+1}(\mathcal{F}_{L}):\alpha\mapsto d_{\mathcal{F}_{L}}\alpha+\eta\wedge\alpha.
			\end{equation}
			The associated cohomology groups, which we denote by $H^{k}_{\eta}(\mathcal{F}_{L})$, will be discussed in more detail later. If $\mathcal{F}_{L}$ is the one-leaf foliation on $L$, then we recover what is called the Morse-Novikov cohomology, which appears in the context of locally conformal symplectic structures \cite[Section 1]{HR}.
			\item The Maurer-Cartan equation \eqref{eqns} shows that the problem of deforming $L$ into a nearby Lagrangian $\text{Graph}(\alpha,f)$ can essentially be done in two steps. Indeed, one can  solve the first (linear) equation in \eqref{eqns} for $\alpha$,
			and then solve the second equation -- which for fixed $\alpha$ becomes  linear --  for $f$. Geometrically, this amounts to the following.
			First, one deforms $L$ inside the singular locus along the leafwise closed one-form $\alpha$, and then one moves the obtained Lagrangian $L'=\text{Graph}(\alpha)\subset T^{*}\mathcal{F}_{L}$ in the direction normal to $T^{*}\mathcal{F}_{L}$ along the function $f\in H_{\gamma-\pounds_{X}\alpha}^{0}(\mathcal{F}_{L})$. 
			
			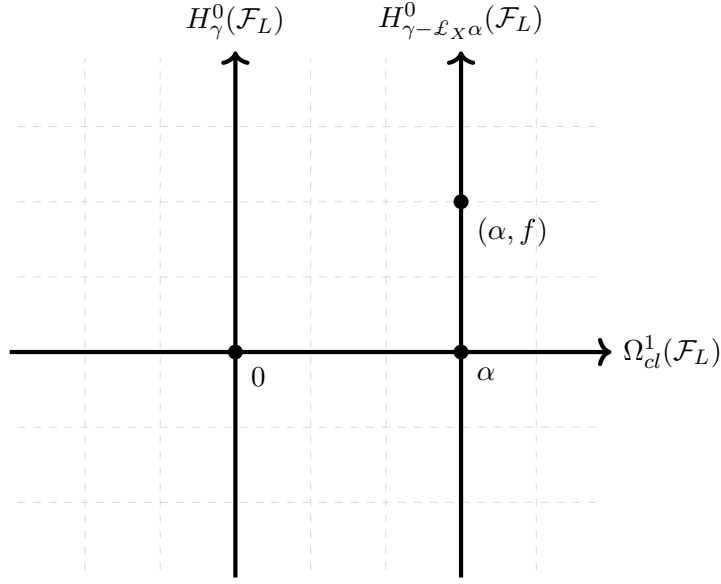
\begin{figure}[H]
				\centering
				\begin{tikzpicture}
				\draw[help lines, color=gray!30, dashed] (-2.9,-2.9) grid (4.9,3.9);
				\draw[->,ultra thick] (-3,0)--(5,0) node[right]{$\Omega^{1}_{cl}(\mathcal{F}_{L})$};
				\draw[->,ultra thick] (0,-3)--(0,4) node[above]{$H^{0}_{\gamma}(\mathcal{F}_{L})$};
				\draw[->,ultra thick] (3,-3)--(3,4) node[above]{$H^{0}_{\gamma-\pounds_{X}\alpha}(\mathcal{F}_{L})$};
				\node[label={315:{$0$}},circle,fill,inner sep=2pt] at (0,0) {};
				\node[label={315:{$\alpha$}},circle,fill,inner sep=2pt] at (3,0) {};
				\node[label={315:{$(\alpha,f)$}},circle,fill,inner sep=2pt] at (3,2) {};
				\end{tikzpicture}
				\caption{Deforming $L$ into $\text{Graph}(\alpha,f)$.}
			\end{figure}
		\end{enumerate}
	\end{remark}

	So heuristically, it seems like deforming $L$ into $\text{Graph}(\alpha)\subset T^{*}\mathcal{F}_{L}$ for closed $\alpha\in\Omega^{1}(\mathcal{F}_{L})$ transforms $\gamma$ into $\gamma-\pounds_{X}\alpha$. We will now make this precise.

	\begin{prop}
		Let
		$
		\psi:U_{0}\subset \left(Z,\Pi|_{Z}\right) \rightarrow U_{1}\subset \left(T^{*}\mathcal{F}_{L},\Pi_{T^{*}\mathcal{F}_{L}}\right)
		$
		be a fixed tubular neighborhood embedding of $L$ into $T^{*}\mathcal{F}_{L}$,
		where $\Pi_{T^{*}\mathcal{F}_{L}}$ denotes the canonical Poisson structure $\Pi_{can}$. Assume that $V=V_{vert}+V_{lift}$ is a representative of the Poisson cohomology class $\psi_{*}[V_{mod}|_{Z}]$ satisfying the requirements of Thm. \ref{thm:modvfL}, with associated data $(X,\gamma)\in\mathfrak{X}(L)^{\mathcal{F}_{L}}\times\Omega^{1}_{cl}(\mathcal{F}_{L})$. Consider a Lagrangian $L'=\text{Graph}(\alpha)\subset U_{1}\subset T^{*}\mathcal{F}_{L}$ for some closed $\alpha\in\Omega^{1}(\mathcal{F}_{L})$. Then the following hold:
		\begin{enumerate}[i)]
			\item There is a canonical diffeomorphism of affine bundles
			\[
			\left(\Phi,\phi\right):\left(T^{*}\mathcal{F}_{L},\Pi_{T^{*}\mathcal{F}_{L}}\right)\rightarrow\left(T^{*}\mathcal{F}_{L'},\Pi_{T^{*}\mathcal{F}_{L'}}\right)
			\] 
			which is a Poisson diffeomorphism between the total spaces and fixes points of $L'$, so that $\Phi\circ\psi$ is a tubular neighborhood embedding of $\psi^{-1}(L')$ into $\left(T^{*}\mathcal{F}_{L'},\Pi_{T^{*}\mathcal{F}_{L'}}\right)$.
			\item The representative $\Phi_{*}V$ satisfies the requirements of Thm. \ref{thm:modvfL} too, and its associated data are $(X',\gamma')=\Big(\phi_{*}X,\left(\phi^{-1}\right)^{*}(\gamma-\pounds_{X}\alpha)\Big)\in\mathfrak{X}(L')^{\mathcal{F}_{L'}}\times\Omega^{1}_{cl}(\mathcal{F}_{L'})$.
		\end{enumerate}
	\end{prop}
	\begin{proof}
		\begin{enumerate}[i)]
			\item Since $\alpha$ is closed, the translation map
			\[
			\phi^{-\alpha}:\left(T^{*}\mathcal{F}_{L},\Pi_{T^{*}\mathcal{F}_{L}}\right)\overset{\sim}{\rightarrow}\left(T^{*}\mathcal{F}_{L},\Pi_{T^{*}\mathcal{F}_{L}}\right):(p,\xi)\mapsto(p,\xi-\alpha(p))
			\]
			is a Poisson diffeomorphism; 	this follows from the computation \eqref{B} and the isomorphism \eqref{cor}. Its restriction to $L'$, which coincides with the restriction of the vector bundle projection $p_L$  to $L'$, is a foliated diffeomorphism $\tau:(L',\mathcal{F}_{L'})\overset{\sim}{\rightarrow} (L,\mathcal{F}_{L})$, and the cotangent lift $T^{*}\tau$ of $\tau$ descends to a Poisson diffeomorphism
			\[
			T^{*}_{\mathcal{F}}\tau:\left(T^{*}\mathcal{F}_{L},\Pi_{T^{*}\mathcal{F}_{L}}\right)\overset{\sim}{\rightarrow}\left(T^{*}\mathcal{F}_{L'},\Pi_{T^{*}\mathcal{F}_{L'}}\right).
			\]
			In summary, we have a commutative diagram 
			\begin{equation}\label{dia}
			\begin{tikzcd}[column sep=large, row sep=large]
			\left(T^{*}L',\Pi_{T^{*}L'}\right)\arrow{d}{r_{L'}} & \left(T^{*}L,\Pi_{T^{*}L}\right) \arrow{d}{r_{L}} \arrow{l}{T^{*}\tau}\\
			\left(T^{*}\mathcal{F}_{L'},\Pi_{T^{*}\mathcal{F}_{L'}}\right)\arrow{d}{p_{L'}}  &\left(T^{*}\mathcal{F}_{L},\Pi_{T^{*}\mathcal{F}_{L}}\right)\arrow{d}{p_{L}}\arrow{l}{T^{*}_{\mathcal{F}}\tau}\\
			(L',\mathcal{F}_{L'})\arrow{r}{\tau} & (L,\mathcal{F}_{L})
			\end{tikzcd}.
			\end{equation}
			
			\noindent
			The affine bundle map $(\Phi,\phi):=\left(T^{*}_{\mathcal{F}}\tau\circ\phi^{-\alpha},\tau^{-1}\right)$ meets the requirements.
			\item Clearly $\Phi_{*}V$ is $p_{L'}$-projectable, since $V$ is $p_{L}$-projectable and $\Phi$ covers the diffeomorphism $\phi$.	Using that $p_{L}\circ\phi^{-\alpha}=p_{L}$, we have
			\[
			\left(p_{L'}\right)_{*}(\Phi_{*}V)=\left(p_{L'}\circ T^{*}_{\mathcal{F}}\tau\circ\phi^{-\alpha}\right)_{*}V=\left(\tau^{-1}\circ p_{L}\circ\phi^{-\alpha}\right)_{*}V=\left(\tau^{-1}\right)_{*}X.
			\]
			Since $\tau^{-1}:(L,\mathcal{F}_{L})\rightarrow (L',\mathcal{F}_{L'})$ is a foliated diffeomorphism and $X\in\mathfrak{X}(L)^{\mathcal{F}_{L}}$, we also have $\left(\tau^{-1}\right)_{*}X\in\mathfrak{X}(L')^{\mathcal{F}_{L'}}$. Moreover, 
			$\widetilde{\left(\tau^{-1}\right)_{*}X}=\left(T^{*}_{\mathcal{F}}\tau\right)_{*}V_{lift}$ by functoriality.
			It remains to show that $\Phi_{*}V- \widetilde{(p_{L'})_{*}(\Phi_{*}V)}=
			\left(T^{*}_{\mathcal{F}}\tau\circ\phi^{-\alpha}\right)_{*}V-\left(T^{*}_{\mathcal{F}}\tau\right)_{*}V_{lift}$ is vertical, fiberwise constant, and that it corresponds with the closed foliated one-form $\tau^{*}\left(\gamma-\pounds_{X}\alpha\right)\in\Omega^{1}(\mathcal{F}_{L'})$. We rewrite it as
			\begin{align}\label{vec}
			&\left(T^{*}_{\mathcal{F}}\tau\right)_{*}\left[\left(\phi^{-\alpha}\right)_{*}(V-V_{lift})\right]+\left(T^{*}_{\mathcal{F}}\tau\right)_{*}\left[\left(\phi^{-\alpha}\right)_{*}V_{lift}-V_{lift}\right]\nonumber\\
			&\hspace{0.5cm}=\left(T^{*}_{\mathcal{F}}\tau\right)_{*}(V-V_{lift})+\left(T^{*}_{\mathcal{F}}\tau\right)_{*}\left[\left(\phi^{-\alpha}\right)_{*}V_{lift}-V_{lift}\right],
			\end{align}
			using that $V-V_{lift}$ is vertical and fiberwise constant.  {The computations done in \eqref{cc} and \eqref{use} show that} $\left(\phi^{-\alpha}\right)_{*}V_{lift}-V_{lift}=-[V_{lift},\alpha]$,  so it is vertical fiberwise constant and it corresponds with the closed one-form $-\pounds_{X}\alpha\in\Omega^{1}(\mathcal{F}_{L})$ under the identification \eqref{cor}. Since $V-V_{lift}$ corresponds with $\gamma\in\Omega^{1}(\mathcal{F}_{L})$, we get that the vertical fiberwise constant vector field \eqref{vec} indeed corresponds with the closed one-form $\tau^{*}\left(\gamma-\pounds_{X}\alpha\right)$.
		\end{enumerate}
	\end{proof}

	\subsection{The DGLA behind the deformation problem}\label{corres}
	\leavevmode
	\vspace{0.1cm}
	
	We now show that the equations \eqref{eqns} obtained in Theorem \ref{equations} represent the Maurer-Cartan equation of a differential graded Lie algebra (DGLA) that governs the deformations of the Lagrangian $L\subset\big(T^{*}\mathcal{F}_{L}\times\mathbb{R},\widetilde{\Pi}\big)$.
	To this end, recall the following.
	
	Suppose $E\rightarrow C$ is a vector bundle and let $\Pi$ be a Poisson structure on $E$ such that $C$ is coisotropic. Cattaneo and Felder showed in \cite{CaFeCo2} that the graded vector space $\Gamma\left(\wedge^{\bullet}E\right)[1]$ supports a canonical $L_{\infty}[1]$-algebra structure whose multibrackets are defined by
	\begin{equation}\label{CaFe}
	\lambda_{k}:\Gamma\left(\wedge^{\bullet}E\right)[1]^{\otimes^{k}}\rightarrow\Gamma\left(\wedge^{\bullet}E\right)[1]:\xi_{1}\otimes\cdots\otimes\xi_{k}\mapsto \wedge P\left([\ldots[[\Pi,\xi_{1}],\xi_{2}]\ldots,\xi_{k}]\right).
	\end{equation}
	Here the $\xi_{i}$ are interpreted as vertical fiberwise constant multivector fields on $E$ and the map $\wedge^{\bullet}P:\mathfrak{X}^{\bullet}(E)\rightarrow\Gamma(\wedge^{\bullet}E)$ is the restriction to $C$ composed with the vertical projection $\Gamma\left(\wedge^{\bullet}TE|_{L}\right)\rightarrow\Gamma(\wedge^{\bullet}E)$. These structure maps $\lambda_{k}$ only depend on the $\infty$-jet of $\Pi$ along the submanifold $C$, so the $L_{\infty}[1]$-algebra usually does not carry enough information to codify $\Pi$ in a neighborhood of $C$.
	Consequently, this $L_{\infty}[1]$-algebra fails to encode coisotropic deformations of $C$ in general (see \cite[Ex. 3.2]{Schatz}).
	
	However, if the Poisson structure $\Pi$ is analytic in the fiber directions, then the $L_{\infty}[1]$-algebra of Cattaneo-Felder does govern the smooth coisotropic deformation problem of $C$. In \cite{fiber}, such bivector fields are called fiberwise entire, and there one proves the following.
	\begin{thm}\cite[Thm. 1.12]{fiber}\label{entire}
		Let $E\rightarrow C$ be a vector bundle and $\Pi$ a fiberwise entire Poisson structure which is defined on a tubular neighborhood $U$ of $C$ in $E$. Suppose that $C$ is coisotropic with respect to $\Pi$, and consider the $L_{\infty}[1]$-algebra associated with $C\subset (U,\Pi)$. For any section $\alpha\in\Gamma(E)$ such that $\text{Graph}(-\alpha)$ is contained in $U$, the Maurer-Cartan series $MC(\alpha)$ converges. Furthermore, for any such $\alpha\in\Gamma(E)$, the following are equivalent:
		\begin{enumerate}
			\item The graph of $-\alpha$ is a coisotropic submanifold of $(U,\Pi)$.
			\item The Maurer-Cartan series $MC(\alpha)$ converges to zero.
		\end{enumerate}
	\end{thm}
	
	In the rest of this section, we show that the $L_{\infty}$-algebra of Cattaneo-Felder associated with $\big(T^{*}\mathcal{F}_{L}\times\mathbb{R},\widetilde{\Pi}\big)$ reduces to a DGLA, and that this DGLA governs the deformation problem of the Lagrangian $L$.
	
	\begin{lemma}\label{fiber}
		The Poisson structure $\widetilde{\Pi}=\left(V_{vert}+V_{lift}\right)\wedge t\partial_{t}+\Pi_{can}$, defined on a neighborhood $U$ of $L$ in $T^{*}\mathcal{F}_{L}\times\mathbb{R}$, is fiberwise entire.
	\end{lemma}
	\begin{proof}
		This is straightforward computation. Choose coordinates $(x_{1},\ldots,x_{n})$ on $L$ adapted to the foliation $\mathcal{F}_{L}$, such that plaques of $\mathcal{F}_{L}$ are level sets of $x_{1}$. Let $(y_{1},\ldots,y_{n})$ be the corresponding fiber coordinates on $T^{*}L$. Then write
		\begin{align*}
		\Pi_{can}=\sum_{i=2}^{n}\partial_{x_{i}}\wedge\partial_{y_{i}},\hspace{0.5cm}
		V_{vert}=\sum_{j=2}^{n}f_{j}(x)\partial_{y_{j}},\hspace{0.5cm}
		p_{*}V=\sum_{j=1}^{n}h_{j}(x)\partial_{x_{j}},
		\end{align*}
		where $p:T^{*}\mathcal{F}_{L}\rightarrow L$ is the projection and $h_{1}(x)$ only depends on $x_{1}$ since $p_{*}V\in\mathfrak{X}(L)^{\mathcal{F}_{L}}$.
		We then obtain
		\[
		V_{lift}=\sum_{j=1}^{n}h_{j}(x)\partial_{x_{j}}-\sum_{i=2}^{n}\sum_{j=2}^{n}y_{j}\frac{\partial h_{j}}{\partial x_{i}}(x)\partial_{y_{i}}.
		\]
		So the Poisson structure $\widetilde{\Pi}$ reads
		\begin{equation}\label{eq:coordinate-expression}
		\widetilde{\Pi}=\left(\sum_{j=2}^{n}f_{j}(x)\partial_{y_{j}}+\sum_{j=1}^{n}h_{j}(x)\partial_{x_{j}}-\sum_{i=2}^{n}\sum_{j=2}^{n}y_{j}\frac{\partial h_{j}}{\partial x_{i}}(x)\partial_{y_{i}}
		\right)\wedge t\partial_{t}+\sum_{i=2}^{n}\partial_{x_{i}}\wedge\partial_{y_{i}},
		\end{equation}
		which is clearly a fiberwise entire bivector field.
	\end{proof}
	
	{Notice that the coefficients in the coordinate expression \eqref{eq:coordinate-expression} are at most quadratic in the fiber coordinates $(y_2,\ldots,y_n,t)$. This has the following consequence.}
	
	\begin{lemma}\label{dgla}
		The $L_{\infty}[1]$-algebra $\big(\Gamma(\wedge^{\bullet}(T^{*}\mathcal{F}_{L}\times\mathbb{R}))[1],\{\lambda_{k}\}\big)$ of Cattaneo-Felder associated with  $\big(T^{*}\mathcal{F}_{L}\times\mathbb{R},\widetilde{\Pi}\big)$ corresponds to a DGLA-structure on $\Gamma(\wedge^{\bullet}(T^{*}\mathcal{F}_{L}\times\mathbb{R}))$.
	\end{lemma}
	\begin{proof}
		We will show that the multibrackets $\lambda_{k}$ defined in \eqref{CaFe} vanish for $k\geq 3$. {Choose coordinates $(x_1,\ldots,x_n,y_2,\ldots,y_n,t)$ on $T^{*}\mathcal{F}_{L}\times\RR$ as in the proof of Lemma \ref{fiber} above.
		Since the $\lambda_{k}$ are multiderivations, it is enough to evaluate them on basic functions and vertical coordinate vector fields. 
		As the  $\lambda_{k}$ have degree one, a degree counting argument shows that 
		they can be non-zero only when evaluated on tuples of the form 
		\[
		\big(Y_{1},\ldots,Y_{k}\big),\ \  \big(h(x),Y_{1},\ldots,Y_{k-1}\big)\ \  \text{and}\ \ \big(h(x),h'(x),Y_{1},\ldots,Y_{k-2}\big),
		\] 
		where $Y_{1},\ldots,Y_{k}\in\{\partial_{y_2},\ldots,\partial_{y_n},\partial_{t}\}$. Firstly, it is clear from the expression \eqref{eq:coordinate-expression} that
		\[
		\left[\left[\widetilde{\Pi},h(x)\right],h'(x)\right]=0.
		\]
		Secondly, the expression \eqref{eq:coordinate-expression} shows that the multivector fields
		\[
		\left[\left[\widetilde{\Pi},Y_1\right],Y_2\right] \text{ and }
		\left[\left[\widetilde{\Pi},h(x)\right],Y_1\right]
		\]
		are vertical and fiberwise constant, hence differentiating them once more along a vertical coordinate vector field gives zero. These observations imply that $\lambda_{k}=0$ whenever $k\geq 3$.}
	\end{proof}
	
	We now established the existence of a DGLA-structure supported on $\Gamma(\wedge^{\bullet}(T^{*}\mathcal{F}_{L}\times\mathbb{R}))$ which governs the deformations of $L$ as a coisotropic submanifold. Thanks to Lemma \ref{lagcois}, this DGLA in fact governs the Lagrangian deformation problem of $L$. We now provide more explicit descriptions for the structure maps of the DGLA.
	
	\begin{cor}\label{DGLA}
		The deformation problem of a Lagrangian submanifold $L^{n}$ contained in the singular locus of an {orientable} log-symplectic manifold $(M^{2n},Z,\Pi)$ is governed by a DGLA supported on the graded vector space $\Gamma\left(\wedge^{\bullet}\left(T^{*}\mathcal{F}_{L}\times\mathbb{R}\right)\right)=\Gamma\left(\wedge^{\bullet}T^{*}\mathcal{F}_{L}\oplus\wedge^{\bullet-1}T^{*}\mathcal{F}_{L}\right)$, whose structure maps $\left(d,[\![\cdot,\cdot]\!]\right)$ are defined by
		\begin{align*}
		&d:\Gamma\left(\wedge^{k}\left(T^{*}\mathcal{F}_{L}\times\mathbb{R}\right)\right)\rightarrow\Gamma\left(\wedge^{k+1}\left(T^{*}\mathcal{F}_{L}\times\mathbb{R}\right)\right): (\alpha,\beta)\mapsto\left(-d_{\mathcal{F}_{L}}\alpha,-d_{\mathcal{F}_{L}}\beta-\gamma\wedge\beta\right),\\
		&[\![\cdot,\cdot]\!]:\Gamma\left(\wedge^{k}\left(T^{*}\mathcal{F}_{L}\times\mathbb{R}\right)\right)\otimes\Gamma\left(\wedge^{l}\left(T^{*}\mathcal{F}_{L}\times\mathbb{R}\right)\right)\rightarrow\Gamma\left(\wedge^{k+l}\left(T^{*}\mathcal{F}_{L}\times\mathbb{R}\right)\right):\\
		&\hspace{3.3cm}(\alpha,\beta)\otimes(\delta,\epsilon)\mapsto\left(0,\pounds_{X}\alpha\wedge\epsilon-(-1)^{kl}\pounds_{X}\delta\wedge\beta\right).
		\end{align*}
	\end{cor}
	\begin{proof}
		We start by writing down explicitly the structure maps $\lambda_{1},\lambda_{2}$ of the $L_{\infty}[1]$-algebra $\big(\Gamma(\wedge^{\bullet}(T^{*}\mathcal{F}_{L}\times\mathbb{R}))[1],\lambda_{1},\lambda_{2}\big)$, as defined in \eqref{CaFe}. We then apply the d\'ecalage isomorphisms to obtain the associated DGLA  $\big(\Gamma(\wedge^{\bullet}(T^{*}\mathcal{F}_{L}\times\mathbb{R})),d,[\![\cdot,\cdot]\!]\big)$. In the computations below, we again identify elements of $\Gamma\big(\wedge^{\bullet}T^{*}\mathcal{F}_{L}\big)$ with vertical fiberwise constant multivector fields on $T^{*}\mathcal{F}_{L}$ via the isomorphism \eqref{cor}.

		Choose homogeneous elements $(\alpha,\beta)\in\Gamma\left(\wedge^{k}\left(T^{*}\mathcal{F}_{L}\times\mathbb{R}\right)\right)$ and $(\delta,\epsilon)\in\Gamma\left(\wedge^{l}\left(T^{*}\mathcal{F}_{L}\times\mathbb{R}\right)\right)$. We then have 
		\begin{align}\label{differential}
		\left[\widetilde{\Pi},\alpha+\beta\wedge\partial_{t}\right]&=\big[(V_{vert}+V_{lift})\wedge t\partial_{t}+\Pi_{can},\alpha+\beta\wedge\partial_{t}\big]\nonumber\\
		&=(-1)^{k-1}[V_{lift},\alpha]\wedge t\partial_{t}-V\wedge\beta\wedge\partial_{t}+[\Pi_{can},\alpha]+[\Pi_{can},\beta]\wedge\partial_{t},
		\end{align}
		which implies that
		\begin{equation}\label{L1diff}
		\lambda_{1}\big((\alpha,\beta)\big)=\left(-d_{\mathcal{F}_{L}}\alpha,-d_{\mathcal{F}_{L}}\beta-\gamma\wedge\beta\right).
		\end{equation}
		Next, using the computation \eqref{differential}, we have
		\begin{align*}
		&\left[\left[\widetilde{\Pi},\alpha+\beta\wedge\partial_{t}\right],\delta+\epsilon\wedge\partial_{t}\right]\\
		&\hspace{1.5cm}=\left[(-1)^{k-1}[V_{lift},\alpha]\wedge t\partial_{t}-V\wedge\beta\wedge\partial_{t}+[\Pi_{can},\alpha]+[\Pi_{can},\beta]\wedge\partial_{t},\delta+\epsilon\wedge\partial_{t}\right]\\
		&\hspace{1.5cm}=\left[(-1)^{k-1}[V_{lift},\alpha]\wedge t\partial_{t}-V_{lift}\wedge\beta\wedge\partial_{t},\delta+\epsilon\wedge\partial_{t}\right]\\
		&\hspace{1.5cm}=(-1)^{k}\left[V_{lift},\alpha\right]\wedge\epsilon\wedge\partial_{t}-(-1)^{k(l-1)}\left[V_{lift},\delta\right]\wedge\beta\wedge\partial_{t},
		\end{align*}
		which implies that
		\begin{equation}\label{L1bracket}
		\lambda_{2}\big((\alpha,\beta)\otimes(\delta,\epsilon)\big)=\left(0,(-1)^{k}\pounds_{X}\alpha\wedge\epsilon-(-1)^{k(l-1)}\pounds_{X}\delta\wedge\beta\right).
		\end{equation}
		The d\'ecalage isomorphisms 
		act as 
		\begin{align*}
		&(\alpha,\beta)\mapsto(\alpha,\beta)\\
		&(\alpha,\beta)\otimes(\delta,\epsilon)\mapsto(-1)^{k}(\alpha,\beta)\otimes(\delta,\epsilon),
		\end{align*}
		and applying them to \eqref{L1diff} and \eqref{L1bracket} yields the expressions stated in the corollary.
	\end{proof}
	
	In more detail, the fact that this DGLA governs the deformations of $L$ means the following. For convenience, we assume the neighborhood $U$ of $L$ in $T^{*}\mathcal{F}_{L}\times\mathbb{R}$ where $\widetilde{\Pi}$ is defined to be invariant under fiberwise multiplication by $-1$. Then for any section $(\alpha,f)\in\Gamma(T^{*}\mathcal{F}_{L}\times\mathbb{R})$ whose graph lies inside $U$, we have
	\begin{align*}
	\text{Graph}(\alpha,f)\ \text{is Lagrangian}&\Leftrightarrow \ \text{Graph}(\alpha,f)\ \text{is coisotropic}\\
	&\Leftrightarrow \ (-\alpha,-f)\ \text{is a Maurer-Cartan element of}\nonumber\\ &\hspace{0.8cm}\text{the}\ L_{\infty}[1]-\text{algebra}\  \big(\Gamma(T^{*}\mathcal{F}_{L}\times\mathbb{R})[1],\lambda_{1},\lambda_{2}\big)\\
	&\Leftrightarrow \ (\alpha,f)\ \text{is a Maurer-Cartan element of}\nonumber\\ &\hspace{0.8cm}\text{the}\ \text{DGLA}\  \big(\Gamma(T^{*}\mathcal{F}_{L}\times\mathbb{R}),d,[\![\cdot,\cdot]\!]\big)\\
	&\Leftrightarrow\begin{cases}
	d_{\mathcal{F}_{L}}\alpha=0\\
	d_{\mathcal{F}_{L}}f+f(\gamma-\pounds_{X}\alpha)=0
	\end{cases},
	\end{align*}
	where the first equivalence is Lemma \ref{lagcois} and the second one is Thm. \ref{entire}. So we recover the equations \eqref{eqns} that we derived in Thm. \ref{equations} by direct computation.

{We will see later that this DGLA also governs the moduli problem of $L$, defined by considering deformations of $L$ up to Hamiltonian isotopy. Indeed, in \S \ref{subsec:eqrig} we prove that the gauge equivalence relation of $\big(\Gamma(\wedge^{\bullet}(T^{*}\mathcal{F}_{L}\times\mathbb{R})),d,[\![\cdot,\cdot]\!]\big)$ agrees with the geometric notion of equivalence given by Hamiltonian isotopies.}	
	
\begin{remark}[Formality]\label{rem:formality}
	
We do not know whether the DGLA in Corollary \ref{DGLA}	is formal, i.e. $L_{\infty}$-quasi-isomorphic to its cohomology $H^{\bullet}(\mathcal{F}_{L})\oplus H^{\bullet-1}_{\gamma}(\mathcal{F}_{L})$ with the induced graded Lie algebra structure.
On one side, such a result would not be so surprising when $L$ is compact, because of the following.
Any graded Lie algebra $(H,[\cdot,\cdot])$ has the property that the Kuranishi map completely characterizes  unobstructedess: a first order deformation $A$ is  unobstructed if and only\footnote{This is immediate, since if $Kr(A)=[A,A]$ vanishes then $t\mapsto tA$ is a curve of Maurer-Cartan elements.}  if $Kr(A)=0$. 
When $L$ is compact, we know that the DGLA in Corollary \ref{DGLA} satisfies this property, as a consequence of Prop. \ref{prop:krunob}. Further we expect the property to be invariant under $L_{\infty}$-quasi-isomorphisms satisfying  mild assumptions.
We do not address the formality question any further here. A possible approach is to apply Manetti's formality criteria in Thm. 3.3 or Thm. 3.4 of \cite{ManettiFormalityDGLA}.
  
\end{remark}

	\begin{remark}\label{generalization}
		We comment on the structure of the DGLA $\big(\Gamma(T^{*}\mathcal{F}_{L}\times\mathbb{R}),d,[\![\cdot,\cdot]\!]\big)$ introduced in Corollary \ref{DGLA}.
		\begin{enumerate}[i)]
			\item One can write down this DGLA in more generality. Let $\big(A,\rho,[\cdot,\cdot]\big)$ be a Lie algebroid over a manifold $M$, and let $\nabla$ be a flat $A$-connection on a line bundle $E\rightarrow M$. Let $D\in\text{Der}(A)$ be a derivation of $A$. Then there is an induced DGLA-structure $(d,[\![\cdot,\cdot]\!])$  on the graded vector space $\Gamma\left(\wedge^{\bullet}(A^{*}\oplus E)\right)=\Gamma\left(\wedge^{\bullet}A^{*}\right)\oplus\Gamma\left(\wedge^{\bullet-1}A^{*}\otimes E\right)$ defined by
			\begin{align}
			&d(\alpha,\varphi)=\left(d_{A}\alpha,d_{\nabla}\varphi\right)\nonumber\\
			&[\![(\alpha,\varphi),(\beta,\psi)]\!]=\left(0,\pounds_{D}\alpha\wedge\psi -(-1)^{kl}\pounds_{D}\beta\wedge\varphi\right),\label{eq:general}
			\end{align}
			for homogeneous elements  $(\alpha,\varphi)\in\Gamma\left(\wedge^{k}\left(A^{*}\oplus E\right)\right)$ and $(\beta,\psi)\in\Gamma\left(\wedge^{l}\left(A^{*}\oplus E\right)\right)$. Here the Lie derivative $\pounds_{D}$ is obtained extending the derivation on $A^*$ dual to $D$. 
			
			\item We discuss the structure of the DGLA $\left(\Gamma\left(\wedge^{\bullet}(A^{*}\oplus E)\right),d,[\![\cdot,\cdot]\!]\right)$. The underlying cochain complex is a direct sum of complexes $\left(\Gamma\left(\wedge^{\bullet}A^{*}\right),d_{A}\right)\oplus\left(\Gamma\left(\wedge^{\bullet}A^{*}\otimes E\right)[-1],d_{\nabla}\right)$. It can also be described as the cochain complex of differential forms on the Lie algebroid $A\oplus E^*$, the semidirect product  of $A$ by the representation on $E^{*}$ given by the dual connection $\nabla^*$.
			The underlying graded Lie algebra structure is the semidirect product of the abelian graded Lie algebras $\Gamma\left(\wedge^{\bullet}A^{*}\right)$ and $\Gamma\left(\wedge^{\bullet}A^{*}\otimes E\right)[-1]$ with respect to the action
			\[
			\Gamma\left(\wedge^{\bullet}A^{*}\right)\rightarrow\text{Der}\big(\Gamma\left(\wedge^{\bullet}A^{*}\otimes E\right)[-1]\big):\alpha\mapsto \pounds_{D}\alpha\wedge\bullet.
			\]
			
			\item We can recover the DGLA   $\big(\Gamma(T^{*}\mathcal{F}_{L}\times\mathbb{R}),d,[\![\cdot,\cdot]\!]\big)$ described in Corollary \ref{DGLA} by making the following choices in the general construction of i) above:
			\begin{itemize}
				\item Take the Lie algebroid $A:=\left(T\mathcal{F}_{L},-\iota,-[\cdot,\cdot]\right)$, where $\iota:T\mathcal{F}_{L}\hookrightarrow TL$ is the inclusion and $[\cdot,\cdot]$ is the Lie bracket of vector fields. The Lie algebroid differential $d_{A}$ on $\Gamma\left(\wedge^{\bullet}A^{*}\right)$ is then $-d_{\mathcal{F}_{L}}$.
				\item Let $D:=[X,\cdot]$ be the derivation determined by $X\in\mathfrak{X}(L)^{\mathcal{F}_{L}}$.
				\item Let $E:=L\times\mathbb{R}\rightarrow L$ be the trivial line bundle.
				\item Let the representation $\nabla$ of $A$ on $E$ be defined by
				\[
				\nabla_{Y}\bullet = \pounds_{-Y}\bullet -\gamma(Y)\bullet
				\]
				for $Y\in\Gamma(A)$. Since $\gamma$ is closed, this is indeed a representation, and the induced differential $d_{\nabla}$ on $\Gamma\left(\wedge^{\bullet}A^{*}\right)$ is given by
				\[
				d_{\nabla}\bullet=-d_{\mathcal{F}_{L}}\bullet - \gamma\wedge\bullet.
				\]
			\end{itemize}
		\end{enumerate}
	\end{remark}
	
	\subsection{On foliated Morse-Novikov cohomology}\label{subsec:Nov}
	\leavevmode
	\vspace{0.1cm}
	
	This subsection discusses the cohomology of the DGLA $\big(\Gamma(\wedge^{\bullet}(T^{*}\mathcal{F}_{L}\times\mathbb{R})),d,[\![\cdot,\cdot]\!]\big)$ in degree one, which in the notation of Remark \ref{rem} is given by $H^{1}(\mathcal{F}_{L})\oplus H^{0}_{\gamma}(\mathcal{F}_{L})$. We explicitly compute the second summand of this cohomology group for Lagrangians that are compact and connected. 
	We first collect some foliated analogs of well-known facts about Morse-Novikov cohomology \cite[Section 1]{HR}. 
	
	\begin{lemma}\label{cohomology}
		Let $L$ be a manifold, $\mathcal{F}_{L}$ a foliation on $L$ and $\eta\in\Omega^{1}(\mathcal{F}_{L})$ a closed foliated one-form. As before, denote by $H_{\eta}^{\bullet}(\mathcal{F}_{L})$ the cohomology groups of the differential $d_{\mathcal{F}_{L}}^{\eta}$ defined in \eqref{m-n}. We then have the following:
		\begin{enumerate}[i)]
			\item If $[\eta]=[\eta']\in H^{1}(\mathcal{F}_{L})$, then $H^{k}_{\eta}(\mathcal{F}_{L})\cong H^{k}_{\eta'}(\mathcal{F}_{L})$. In particular, if $[\eta]=0$ in $H^{1}(\mathcal{F}_{L})$ then $H^{k}_{\eta}(\mathcal{F}_{L})\cong H^{k}(\mathcal{F}_{L})$.
			\item Assume $[\eta]\neq 0$ in $H^{1}(\mathcal{F}_{L})$ and let $f\in H^{0}_{\eta}(\mathcal{F}_{L})$. Then there is a leaf $\mathcal{O}$ of $\mathcal{F}_{L}$ on which $f$ vanishes identically.
		\end{enumerate}
	\end{lemma}
	\begin{proof}
		\begin{enumerate}[i)]
			\item If $\eta'=\eta+d_{\mathcal{F}_{L}}g$ for $g\in C^{\infty}(L)$, then the following map is an isomorphism of cochain complexes:
			\begin{equation}\label{isocompl}
			\big(\Omega^{\bullet}(\mathcal{F}_{L}),d^{\eta'}_{\mathcal{F}_{L}}\big)\rightarrow \big(\Omega^{\bullet}(\mathcal{F}_{L}),d^{\eta}_{\mathcal{F}_{L}}\big):\beta\mapsto e^{g}\beta.
			\end{equation}
			\item By assumption we have that 
			\begin{equation}\label{eq}
			d_{\mathcal{F}_{L}}f+f\eta=0.
			\end{equation}
			If $f$ would be nowhere zero, then we could write $\eta=-d_{\mathcal{F}_{L}}log|f|$, contradicting that $\eta$ is not exact. So $f$ must have a zero, say in the leaf $\mathcal{O}\in\mathcal{F}_{L}$. Consider the vanishing set $\mathcal{Z}_{f}:=\{x\in\mathcal{O}:f(x)=0\}$, which is nonempty and closed in $\mathcal{O}$. If we show that $\mathcal{Z}_{f}$ is also open in $\mathcal{O}$, then we reach the conclusion $f|_{\mathcal{O}}\equiv0$, since $\mathcal{O}$ is connected.
			
			To this end, let $x\in \mathcal{Z}_{f}$. Since $\eta|_{\mathcal{O}}\in\Omega^{1}(\mathcal{O})$ is closed, there exist a neighborhood $U$ of $x$ in $\mathcal{O}$ and $g\in C^{\infty}(U)$ such that $\eta|_{U}=dg$. Using the isomorphism \eqref{isocompl} for the one-leaf foliation on $U$, we obtain that $d(e^{g}f|_{U})=0$. So $e^{g}f|_{U}$ is constant on $U$, and since $f(x)=0$ we must have $e^{g}f|_{U}\equiv0$. Consequently $f|_{U}\equiv 0$, which shows that $U\subset \mathcal{Z}_{f}$. So $\mathcal{Z}_{f}$ is open, and this finishes the proof.
		\end{enumerate}
	\end{proof}
	
	\begin{remark}
		If we replace the hypothesis $[\eta]\neq 0$ in $ii)$ of Lemma \ref{cohomology} by the stronger requirement that $\eta|_{\mathcal{O}}\in\Omega^{1}(\mathcal{O})$ be not exact for all leaves $\mathcal{O}\in\mathcal{F}_{L}$, then, restricting the equality \eqref{eq} to each leaf $\mathcal{O}$, the above proof shows that $H^{0}_{\eta}(\mathcal{F}_{L})=0$.
	\end{remark}

 {
\begin{remark}\label{rem:wedge}
Let $(L,\mathcal{F}_{L})$ be a foliated manifold, and $\eta,\delta \in\Omega^{1}(\mathcal{F}_{L})$  closed foliated one-forms.
Then the wedge product induces a  bilinear map
$H^{\bullet}_{\eta}(\mathcal{F}_{L})\times  H^{\bullet}_{\delta}(\mathcal{F}_{L})\to H^{\bullet}_{\eta+\delta}(\mathcal{F}_{L})$. This can be proven like the corresponding statement for   manifolds \cite[\S 1]{HR}.
\end{remark}
}

	We now specialize to compact, connected manifolds $L$ endowed with a codimension-one foliation $\mathcal{F}_{L}$ defined by a nowhere vanishing closed one-form. 
	Under these assumptions it is well-known  \cite[Theorem 9.3.13]{conlon} that:
	
	\begin{itemize}
		\item either $(L,\mathcal{F}_{L})$ is the fiber foliation of a fiber bundle $p:L\rightarrow S^{1}$, 
		\item or all leaves of $\mathcal{F}_{L}$ are dense.\hspace{9cm}($\star$)
	\end{itemize}
	
	Recall moreover that in the former case, the $k$-th cohomology groups of the fibers of $p:L\rightarrow S^{1}$ constitute a vector bundle $\mathcal{H}^{k}$ over $S^{1}$:
	\[
	\mathcal{H}^{k}_{q}=H^{k}\left(p^{-1}(q)\right),
	\] 
	and one has 
	\begin{equation}\label{isom}
	H^{k}(\mathcal{F}_{L})\overset{\sim}{\rightarrow}\Gamma\left(\mathcal{H}^{k}\right):[\alpha]\mapsto \left(\sigma_{\alpha}:q\mapsto\left[\left.\alpha\right|_{p^{-1}(q)}\right]\right).
	\end{equation}
	Using the identification 
	\[
	\mathfrak{X}(S^{1})\overset{\sim}{\longrightarrow}\frac{\mathfrak{X}(L)^{\mathcal{F}_{L}}}{\Gamma(T\mathcal{F}_{L})}:Y\mapsto\overline{Y},
	\]	
	one can define a natural flat connection $\nabla$
	on the vector bundle $\mathcal{H}^{k}$ by the formula
	\begin{equation}\label{connection}
	\nabla_{Y}\sigma_{\alpha}:=\sigma_{\pounds_{\overline{Y}}\alpha},
	\end{equation}
	for $\alpha\in\Omega^{1}_{cl}(\mathcal{F}_{L})$ and $Y\in\mathfrak{X}(S^{1})$. Note that $\nabla$ is well-defined, because of Cartan's formula. If $F$ denotes the typical fiber of $p:L\rightarrow S^{1}$ and $\{[\beta_{1}],\ldots,[\beta_{m}]\}$ is a basis of $H^{k}(F)$, then in a local trivialization $U\times F$, the constant functions $[\beta_{1}],\ldots,[\beta_{m}]\in C^{\infty}(U,H^{k}(F))\cong\Gamma(\mathcal{H}^{k}|_{U})$ constitute a local frame of flat sections.
	
	To compute the foliated Morse-Novikov cohomology, we will need the following lemma.

	\begin{lemma}\label{function}
		Let $L$ be a compact manifold endowed with a foliation $\mathcal{F}_{L}$ that is the fiber foliation of a fiber bundle\footnote{Note that under these assumptions, $L$ is automatically connected.} $p:L\rightarrow S^{1}$. Let $\eta\in\Omega^{1}(\mathcal{F}_{L})$ be a closed foliated one-form, denote by $\sigma_{\eta}\in\Gamma(\mathcal{H}^{1})$ the section corresponding with $[\eta]\in H^{1}(\mathcal{F}_{L})$ under \eqref{isom}, and let $\mathcal{Z}_{\eta}:=\sigma_{\eta}^{-1}(0)$. Then there exists a smooth function $g\in C^{\infty}(L)$ such that
		\[
		\left.\eta\right|_{p^{-1}(q)}=d\left(\left.g\right|_{p^{-1}(q)}\right)\hspace{1cm}\text{for all}\ q\in\mathcal{Z}_{\eta}.
		\]
	\end{lemma}
	\begin{proof}
		By \cite[Lemma 2.28]{moerdijk}, we can fix an embedded loop $\tau:S^{1}\rightarrow L$ transverse to the leaves of $\mathcal{F}_{L}$ which hits each leaf of $\mathcal{F}_{L}$ exactly once. Define a function $h$ on $p^{-1}(\mathcal{Z}_{\eta})$ by setting $h|_{p^{-1}(q)}$ to be the unique primitive of $\eta|_{p^{-1}(q)}$ that vanishes at the point $p^{-1}(q)\cap\tau(S^{1})$. We claim that $h$ extends to a smooth function $g\in C^{\infty}(L)$. To prove this, it suffices to show that around each point $x\in p^{-1}(\mathcal{Z}_{\eta})$ there exist a neighborhood $U\subset L$ and a smooth function on $U$ that agrees with $h$ on $U\cap p^{-1}(\mathcal{Z}_{\eta})$.
		
		Let $x\in p^{-1}(q)$ for $q\in\mathcal{Z}_{\eta}$ and denote $y:=p^{-1}(q)\cap\tau(S^{1})$. 
		Working in a local trivialization $V\times p^{-1}(q)$, choose a path $\gamma:(-\epsilon,1+\epsilon)\rightarrow p^{-1}(q)$ such that $\gamma(0)=x$ and $\gamma(1)=y$, take a tubular neighborhood $N$ of this path in $p^{-1}(q)$ and define $U:=V\times N$. Since $N$ is contractible, we have for each value of $v\in V$ that $\eta_{v}\in\Omega^{1}(N)$ is exact. Since one can choose primitives varying smoothly in $v$ (see \cite{primitives}), it follows that $\eta|_{U}$ is foliated exact. 
		Choose any primitive $k\in C^{\infty}(U)$ of $\eta|_{U}$. Shrinking $V$ if necessary, we can assume that each fiber $\{v\}\times N$ intersects the loop $\tau(S^{1})$. Define a map $\phi:U\rightarrow U\cap\tau(S^{1})$ by setting $\phi(z)$ to be the intersection point of $\tau(S^{1})$ with the fiber through $z$. Then setting $\widetilde{h}:=k-\phi^{*}\big(k|_{U\cap\tau(S^{1})}\big)$, we obtain a primitive of $\eta|_{U}$ that vanishes along $U\cap\tau(S^{1})$.

		Uniqueness of such primitives implies that $\widetilde{h}$ agrees with $h$ wherever both of them are defined. This shows that $h$ can be extended to a smooth function $g\in C^{\infty}(L)$.
	\end{proof}
	
	We can now compute the zeroth foliated Morse-Novikov cohomology group.

	\begin{thm}\label{H}
		Let $(L,\mathcal{F}_{L})$ be a compact, connected manifold with codimension-one foliation defined by a closed one-form. Let $\eta\in\Omega^{1}(\mathcal{F}_{L})$ be a closed foliated one-form.
		\begin{enumerate}[i)]
			\item Assume $\mathcal{F}_{L}$ is the fiber foliation of a fiber bundle $p:L\rightarrow S^{1}$. Then we have 
			\[
			H_{\eta}^{0}(\mathcal{F}_{L})\cong \{f\in C^{\infty}(S^{1}):f\cdot\sigma_{\eta}=0\},
			\]
			where $\sigma_{\eta}\in\Gamma(\mathcal{H}^{1})$ denotes the section corresponding with $[\eta]\in H^{1}(\mathcal{F}_{L})$ under \eqref{isom}.
			\item Assume all leaves of $\mathcal{F}_{L}$ are dense. Then 
			\[
			H_{\eta}^{0}(\mathcal{F}_{L})=\begin{cases}\mathbb{R}\hspace{1cm} \text{if}\  \eta\  \text{is foliated exact}\\ 0\hspace{1.1cm} \text{otherwise}
			\end{cases}.
			\]
		\end{enumerate}
	\end{thm}
	\begin{proof}
		\begin{enumerate}[i)]
			\item Fix a smooth function $g\in C^{\infty}(L)$ as constructed in Lemma \ref{function} and define
			\[
			\Psi: H_{\eta}^{0}(\mathcal{F}_{L})\rightarrow\{f\in C^{\infty}(S^{1}):f\cdot\sigma_{\eta}=0\}:h\mapsto e^{g}h.
			\]
			We first check that $\Psi$ is well-defined. Choosing $h\in H_{\eta}^{0}(\mathcal{F}_{L})$, we must show that $e^{g}h$ is constant along the leaves of $\mathcal{F}_{L}$, and that the induced function on the leaf space $S^{1}$ lies in the annihilator ideal of $\sigma_{\eta}\in\Gamma(\mathcal{H}^{1})$. Note that for any $q\in S^{1}$, we have
			\[
			d\left(\left.h\right|_{p^{-1}(q)}\right)+\left.h\right|_{p^{-1}(q)}\left.\eta\right|_{p^{-1}(q)}=0.
			\]
			In case $\sigma_{\eta}(q)=0$, then $\eta|_{p^{-1}(q)}=d(g|_{p^{-1}(q)})$ and the isomorphism \eqref{isocompl} implies that
			\[
			\left.(e^{g}h)\right|_{p^{-1}(q)}\in H^{0}(p^{-1}(q))=\mathbb{R}.
			\]
			Next, assume that $\sigma_{\eta}(q)\neq 0$, i.e. $\eta|_{p^{-1}(q)}$ is not exact. Then $h|_{p^{-1}(q)}\equiv 0$ by applying $ii)$ of Lemma \ref{cohomology} to the one-leaf foliation on $p^{-1}(q)$, and therefore $(e^{g}h)|_{p^{-1}(q)}\equiv 0$. 
			
			Clearly, the map $\Psi$ is linear and injective. For surjectivity, we let $f\in C^{\infty}(S^{1})$ be such that $f\cdot\sigma_{\eta}=0$ and we have to check that $e^{-g}p^{*}f\in H_{\eta}^{0}(\mathcal{F}_{L})$, i.e.
			\begin{equation}\label{tocheck}
			d_{\mathcal{F}_{L}}\left(\frac{p^{*}f}{e^{g}}\right)+\frac{p^{*}f}{e^{g}}\eta=0.
			\end{equation}
			On fibers $p^{-1}(q)$ with $\sigma_{\eta}(q)\neq 0$, the equality \eqref{tocheck} is satisfied since $p^{*}f$ vanishes there. On fibers $p^{-1}(q)$ with $\sigma_{\eta}(q)=0$, we have $\eta|_{p^{-1}(q)}=d(g|_{p^{-1}(q)})$, so that the left hand side of \eqref{tocheck} becomes
			\begin{align*}
			-f(q)(e^{-g}|_{p^{-1}(q)})d(g|_{p^{-1}(q)})+f(q)(e^{-g}|_{p^{-1}(q)})d(g|_{p^{-1}(q)})=0.
			\end{align*}
			
			\item This is an immediate consequence of Lemma \ref{cohomology}.
		\end{enumerate}
	\end{proof}
	
	\begin{ex}\label{coh}
		Take $L=(S^{1}\times S^{1},\theta_{1},\theta_{2})$ and let $\mathcal{F}_{L}$ be the foliation by fibers of the projection $(S^{1}\times S^{1},\theta_{1},\theta_{2})\rightarrow (S^{1},\theta_{1})$. To compute $H^{0}_{\eta}(\mathcal{F}_{L})$ for closed $\eta\in\Omega^{1}(\mathcal{F}_{L})$, we can choose a convenient representative of $[\eta]\in H^{1}(\mathcal{F}_{L})$, by $i)$ of Lemma \ref{cohomology}.  
		In this respect, notice that every class $[g(\theta_{1},\theta_{2})d\theta_{2}]\in H^{1}(\mathcal{F}_{L})$ has a unique representative of the form $h(\theta_{1})d\theta_{2}$. Namely, setting $h(\theta_{1}):=\frac{1}{2\pi}\int_{S^{1}}g(\theta_{1},\theta_{2})d\theta_{2}$, we have
		\[
		\int_{ S^{1}}\left[g(\theta_{1},\theta_{2})-h(\theta_{1})\right]d\theta_{2}=0,
		\]
		which implies that there exists $k(\theta_{1},\theta_{2})\in C^{\infty}(S^{1}\times S^{1})$ such that
		\[
		g(\theta_{1},\theta_{2})-h(\theta_{1})=\frac{\partial k}{\partial \theta_{2}}(\theta_{1},\theta_{2}).
		\]
		This implies that
		\[
		g(\theta_{1},\theta_{2})d\theta_{2}-h(\theta_{1})d\theta_{2}=\frac{\partial k}{\partial \theta_{2}}(\theta_{1},\theta_{2})d\theta_{2}=d_{\mathcal{F}_{L}}k.
		\]
		Uniqueness of such representatives follows by integrating around circles $\{\theta_{1}\}\times S^{1}$. Now, fix $\eta=h(\theta_{1})d\theta_{2}$ in $\Omega^{1}(\mathcal{F}_{L})$ and assume that $f\in H_{\eta}^{0}(\mathcal{F}_{L})$. Then
		\begin{equation}\label{zero}
		0=\frac{\partial f}{\partial\theta_{2}}d\theta_{2}+f\cdot h(\theta_{1})d\theta_{2}.
		\end{equation}
		For fixed $\theta_{1}$, the restriction of $f$ to $\{\theta_{1}\}\times S^{1}$ reaches a maximum $M$ and a minimum $m$. The equality \eqref{zero} implies that
		\[
		\begin{cases}
		M\cdot h(\theta_{1})=0\\
		m\cdot h(\theta_{1})=0
		\end{cases}.
		\]
		So either $h(\theta_{1})=0$ or $\left.f\right|_{\{\theta_{1}\}\times S^{1}}\equiv 0$. Hence, we get that $f\cdot h(\theta_{1})=0$, and \eqref{zero} then implies that also $\partial f/\partial\theta_{2}=0$. In conclusion, we get
		\begin{align*}
		H_{h(\theta_{1})d\theta_{2}}^{0}(\mathcal{F}_{L})&=\{f(\theta_{1}):f(\theta_{1})h(\theta_{1})d\theta_{2}=0\}\\
		&=\{f(\theta_{1}):f(\theta_{1})\cdot\sigma_{h(\theta_{1})d\theta_{2}}=0\},
		\end{align*}
		using in the last equality that $\sigma_{h(\theta_{1})d\theta_{2}}(\theta_{1})=0\Leftrightarrow h(\theta_{1})d\theta_{2}=0$. So we obtain the result that was predicted by $i$) of {Theorem} \ref{H}.
	\end{ex}
	
	\begin{remark}
		The example we have in mind throughout this subsection is of course that of a compact connected Lagrangian submanifold $L^{n}$ contained in the singular locus $Z$ of a log-symplectic manifold $(M^{2n},Z,\Pi)$. The induced foliation $\mathcal{F}_{L}$ on $L$ is defined by a nowhere vanishing closed one-form, which is obtained by pulling back a closed defining one-form for the foliation on $Z$. So $(L,\mathcal{F}_{L})$ is either the fiber foliation of a fiber bundle $L\rightarrow S^{1}$, or all leaves of $\mathcal{F}_{L}$ are dense in $L$.
		
		Moreover, the foliation type of $\mathcal{F}_{L}$ is stable under small deformations of the Lagrangian $L$ inside $Z$. To see this, we can work in the local model $p:T^{*}\mathcal{F}_{L}\rightarrow L$, where the total space $T^{*}\mathcal{F}_{L}$ is endowed with the pullback foliation $p^{-1}(\mathcal{F}_{L})$. Any Lagrangian deformation $L'$ of $L$ is of the form $L'=\text{Graph}(\alpha)$ for some $\alpha\in\Omega^{1}_{cl}(\mathcal{F}_{L})$, and the induced foliation $\mathcal{F}_{L'}$ is obtained by intersecting $L'$ with the leaves of $p^{-1}(\mathcal{F}_{L})$. Therefore, the map $p:(L',\mathcal{F}_{L'})\rightarrow(L,\mathcal{F}_{L})$ is a foliated diffeomorphism (with inverse $\alpha:(L,\mathcal{F}_{L})\rightarrow(L',\mathcal{F}_{L'})$), 
		which shows that $(L,\mathcal{F}_{L})$ and $(L',\mathcal{F}_{L'})$ are of the same type.
	\end{remark}

	\section{Deformations of Lagrangian submanifolds in log-symplectic manifolds: geometric aspects}\label{sec:geom}
	
	We  present some geometric consequences of the algebraic results obtained in the previous section. 
	We address three different geometric questions, each in a separate  subsection, as we now outline. Throughout, we assume the set-up given at the beginning of \S\ref{sec:alg}.

\begin{description}
\item[\S\ref{subs} Deformations constrained to the singular locus]
We investigate when all sufficiently small deformations of the Lagrangian $L$ are constrained to the singular locus. Prop. \ref{prop:Vtangent} gives a condition under which this does not happen.
On the opposite extreme, in   Cor. \ref{constrained} and Prop. \ref{denserestr} we obtain positive results assuming that $L$ is compact, by considering separately the case that $L$ is the total space of a fibration and the case that $L$ has a dense leaf. The latter case is subtle, and we show that the conclusion of Prop. \ref{denserestr} fails to hold if we remove a certain finite dimensionality assumption.
\item[\S\ref{subsec:obstr} Obstructedness of deformations] We ask when infinitesimal deformations of the Lagrangian $L$ can be extended to a smooth curve of Lagrangian deformations.
A sufficient criterium is given in Prop. \ref{unobstructed}. 
(All smoothly unobstructed deformations arise this way under an additional assumption, see Lemma \ref{lem:Kurexact}). Our main results here, under the assumption that $L$ is compact, are the computable ``if and only if'' criteria of Prop. \ref{prop:krunob} and Cor. \ref{cor:Krsimple}.
\item[\S\ref{subsec:eqrig} Equivalences and rigidity of deformations] On the set of Lagrangian deformations of $L$ there are two natural notions of equivalence: an algebraic one
and a geometric one, given by 
Hamiltonian isotopies. In Prop. \ref{prop:ham} we show that they coincide. We also show that there are no Lagrangian submanifolds which are infinitesimally rigid under Hamiltonian isotopies, so that the moduli space (which typically is not smooth) does not have any isolated points.
This leads us to consider the more flexible equivalence relation given by Poisson isotopies.
The formal tangent space of its moduli space is computed in
Prop. \ref{prop:poisequiv}. There do exist Lagrangians which are rigid under Poisson isotopies, as follows using
Prop. \ref{rigidity}.
\end{description}

	\begin{remark}[The local deformation problem]
	We summarize here how  our results specialize to the local deformation problem,
	i.e. to a Lagrangian $L$ as in the local model of Prop. \ref{coordinates}:
	\begin{itemize}
		\item  $L$ can be deformed smoothly to a Lagrangian submanifold outside of the singular locus (Remark \ref{rem:locpush}).
		\item  all first order deformations of $L$ are smoothly unobstructed (Cor. \ref{H1}).
		\item  The space of local Lagrangian deformations modulo Hamiltonian isotopies is not smooth at $[L]$. Indeed, the formal tangent space at $[L]$ is isomorphic to $C^{\infty}(\RR)$ (see eq. \eqref{formalmoduli}), while at   Lagrangians contained in $M\setminus Z$ it is the zero vector space.  The same is true  if one replaces Hamiltonian isotopies by Poisson isotopies.
	\end{itemize}
\end{remark}

	\subsection{Deformations constrained to the singular locus}\label{subs}
	\leavevmode
	\vspace{0.1cm}
	
	We now investigate whether it is always possible to find deformations of the Lagrangian $L$ that escape from the singular locus. {Working in the model $\big(U\subset T^{*}\mathcal{F}_{L}\times\RR,V\wedge t\partial_{t}+\Pi_{can}\big)$, 
	a sufficient condition is the existence of a representative of the fixed first Poisson cohomology class $[V]$ that is tangent to $L$. Below, we denote by $W:=U\cap\{t=0\}\subset T^{*}\mathcal{F}_{L}$ the neighborhood of $L$ in $T^{*}\mathcal{F}_{L}$ where $V$ is defined}.
	
	\begin{prop}\label{prop:Vtangent}
		The Poisson cohomology class $[V]\in H^{1}_{\Pi_{can}}({W})$ has a representative tangent to $L$ if and only if $[\gamma]=0\in H^1(\mathcal{F}_{L})$. If these equivalent conditions hold, then there is a smooth path of  Lagrangian deformations $L_s$ starting at $L_0=L$ which is not contained in the singular locus  for  $s>0$.
	\end{prop}
	\begin{proof}
		We start by showing that the conditions are equivalent. First assume that $V-X_{g}$ for $g\in C^{\infty}({W})$ is a representative of $[V]$ that is tangent to $L$. As before, let $P$ denote the map that restricts vector fields on ${W}$ to $L$ and then takes their vertical component. By assumption, we then have
		\begin{align*}
		0=P(V-X_{g})&=P(V_{vert}+V_{lift}-X_{g})\\
		&=\gamma-P(X_{g}-X_{p^{*}(i^{*}g)}+X_{p^{*}(i^{*}g)})\\
		&=\gamma-P(X_{p^{*}(i^{*}g)})\\
		&=\gamma-d_{\mathcal{F}_{L}}(i^{*}g).
		\end{align*}
		Here $p:{W}\rightarrow L$ and $i:L\hookrightarrow {W}$ denote the projection and inclusion, respectively, the  fourth equality holds since $L$ is coisotropic,  and the last equality holds by the correspondence \eqref{cor}. This shows that $\gamma=d_{\mathcal{F}_{L}}(i^{*}g)$, and therefore $[\gamma]=0\in H^{1}(\mathcal{F}_{L})$.
		Conversely, if $\gamma=d_{\mathcal{F}_{L}}g$ for some $g\in C^{\infty}(L)$, then $V-X_{p^{*}g}$ is a representative of $[V]$ that is tangent to $L$.
		
		If the equivalent conditions hold, then by Remark \ref{rem:freedom} we can assume that $\gamma=0$. The Maurer-Cartan equation \eqref{eqns} then shows that any path of the form $s\mapsto(0,sf)$ for a nonzero leafwise constant function $f\in C^{\infty}(L)$ consists of Lagrangian deformations of $L$ that are no longer contained in the singular locus for $s>0$. Alternatively, if $\gamma=d_{\mathcal{F}_{L}}g$ for some $g\in C^{\infty}(L)$, then for any nonzero leafwise constant function $f$ on $L$, the path $s\mapsto(0,sfe^{-g})$ meets the criteria.
	\end{proof}
	
	\begin{remark}\label{rem:locpush}
		For the local deformation problem  we have $\gamma=0$, see Prop. \ref{coordinates}. Hence locally every half-dimensional Lagrangian submanifold contained in the singular locus can be deformed smoothly to one outside of the singular locus, by Prop. \ref{prop:Vtangent}.
	\end{remark}

	We will single out some Lagrangians whose deformations are constrained to the singular locus. We restrict ourselves to Lagrangians $L$ that are compact and connected. Recalling the dichotomy ($\star$) from \S\ref{subsec:Nov}, these assumptions imply that 
	either $(L,\mathcal{F}_{L})$ is the fiber foliation of a fiber bundle $L\rightarrow S^{1}$ or the leaves of $\mathcal{F}_{L}$ are dense.
	
	\subsubsection{\underline{The fibration case}}
	
	We need the following lemma about the map which, under the identification  \eqref{isom}, assigns to a closed foliated one-form  its cohomology class.
	
	\begin{lemma}\label{continuous}
		Let $(L,\mathcal{F}_{L})$ be a compact {foliated} manifold, where $\mathcal{F}_{L}$ is the fiber foliation of a fiber bundle $p:L\rightarrow S^{1}$. Then the following map is continuous for the $\mathcal{C}^{0}$-topology:
		\begin{equation}\label{cont}
		\Omega^{1}_{cl}(\mathcal{F}_{L})\rightarrow\Gamma(\mathcal{H}^{1}):\alpha\mapsto\sigma_{\alpha}.
		\end{equation}
	\end{lemma}
	\begin{proof}
		Let $F$ denote the typical fiber of $p:L\rightarrow S^{1}$. {Since $F$ is compact, the homology group $H_{1}(F;\mathbb{Z})$ is finitely generated, hence it is the direct sum of a free abelian group and a torsion subgroup. Let $[\gamma_{1}],\ldots,[\gamma_{n}]$ be independent generators of the free part of $H_{1}(F;\mathbb{Z})$, where the representatives $\gamma_{i}$ are smooth $1$-cycles. Consider the de Rham isomorphism
		\[
		H^{1}_{dR}(F)\rightarrow\text{Hom}\left(H_{1}(F;\mathbb{Z}),\mathbb{R}\right):[\omega]\mapsto\left([\gamma_{i}]\mapsto\int_{\gamma_{i}}\omega\right),
		\]
		where we use implicitly that a group homomorphism from $H_{1}(F;\mathbb{Z})$ to $\mathbb{R}$ is determined by the images of the generators of the free part. Let $\{[\alpha_{1}],\ldots,[\alpha_{n}]\}$ be the basis of $H^{1}_{dR}(F)$ that corresponds under this isomorphism with the basis of $\text{Hom}\left(H_{1}(F;\mathbb{Z}),\mathbb{R}\right)$ induced by $[\gamma_1],\ldots,[\gamma_n]$, so that}
		\[
		\int_{\gamma_{i}}\alpha_{j}=\delta_{ij}.
		\] 
		This provides an isomorphism $H^{1}_{dR}(F)\cong \RR^n$.
		Choose local trivialisations of $p:L\rightarrow S^{1}$ over open subsets $U_{1},\ldots,U_{r}$ covering $S^{1}$, and let $V_{1},\ldots,V_{r}$ be open subsets whose compact closures satisfy $\overline{V_{i}}\subset U_{i}$, such that $V_{1},\ldots,V_{r}$ still cover $S^{1}$. Then locally the map \eqref{cont} reads
		\begin{equation*}
		\left.\Omega_{cl}^{1}\left(\mathcal{F}\right)\right|_{p^{-1}(U_{i})}\rightarrow\Gamma\left(U_{i}\times H^{1}_{dR}(F)\right)\cong C^{\infty}(U_{i},\mathbb{R})^{n}:\alpha_{\theta}\mapsto\left(\int_{\gamma_{1}}\alpha_{\theta},\ldots,\int_{\gamma_{n}}\alpha_{\theta}\right).
		\end{equation*}
		Therefore the $\mathcal{C}^{0}$-norm of $\sigma_{\alpha}$ is
		\[
		\sum_{1\leq i\leq r}\sum_{1\leq j\leq n}\sup_{\theta\in \overline{V_{i}}}\left|\int_{\gamma_{j}}\alpha_{\theta}\right|,
		\]
		which can be made arbitrarily small by shrinking $\alpha$ in $\mathcal{C}^{0}$. Since the map \eqref{cont} is linear, this proves the lemma. \end{proof}
	
	The following corollary states that, under hypotheses that are antipodal to those of Prop. \ref{prop:Vtangent}, small deformations of $L$ stay inside the singular locus $Z$.
	
	\begin{cor}\label{constrained}
		Let $L^{n}$ be a compact connected Lagrangian submanifold contained in the singular locus $Z$ of an orientable log-symplectic manifold $(M^{2n},Z,\Pi)$. Assume that the induced foliation $\mathcal{F}_{L}$ on $L$ is the fiber foliation of a fiber bundle $p:L\rightarrow S^{1}$, and that the section $\sigma_{\gamma}\in\Gamma\left(\mathcal{H}^{1}\right)$ is nowhere zero. Then $\mathcal{C}^{1}$-small deformations of $L$ stay inside $Z$.
	\end{cor}
	\begin{proof}
		Clearly, we have a continuous map
		\[
		\left(\Omega^{1}_{cl}(\mathcal{F}_{L}),\mathcal{C}^{1}\right)\rightarrow\left(\Omega^{1}_{cl}(\mathcal{F}_{L}),\mathcal{C}^{0}\right):\alpha\mapsto\gamma-\pounds_{X}\alpha,
		\]
		so composing with the map \eqref{cont} gives a continuous map
		\[
		\left(\Omega^{1}_{cl}(\mathcal{F}_{L}),\mathcal{C}^{1}\right)\rightarrow\left(\Gamma(\mathcal{H}^{1}),\mathcal{C}^{0}\right):\alpha\mapsto\sigma_{\gamma-\pounds_{X}\alpha}.
		\]
		Therefore, since $\sigma_{\gamma}\in\Gamma(\mathcal{H}^{1})$ is nowhere zero, the same holds for $\sigma_{\gamma-\pounds_{X}\alpha}\in\Gamma(\mathcal{H}^{1})$ provided that $\alpha\in\Omega^{1}_{cl}(\mathcal{F}_{L})$ is sufficiently $\mathcal{C}^{1}$-small. By {Theorem} \ref{H}, this means that the cohomology group $H_{\gamma-\pounds_{X}\alpha}^{0}(\mathcal{F}_{L})$ is zero for $\mathcal{C}^{1}$-small $\alpha$. Looking at the Maurer-Cartan equation \eqref{eqns}, this implies that $\mathcal{C}^{1}$-small deformations of $L$ necessarily lie inside the singular locus.
	\end{proof}

	\begin{remark}
		The assumption in Corollary \ref{constrained} cannot be weakened. Clearly, if the interior of $\sigma_{\gamma}^{-1}(0)$ is nonempty, then by {Thm.} \ref{H} i) there exists $f\in H_{\gamma}^{0}(\mathcal{F}_{L})$ which is not identically zero, and $s\mapsto(0,sf)$ is a path of Lagrangian sections not inside the singular locus for $s>0$.
		
		Even if we ask that the support of $\sigma_{\gamma}$ be all of $S^{1}$, then the conclusion of Corollary \ref{constrained} does not hold. Indeed, one can construct counterexamples where the vanishing set of $\sigma_{\gamma-\pounds_{X}\alpha}$ has nonempty interior, for arbitrarily $\mathcal{C}^{1}$-small $\alpha\in\Omega_{cl}^{1}(\mathcal{F}_{L})$. 
	\end{remark}

	The following is an example of a Lagrangian $L$ for which small deformations stay inside the singular locus. However there exist  ``long'' paths of Lagrangian deformations that start at  $L$ and end at a Lagrangian that is no longer contained in the singular locus.

	\begin{ex}\label{inZ} 	Consider the manifold $(\mathbb{T}^{2}\times\mathbb{R}^{2},\theta_{1},\theta_{2},\xi_{1},\xi_{2})$ with log-symplectic structure
		\[
		\Pi:=(\partial_{\theta_{1}}-\partial_{\xi_{2}})\wedge \xi_{1}\partial_{\xi_{1}}+\partial_{\theta_{2}}\wedge\partial_{\xi_{2}}
		\]
		and Lagrangian $L:=\mathbb{T}^{2}\times\{(0,0)\}$. Note that the leaves of $\mathcal{F}_{L}$ are the fibers of the fibration $(\mathbb{T}^{2},\theta_{1},\theta_{2})\rightarrow(S^{1},\theta_{1})$.
		Considering $(\xi_{1},\xi_{2})$ as the fiber coordinates on $T^{*}L$ induced by the frame $\{d\theta_{1},d\theta_{2}\}$, we have that $T^{*}\mathcal{F}_{L}=(\mathbb{T}^{2}\times\mathbb{R},\theta_{1},\theta_{2},\xi_{2})$ with canonical Poisson structure $\partial_{\theta_{2}}\wedge\partial_{\xi_{2}}$. Therefore, using the notation established in the beginning of this section, we have
		\[
		\begin{cases}
		X=\partial_{\theta_{1}}\\
		\gamma=-d\theta_{2}
		\end{cases}.
		\]
		So the section $\sigma_{\gamma}\in\Gamma(\mathcal{H}^{1})$ is nowhere zero, and Corollary \ref{constrained} shows that $\mathcal{C}^{1}$-small deformations of the Lagrangian $L$ stay inside the singular locus.
		
		It is however possible to construct (large) deformations of $L$ that don't lie in the singular locus $T^{*}\mathcal{F}_{L}\times\{0\}\subset T^{*}\mathcal{F}_{L}\times\mathbb{R}$, first deforming $L$ inside the singular locus by large enough $\alpha\in\Omega_{cl}^{1}(\mathcal{F}_{L})$ such that $H_{\gamma-\pounds_{X}\alpha}^{0}(\mathcal{F}_{L})$ is no longer zero. To do so explicitly, note that $\big(g(\theta_{1},\theta_{2})d\theta_{2},f(\theta_{1},\theta_{2})\big)\in\Omega^{1}(\mathcal{F}_{L})\times C^{\infty}(L)$, gives rise to a Lagrangian section of $T^{*}\mathcal{F}_{L}\times\mathbb{R}$ exactly when
		\begin{equation}\label{eqn}
		\frac{\partial f}{\partial\theta_{2}}-f=f\frac{\partial g}{\partial\theta_{1}}.
		\end{equation}
		
		We construct a solution $(g,f)$ to \eqref{eqn} with $f$ not identically zero. For instance, let $f(\theta_{1})$ be any bump function and let $H(\theta_{1})$ be another bump function with $H|_{\text{supp}(f)}\equiv -1$ and $-1\leq H(\theta_{1})\leq 0$. Define $C:=\int_{S^{1}}H(\theta_{1})d\theta_{1}$, so $C>-2\pi$. Let $K:=-C/(C+2\pi)$ and put $G(\theta_{1}):=K\big(1+H(\theta_{1})\big)+H(\theta_{1})$. Notice that $G|_{\text{supp}(f)}\equiv -1$, and since 
		\[
		\int_{S^{1}}G(\theta_{1})d\theta_{1}=K(2\pi+C)+C=0,
		\]
		there exists a periodic primitive $g(\theta_{1})$ with $\partial g/\partial\theta_{1}=G$. We check that $(g,f)$ is a solution to the Maurer-Cartan equation \eqref{eqn}: for $p\notin\text{supp}(f)$ both sides of \eqref{eqn} evaluate to zero, whereas for $p\in\text{supp}(f)$ both sides of \eqref{eqn} are equal to $-f(p)$. Clearly, $f\not\equiv 0$.
		
		So first deforming $L$ along $\alpha:=g(\theta_{1})d\theta_{2}$ and then moving outside of $T^{*}\mathcal{F}_{L}\times\{0\}$ along $f$ gives a Lagrangian deformation that is no longer contained in the singular locus. As a sanity check, looking at $i)$ of {Theorem} \ref{H}, we notice that $H_{\gamma-\pounds_{X}\alpha}^{0}(\mathcal{F}_{L})$ is indeed nonzero and that $f\in H_{\gamma-\pounds_{X}\alpha}^{0}(\mathcal{F}_{L})$, since the section $\sigma_{\gamma-\pounds_{X}\alpha}$ vanishes on the support of $f$.
		
		Moreover, the proof of Corollary \ref{connected} shows that this procedure can be done smoothly, in the sense that one can construct a smooth ``long'' path of Lagrangians that connects $L$ with Lagrangians that are no longer contained in the singular locus. Concretely, let $(g,f)$ be the solution to \eqref{eqn} just constructed, and let $\Psi:\RR\rightarrow\RR$ be any smooth function satisfying $\Psi(s)=0$ for $s\leq 0$, $0<\Psi(s)<1$ for $0<s<1$ and $\Psi(s)=1$ for $s\geq 1$. Take $\Phi:\RR\rightarrow\RR$ to be defined by $\Phi(s)=\Psi(s-1)$.
		Then the path $s\mapsto(\Psi(s)gd\theta_{2},\Phi(s)f)$ consists of Lagrangian sections, it starts at $L$ for $s=0$, passes through $(\alpha,0)$ at time $s=1$, and it reaches $\text{Graph}(\alpha,f)$ at time $s=2$.
		
		The Lagrangian $\text{Graph}(\alpha,f)$ constructed above does not lie entirely outside of the singular locus. Interestingly, it is not possible to find such deformations of $L$. For if we assume by contradiction that $(g,f)$ is a solution to \eqref{eqn} with $f$ nowhere zero, then  
		\[
		\int_{0}^{2\pi}\frac{1}{f}\left(\frac{\partial f}{\partial\theta_{2}}-f\right)d\theta_{1}=\int_{0}^{2\pi}\frac{\partial g}{\partial\theta_{1}}d\theta_{1}=0,
		\]
		so that
		\begin{equation}\label{q}
		\int_{0}^{2\pi}\frac{1}{f}\frac{\partial f}{\partial\theta_{2}}d\theta_{1}=2\pi.
		\end{equation}
		But then we would get
		\begin{align*}
		0=\int_{\mathbb{T}^{2}}d\big(\ln|f|d\theta_{1}\big)=\int_{\mathbb{T}^{2}}\frac{1}{f}\frac{\partial f}{\partial\theta_{2}}d\theta_{2}\wedge d\theta_{1}=-\int_{0}^{2\pi}\left(\int_{0}^{2\pi}\frac{1}{f}\frac{\partial f}{\partial\theta_{2}}d\theta_{1}\right)d\theta_{2}=-4\pi^{2},
		\end{align*}
		using Stokes' theorem in the first and \eqref{q} in the last equality. This contradiction shows that $f$ must have a zero, i.e. the Lagrangian $\text{Graph}(\alpha,f)$ intersects the singular locus.
		
		Alternatively, if there would exist $\alpha=g(\theta_{1},\theta_{2})d\theta_{2}$ and a function $f\in H_{\gamma-\pounds_{X}\alpha}^{0}(\mathcal{F}_{L})$ that is nowhere zero, then {Theorem} \ref{H} implies that $\sigma_{\gamma-\pounds_{X}\alpha}\equiv 0$. Therefore, $\gamma-\pounds_{X}\alpha$ is foliated exact, which implies that there exists a function $k\in C^{\infty}(\mathbb{T}^{2})$ such that
		\[
		-1-\frac{\partial g}{\partial\theta_{1}}=\frac{\partial k}{\partial\theta_{2}}.
		\]
		Integrating this equality against the standard volume form $d\theta_{1}\wedge d\theta_{2}$ on the torus $\mathbb{T}^{2}$ gives a contradiction, since the left hand side integrates to $-4\pi^{2}$ and the right hand side to zero.
	\end{ex}
	
	\subsubsection{\underline{The case with dense leaves}}\label{denseleaves}
	Corollary \ref{constrained} has no counterpart for Lagrangians whose induced foliation $\mathcal{F}_{L}$ has dense leaves, at least not without additional assumptions. Indeed, looking at {Theorem} \ref{H} and the Maurer-Cartan equation \eqref{eqns}, we would need a positive answer to the following question: 
	
	\vspace{0.2cm}
	\noindent
	\emph{If $\gamma\in\Omega^{1}_{cl}(\mathcal{F}_{L})$ is not exact, then is $\gamma-\pounds_{X}\alpha$ still not exact for small enough $\alpha\in\Omega^{1}_{cl}(\mathcal{F}_{L})$?}  
	\vspace{0.2cm}
	
	Drawing inspiration from \cite[Section 4]{bertelson}, we construct an explicit counterexample which answers this question in the negative. Let $L=\left(\mathbb{T}^{2},\theta_{1},\theta_{2}\right)$ be the torus with Kronecker foliation $T\mathcal{F}_{L}=\text{Ker}(d\theta_{1}-\lambda d\theta_{2})$, for $\lambda\in\mathbb{R}\setminus\mathbb{Q}$ irrational. A global frame for $T^{*}\mathcal{F}_{L}$ is given by $d\theta_{2}$, so that every foliated one-form looks like $f(\theta_{1},\theta_{2})d\theta_{2}$, which is automatically closed by dimension reasons. It is exact when there exists $g(\theta_{1},\theta_{2})\in C^{\infty}\left(\mathbb{T}^{2}\right)$ such that
	\begin{equation}\label{exact}
	f=\lambda\frac{\partial g}{\partial\theta_{1}}+\frac{\partial g}{\partial\theta_{2}}.
	\end{equation}
	Expanding $f$ and $g$ in double Fourier series,
	\[
	f(\theta_{1},\theta_{2})=\sum_{m,n\in\mathbb{Z}}f_{m,n}e^{2\pi i(n\theta_{1}+m\theta_{2})} \hspace{0.5cm}\text{and}\hspace{0.5cm} g(\theta_{1},\theta_{2})=\sum_{m,n\in\mathbb{Z}}g_{m,n}e^{2\pi i(n\theta_{1}+m\theta_{2})},
	\]
	the equality \eqref{exact} is equivalent with
	\begin{equation}\label{coef}
	f_{m,n}=2\pi i(m+\lambda n)g_{m,n}\hspace{1cm}\forall m,n\in\mathbb{Z},
	\end{equation}
	which implies in particular that $f_{0,0}=0$. Note that the $g_{m,n}$ for $(m,n)\neq (0,0)$ are uniquely determined by the relation \eqref{coef} since $\lambda$ is irrational. 
	
	Assume moreover that the slope $\lambda\in\mathbb{R}\setminus\mathbb{Q}$
	is a Liouville number (see Definition \ref{liouville} in the Appendix). In this case the foliated cohomology group $H^{1}(\mathcal{F}_{L})$ is infinite dimensional \cite{haefliger}, \cite[Chapter III]{kronecker}, as one can construct smooth functions $f(\theta_{1},\theta_{2})$ in such a way that the $g_{m,n}$ defined in \eqref{coef} are not the Fourier coefficients of a smooth function. We give an example of such a function $f(\theta_{1},\theta_{2})$ in part i) of the proof below.
	
	\begin{lemma}\label{dense}
		Consider the torus $L=\left(\mathbb{T}^{2},\theta_{1},\theta_{2}\right)$ endowed with the Kronecker foliation $T\mathcal{F}_{L}=\text{Ker}(d\theta_{1}-\lambda d\theta_{2})$, for $\lambda\in\mathbb{R}\setminus\mathbb{Q}$ a Liouville number.
		There exist a non-exact foliated one-form $\gamma\in\Omega^{1}(\mathcal{F}_{L})$ and $X\in\mathfrak{X}(L)^{\mathcal{F}_{L}}$ such that every  $\mathcal{C}^{\infty}$-open neighborhood of $0$ in $\Omega^{1}(\mathcal{F}_{L})$ contains a one-form $\alpha$ for which $\gamma-\pounds_{X}\alpha$ is foliated exact.
	\end{lemma}
	\begin{proof}
		The proof is divided into two steps. In the first step, we construct $\gamma\in\Omega^{1}(\mathcal{F}_{L})$. In the second step, we fix $X\in\mathfrak{X}(L)^{\mathcal{F}_{L}}$ and we construct a sequence of foliated one-forms $\alpha_{k}$ such that $\gamma-\pounds_{X}\alpha_{k}$ is exact for each value of $k$, and $\alpha_{k}\rightarrow 0$ in the Fr\'echet $\mathcal{C}^{\infty}$-topology.
		\begin{enumerate}[i)]
			\item We first have to find a foliated one-form $\gamma=f(\theta_{1},\theta_{2})d\theta_{2}$ that is not exact. Moreover, since we want to approach $\gamma$ by means of exact one-forms, we need that the  coefficient $f_{0,0}=\frac{1}{4\pi^{2}}\int_{\mathbb{T}^{2}}fd\theta_{1}\wedge d\theta_{2}$ is zero. This can be done as follows. By Lemma \ref{liou} in the Appendix, for each integer $p\geq 1$, there exists a pair of integers $(m_{p},n_{p})$ such that
			\begin{equation}\label{e}
			|m_{p}+\lambda n_{p}|\leq\frac{1}{\left(|m_{p}|+|n_{p}|\right)^{p}}.
			\end{equation}
			We can moreover assume that $(m_{p},n_{p})\neq(m_{q},n_{q})$ for $p\neq q$, and that $n_{p}\geq p$ (see Remark \ref{assumptions}). We now define $f(\theta_{1},\theta_{2})$ by means of its Fourier coefficients $f_{m,n}$, setting
			\[
			f_{m,n}=\begin{cases}
			(m_{p}+\lambda n_{p})n_{p}\hspace{1cm}\text{if}\ (m,n)=(m_{p},n_{p})\\
			0\hspace{3.08cm}\text{else}
			\end{cases}.
			\]
			To see that these coefficients define a smooth function, we estimate for $k\in\mathbb{N}$:
			\begin{align*}
			|f_{m_{p},n_{p}}|\cdot \|(m_{p},n_{p})\|^{k}&=|m_{p}+\lambda n_{p}|\cdot n_{p}\cdot\|(m_{p},n_{p})\|^{k}\\
			&\leq|m_{p}+\lambda n_{p}|\cdot(|m_{p}|+|n_{p}|)\cdot\left(|m_{p}|+|n_{p}|\right)^{k}\\
			&=|m_{p}+\lambda n_{p}|\cdot\left(|m_{p}|+|n_{p}|\right)^{k+1}\\
			&\leq \frac{1}{\left(|m_{p}|+|n_{p}|\right)^{p}}\cdot\left(|m_{p}|+|n_{p}|\right)^{k+1}\\
			&\leq\left(\frac{1}{p}\right)^{p-k-1},
			\end{align*}
			where the last inequality holds for $p\geq k+1$. This shows that $\sup_{(m,n)\in\mathbb{Z}^{2}}|f_{m,n}|\|(m,n)\|^{k}$ is finite for each value of $k\in\mathbb{N}$, and therefore $f(\theta_{1},\theta_{2})$ is indeed smooth. To see that $\gamma=f(\theta_{1},\theta_{2})d\theta_{2}$ is not exact, note that the Fourier coefficients of a primitive $g(\theta_{1},\theta_{2})$ are given by \eqref{coef}:
			\begin{equation}\label{gmn}
			g_{m,n}=\frac{f_{m,n}}{2\pi i (m+\lambda n)}\hspace{1cm} \text{for}\ (m,n)\neq(0,0).
			\end{equation}
			Therefore
			\[
			|g_{m_{p},n_{p}}|=\frac{1}{2\pi}n_{p}\geq \frac{1}{2\pi}p,
			\]
			which does not tend to zero for $p\rightarrow\infty$. So the $g_{m,n}$ defined in \eqref{gmn} are not the Fourier coefficients of a smooth function.
			\item We let $X:=\partial_{\theta_{1}}$. Notice that $X\in\mathfrak{X}(L)^{\mathcal{F}_{L}}$ and that $X$ is transverse to the leaves of $\mathcal{F}_{L}$. We now construct a sequence $\alpha_{k}\in\Omega^{1}_{cl}(\mathcal{F}_{L})$ such that $\gamma-\pounds_{X}\alpha_{k}$ is exact and $\alpha_{k}\rightarrow 0$ in the $\mathcal{C}^{\infty}$-topology. For each integer $k\geq 1$ we define $\alpha_{k}=h_{k}(\theta_{1},\theta_{2})d\theta_{2}$, where $h_{k}(\theta_{1},\theta_{2})$ is given by its Fourier coefficients
			\[
			h_{m,n}^{k}=\begin{cases}
			\left(\frac{m_{p}+\lambda n_{p}}{2\pi i}\right)\cdot \phi_{k}\left(\frac{1}{p}\right)\hspace{1cm}\text{if}\  (m,n)=(m_{p},n_{p})\\
			0\hspace{4cm}\text{else}
			\end{cases}.
			\] 
			Here $\phi_{k}$ is a bump function on $\mathbb{R}$ that is identically equal to $1$ on the interval $[0,\frac{1}{k}]$. As before, we see that $h_{k}$ is a smooth function by the estimate
			\[
			|h^{k}_{m_{p},n_{p}}|\cdot\|(m_{p},n_{p})\|^{l}\leq \frac{1}{2\pi}|m_{p}+\lambda n_{p}|\cdot\left(|m_{p}|+|n_{p}|\right)^{l}\leq \frac{1}{2\pi}\left(\frac{1}{p}\right)^{p-l},
			\]
			where the last inequality holds for $p\geq l$. Note that $\gamma-\pounds_{X}\alpha_{k}=(f-\partial_{\theta_{1}}h_{k})d\theta_2$ is indeed exact because the Fourier coefficients of $f-\partial_{\theta_{1}}h_{k}$ are given by
			\[
			f_{m,n}-2\pi i \cdot n\cdot h^{k}_{m,n}=\begin{cases}
			(m_{p}+\lambda n_{p})n_{p}\left(1-\phi_{k}\left(\frac{1}{p}\right)\right)\hspace{1cm}\text{if}\ (m,n)=(m_{p},n_{p})\\
			0\hspace{5.5cm}\text{else}
			\end{cases},
			\]
			only finitely many of which are nonzero. Finally, by letting $k$ increase, we can make $\alpha_{k}$ as $\mathcal{C}^{\infty}$-small as desired. Indeed, for each integer $l$ we have
			\begin{align}\label{est}
			\|h_{k}\|_{l}&\leq \sum_{0\leq j\leq l} \sum_{p\geq k}|m_{p}+\lambda n_{p}|\cdot\|(m_{p},n_{p})\|^{j}\cdot(2\pi)^{j-1}\nonumber\\
			&\leq \sum_{0\leq j\leq l} \sum_{p\geq k}\left(\frac{1}{p}\right)^{p-j}\cdot(2\pi)^{j-1},
			\end{align}
			where the last inequality holds for $k\geq l$. The expression \eqref{est} tends to zero for $k\rightarrow\infty$, since the inner sum is the tail of a convergent series for each value of $j\in\{0,\ldots,l\}$.
		\end{enumerate}
	\end{proof}
	
	The above construction gives a concrete counterexample to the version of Corollary \ref{constrained} for Lagrangians $(L,\mathcal{F}_{L})$ with dense leaves. We only have to realize $L=(\mathbb{T}^{2},\theta_{1},\theta_{2})$ with $T\mathcal{F}_{L}=\text{Ker}(d\theta_{1}-\lambda d\theta_{2})$ as a Lagrangian submanifold contained in the singular locus of some log-symplectic manifold. The normal form \eqref{pitilde} tells us how to construct this log-symplectic manifold. If $(\xi_{1},\xi_{2})$ are the fiber coordinates on $T^{*}L$, and $\xi$ is the fiber coordinate on $T^{*}\mathcal{F}_{L}$ corresponding with the frame $\{d\theta_{2}\}$, then the restriction map reads
	\[
	r:T^{*}L\rightarrow T^{*}\mathcal{F}_{L}:(\theta_{1},\theta_{2},\xi_{1},\xi_{2})\mapsto(\theta_{1},\theta_{2},\lambda\xi_{1}+\xi_{2}),
	\]
	and therefore the canonical Poisson structure on $T^{*}\mathcal{F}_{L}$ is
	\[
	\Pi_{can}=r_{*}\left(\partial_{\theta_{1}}\wedge\partial_{\xi_{1}}+\partial_{\theta_{2}}\wedge\partial_{\xi_{2}}\right)=\left(\lambda\partial_{\theta_{1}}+\partial_{\theta_{2}}\right)\wedge\partial_{\xi}.
	\]
	Let $V$ denote the vertical Poisson vector field on $T^{*}\mathcal{F}_{L}$ defined by the one-form $\gamma\in\Gamma(T^{*}\mathcal{F}_{L})$ constructed in Lemma \ref{dense}, and let $X:=\partial_{\theta_{1}}$. Then $V+X$ is a Poisson vector field on $\left(T^{*}\mathcal{F}_{L},\Pi_{can}\right)$ transverse to the symplectic leaves, so the following is a log-symplectic structure:
	\begin{equation*}
	\Big(T^{*}\mathcal{F}_{L}\times\mathbb{R},(V+X)\wedge t\partial_{t}+\left(\lambda\partial_{\theta_{1}}+\partial_{\theta_{2}}\right)\wedge\partial_{\xi}\Big),
	\end{equation*}
	and $L$ is Lagrangian inside $T^{*}\mathcal{F}_{L}$ with induced foliation $\mathcal{F}_{L}$. The above argument shows that, for each integer $k\geq 0$, there exists arbitrarily $\mathcal{C}^{k}$-small $\alpha\in\Omega^{1}_{cl}(\mathcal{F}_{L})$ for which $\gamma-\pounds_{X}\alpha$ is exact. By {Theorem} \ref{H}, there exists  $f\in H^{0}_{\gamma-\pounds_{X}\alpha}(\mathcal{F}_{L})$ not identically zero, where $f$ can be made arbitrarily $\mathcal{C}^{k}$-small by rescaling with a nonzero constant. The Maurer-Cartan equation \eqref{eqns} now implies that $\text{Graph}(\alpha,f)$ is an arbitrarily $\mathcal{C}^{k}$-small Lagrangian deformation of $L$ that is not completely contained in the singular locus $T^{*}\mathcal{F}_{L}$.

	\begin{remark}\label{generic}
		In the above counterexample, it is crucial that the slope $\lambda$ is a Liouville number. If $\mathcal{F}_{L}$ is the Kronecker foliation with generic (i.e. not Liouville) irrational slope $\lambda$, then $H^{1}(\mathcal{F}_{L})=\mathbb{R}[d\theta_{2}]$. In this case, exactness is detected by integration, for if we denote 
		\[
		I:C^{\infty}(\mathbb{T}^{2})\rightarrow\mathbb{R}:h(\theta_{1},\theta_{2})\mapsto\int_{\mathbb{T}^{2}}h(\theta_{1},\theta_{2})d\theta_{1}\wedge d\theta_{2},
		\]
		then $hd\theta_{2}\in\Omega^{1}(\mathcal{F}_{L})$ being exact is equivalent with $h\in I^{-1}(0)$. Since integration is $\mathcal{C}^{0}$-continuous, it follows that the space of exact one-forms $\text{Im}(d_{\mathcal{F}_{L}})\subset\left(\Omega^{1}(\mathcal{F}_{L}),\mathcal{C}^{0}\right)$ is closed. Therefore, if we take $\gamma\in\Omega^{1}(\mathcal{F}_{L})$ not exact, so that $H^{0}_{\gamma}(\mathcal{F}_{L})=0$, then also $\gamma-\pounds_{X}\alpha$ is not exact for $\mathcal{C}^{1}$-small $\alpha$, so that still $H^{0}_{\gamma-\pounds_{X}\alpha}(\mathcal{F}_{L})=0$. This shows that, if in the above counterexample we take a generic slope $\lambda\in\mathbb{R}\setminus\mathbb{Q}$, then $\mathcal{C}^{1}$-small deformations of $L$ do stay inside the singular locus.
	\end{remark}
	
	The problem in the above counterexample is that the space of exact one-forms in $\Omega^{1}_{cl}(\mathcal{F}_{L})$ is not closed with respect to the Fr\'echet $\mathcal{C}^{\infty}$-topology generated by $\mathcal{C}^{k}$-norms $\{\|\cdot\|_{k}\}_{k\geq 0}$. Under the additional assumption that $H^{1}(\mathcal{F}_{L})$ is finite dimensional, this problem does not occur, and we obtain the following analog to Corollary \ref{constrained}. 
	
	\begin{prop}\label{denserestr}
		Let $L$ be a compact, connected Lagrangian whose induced foliation $\mathcal{F}_{L}$ has dense leaves. Assume that $H^{1}(\mathcal{F}_{L})$ is finite dimensional and that $\gamma\in\Omega^{1}_{cl}(\mathcal{F}_{L})$ is not exact. Then there exists a neighborhood $\mathcal{V}$ of $0$ in $\big(\Gamma(T^{*}\mathcal{F}_{L}\times\RR),\mathcal{C}^{\infty}\big)$ such that if $\text{Graph}(\alpha,f)$ is Lagrangian for $(\alpha,f)\in\mathcal{V}$, then $f\equiv 0$.
	\end{prop}
	\begin{proof}
		Consider the Fr\'echet space $\big(\Omega^{1}(\mathcal{F}_{L}),\mathcal{C}^{\infty}\big)$ and notice that the space of closed foliated one-forms $\Omega^{1}_{cl}(\mathcal{F}_{L})\subset\big(\Omega^{1}(\mathcal{F}_{L}),\mathcal{C}^{\infty}\big)$ is closed. To see this, it is enough to note that $d_{\mathcal{F}_{L}}$ is continuous with respect to the $\mathcal{C}^{\infty}$-topology and that $\{0\}\subset\big(\Omega^{2}(\mathcal{F}_{L}),\mathcal{C}^{\infty}\big)$ is closed since Fr\'echet spaces are Hausdorff. Consequently, $\big(\Omega^{1}_{cl}(\mathcal{F}_{L}),\mathcal{C}^{\infty}\big)$ is itself a Fr\'echet space.  Moreover, the space of exact foliated one-forms $\text{Im}(d_{\mathcal{F}_{L}})\subset\left(\Omega^{1}_{cl}(\mathcal{F}_{L}),\mathcal{C}^{\infty}\right)$ is a closed subspace. Indeed, by assumption, the range of $d_{\mathcal{F}_{L}}:\big(C^{\infty}(L),\mathcal{C}^{\infty}\big)\rightarrow\big(\Omega^{1}_{cl}(\mathcal{F}_{L}),\mathcal{C}^{\infty}\big)$ has finite codimension, so it must be closed because of the open mapping theorem (see \cite[Remark 3.2]{bertelson}).
		\newline
		\indent
		As $\gamma$ is not foliated exact, there exists a $\mathcal{C}^{\infty}$-open neighborhood of $\gamma$ consisting of non-exact one-forms. By continuity of the map $\big(\Omega^{1}_{cl}(\mathcal{F}_{L}),\mathcal{C}^{\infty}\big)\rightarrow\big(\Omega^{1}_{cl}(\mathcal{F}_{L}),\mathcal{C}^{\infty}\big):\alpha\mapsto\gamma-\pounds_{X}\alpha$, we find a $\mathcal{C}^{\infty}$-open neighborhood $\mathcal{U}$ of $0$ in $\Omega^{1}_{cl}(\mathcal{F}_{L})$ such that $\gamma-\pounds_{X}\alpha$ is not exact for all $\alpha\in\mathcal{U}$. There exists a $\mathcal{C}^{\infty}$-open subset $\mathcal{U}'\subset\Omega^{1}(\mathcal{F}_{L})$ such that $\mathcal{U}= \mathcal{U}'\cap\Omega^{1}_{cl}(\mathcal{F}_{L})$. We now define the $\mathcal{C}^{\infty}$-neighborhood $\mathcal{V}$ of $0$ in $\left(\Gamma(T^{*}\mathcal{F}_{L}\times\mathbb{R}),\mathcal{C}^{\infty}\right)$ by
		\[
		\mathcal{V}:=\{(\alpha,f)\in\Gamma(T^{*}\mathcal{F}_{L}\times\mathbb{R}):\alpha\in\mathcal{U}'\}.
		\]
		To see that $\mathcal{V}$ satisfies the criteria, let $(\alpha,f)\in\mathcal{V}$ be such that $\text{Graph}(\alpha,f)$ is Lagrangian in $(T^{*}\mathcal{F}_{L}\times\mathbb{R},\widetilde{\Pi})$. {Then $f\in H^{0}_{\gamma-\pounds_{X}\alpha}(\mathcal{F}_{L})$ and $\alpha\in\mathcal{U}$, so that $\gamma-\pounds_{X}\alpha$ is not foliated exact. Recalling that the leaves of $\mathcal{F}_{L}$ are assumed to be dense, we now use $ii)$ of {Theorem} \ref{H} to see that $ H^{0}_{\gamma-\pounds_{X}\alpha}(\mathcal{F}_{L})$ vanishes. Consequently $f\equiv 0$, which finishes the proof.}
	\end{proof}
	
	{Since the $\mathcal{C}^{\infty}$-topology is generated by the increasing family of $\mathcal{C}^{k}$-norms, every $\mathcal{C}^{\infty}$-open neighborhood contains a $\mathcal{C}^{k}$-open neighborhood for some $k\in\mathbb{N}$. So shrinking the neighborhood $\mathcal{V}$ obtained in the above proposition, one can assume that it is a $\mathcal{C}^{k}$-neighborhood of the zero section, for some (unspecified) $k\in\mathbb{N}$.}
	
	\subsection{Obstructedness of deformations}\label{subsec:obstr}
	\leavevmode
	\vspace{0.1cm}
	
	Recall that a deformation problem governed by a DGLA $(W,d,[\![\cdot,\cdot]\!])$ is formally/smoothly unobstructed if every closed element $\alpha\in W_{1}$ -- i.e. every first order deformation -- can be extended to a formal/smooth curve of Maurer-Cartan elements. A way to detect obstructedness is by means of the Kuranishi map 
	\[
	Kr:H^{1}(W)\rightarrow H^{2}(W):[\alpha]\mapsto \big[[\![\alpha,\alpha]\!]\big],
	\]
	for if $Kr([\alpha])$ is nonzero, then $\alpha$ is formally (hence smoothly) obstructed \cite[Theorem 11.4]{OP}.

	For the deformation problem of a Lagrangian $L^{n}$ contained in the singular locus of a log-symplectic manifold $(M^{2n},Z,\Pi)$, a first order deformation is a pair $(\alpha_{1},f_{1})\in\Gamma\left(T^{*}\mathcal{F}_{L}\times\mathbb{R}\right)$ such that 
	\begin{equation}\label{eq:linMC}
	\begin{cases}
	d_{\mathcal{F}_{L}}\alpha_{1}=0\\
	d_{\mathcal{F}_{L}}f_{1}+f_{1}\gamma=0
	\end{cases}.
	\end{equation}
	Clearly, first order deformations of the specific form $(\alpha_1,0)$ or $(0,f_1)$ are smoothly unobstructed, since $s(\alpha_1,0)$ and $s(0,f_1)$ satisfy the Maurer-Cartan equation
	\eqref{eqns} for all $s\in\RR$.
	
	\subsubsection{\underline{Obstructedness}}
	We show that the above deformation problem is formally obstructed in general. The Kuranishi map of the DGLA $\big(\Gamma\left(\wedge^{\bullet}\left(T^{*}\mathcal{F}_{L}\times\mathbb{R}\right)\right),d,[\![\cdot,\cdot]\!]\big)$ described in Corollary \ref{DGLA} reads
	\begin{equation}\label{kur}
	Kr:H^{1}\left(\Gamma\left(\wedge^{\bullet}(T^{*}\mathcal{F}_{L}\times\mathbb{R})\right)\right)\rightarrow H^{2}\left(\Gamma\left(\wedge^{\bullet}(T^{*}\mathcal{F}_{L}\times\mathbb{R})\right)\right):\left[(\alpha,f)\right]\mapsto\left[(0,2f\pounds_{X}\alpha)\right],
	\end{equation}
	and the following example shows that this map need not be identically zero.
{(This example will be generalized later on, in Ex. \ref{ex:T2obstr} i)).}
	
	\begin{ex}[An obstructed example]\label{ex:unob}
		Consider the manifold $\mathbb{T}^{2}\times\mathbb{R}^{2}$, regarded as a trivial vector bundle over $\mathbb{T}^{2}$. Denote its coordinates by $(\theta_{1},\theta_{2},\xi_1,\xi_2)$ and endow it with a log-symplectic structure $\Pi$ given by
		\[
		\Pi=\partial_{\theta_{1}}\wedge \xi_1\partial_{\xi_1}+\partial_{\theta_{2}}\wedge\partial_{\xi_2}.
		\]
		Note that $L:=\mathbb{T}^{2}\times\{(0,0)\}$ is a Lagrangian submanifold contained in the singular locus $\mathbb{T}^{2}\times\mathbb{R}=\{\xi_1=0\}$. It inherits a codimension-one foliation $\mathcal{F}_{L}$ with tangent distribution $T\mathcal{F}_{L}=\text{Ker}(d\theta_{1})$, so the cotangent bundle $T^{*}\mathcal{F}_{L}$ has a global frame given by $d\theta_{2}$. In the notation established earlier, we now have
		\[
		\begin{cases}
		\gamma=0\\
		X=\partial_{\theta_{1}}
		\end{cases},
		\]
		and the differential $d$ of the DGLA acts as
		\begin{equation}\label{dif}
		d:\Gamma\left(T^{*}\mathcal{F}_{L}\times\mathbb{R}\right)\rightarrow\Gamma\left(\wedge^{2}(T^{*}\mathcal{F}_{L}\times\mathbb{R})\right):(g d\theta_{2},k)\mapsto \left(0,-\frac{\partial k}{\partial\theta_{2}}d\theta_{2}\right).
		\end{equation}
		Since the Kuranishi map \eqref{kur} is given by
		\[
		Kr\Big(\left[\big(gd\theta_{2},f\big)\right]\Big)=\left[\left(0,2f\frac{\partial g}{\partial\theta_{1}}d\theta_{2}\right)\right],
		\]
		it is clear that
		\begin{align}\label{obstruction}
		Kr\Big(\left[\big(gd\theta_{2},f\big)\right]\Big)=0&\Leftrightarrow f\frac{\partial g}{\partial\theta_{1}}=\frac{\partial k}{\partial\theta_{2}}\ \text{for some}\ k\in C^{\infty}(\mathbb{T}^{2})\nonumber\\
		&\Leftrightarrow \int_{S^{1}}f\frac{\partial g}{\partial\theta_{1}}d\theta_{2}=0.
		\end{align}
		The equation \eqref{obstruction} is a non-trivial obstruction to the prolongation of infinitesimal deformations. For instance, the section $\big(\sin(\theta_{1})d\theta_{2},\cos(\theta_{1})\big)\in\Gamma\left(T^{*}\mathcal{F}_{L}\times\mathbb{R}\right)$ is an infinitesimal deformation of $L$ since it is closed with respect to the differential \eqref{dif}. But it cannot be prolonged to a path of deformations, since the integral \eqref{obstruction} is nonzero. 
	\end{ex}

	\subsubsection{\underline{Formally unobstructed deformations}}
	It is well-known that a deformation problem is formally unobstructed whenever the second cohomology group of the DGLA governing it vanishes \cite[Theorem 11.2]{OP}. Specializing to our situation, say we have a first order deformation $(\alpha_{1},f_{1})$ as in eq. \eqref{eq:linMC}
	and we wish to prolong it to a formal power series solution $\sum_{k\geq 1}(\alpha_{k},f_{k})\epsilon^{k}$ of the Maurer-Cartan equation. So we require that
	\[
	d\left(\sum_{k\geq 1}(\alpha_{k},f_{k})\epsilon^{k}\right)+\frac{1}{2}\left[\!\left[\sum_{k\geq 1}(\alpha_{k},f_{k})\epsilon^{k},\sum_{k\geq 1}(\alpha_{k},f_{k})\epsilon^{k}\right]\!\right]=0.
	\] 
	Collecting all terms in $\epsilon^{n}$ gives 
	\begin{align*}
	&d(\alpha_{n},f_{n})+\frac{1}{2}\sum_{\substack{k,l\geq 1\\k+l=n }}\left[\!\left[(\alpha_{k},f_{k}),(\alpha_{l},f_{l})\right]\!\right]=0\\
	&\hspace{-0.7cm}\Leftrightarrow\left(-d_{\mathcal{F}_{L}}\alpha_{n},-d_{\mathcal{F}_{L}}f_{n}-f_{n}\gamma\right)+\frac{1}{2}\sum_{\substack{k,l\geq 1\\k+l=n }}\left(0,f_{l}\pounds_{X}\alpha_{k}+f_{k}\pounds_{X}\alpha_{l}\right)=0\\
	&\hspace{-0.7cm}\Leftrightarrow\begin{cases}
	d_{\mathcal{F}_{L}}\alpha_{n}=0\\
	d_{\mathcal{F}_{L}}f_{n}+f_{n}\gamma-\frac{1}{2}\sum_{\substack{k,l\geq 1\\k+l=n }}\left(f_{l}\pounds_{X}\alpha_{k}+f_{k}\pounds_{X}\alpha_{l}\right)=0
	\end{cases}.
	\end{align*}
	We can always construct a formal power series solution if $H^{1}_{\gamma}(\mathcal{F}_{L})=0$. Concretely, constructing $(\alpha_{k},f_{k})$ inductively, we can set $\alpha_{k}=0$ for $k\geq 2$ and choose $f_{k}$ such that
	\begin{equation}\label{formalpath}
	d_{\mathcal{F}_{L}}f_{k}+f_{k}\gamma=f_{k-1}\pounds_{X}\alpha_{1}.
	\end{equation}
	A quick proof by induction indeed shows that the right hand side of \eqref{formalpath} is closed with respect to the differential $d_{\mathcal{F}_{L}}^{\gamma}$ for each $k\geq 2$. In conclusion, we have proved the following:
	
	\begin{cor}\label{formal}
		If $H^{1}_{\gamma}(\mathcal{F}_{L})=0$, then the deformation problem is formally unobstructed.
	\end{cor}
	
	Note that this assumption is weaker than requiring that the second cohomology group of the DGLA is zero, since the latter is given by $H^{2}(\mathcal{F}_{L})\oplus H^{1}_{\gamma}(\mathcal{F}_{L})$. 
	
	We will see that the vanishing of $H^{1}_{\gamma}(\mathcal{F}_{L})$ in fact ensures that the deformation problem is smoothly unobstructed, at least for Lagrangians that are compact and connected.

	\subsubsection{\underline{Smoothly unobstructed deformations: general results}}
	We  give a sufficient condition for smooth unobstructedness.
	When $\pounds_{X}\alpha$ is foliated exact, we have 
	{an isomorphism in cohomology}
  $H^{0}_{\gamma}(\mathcal{F}_{L})\cong H^{0}_{\gamma-\pounds_{X}\alpha}(\mathcal{F}_{L})$, 
 {obtained from the isomorphism of cochain complexes given in eq. \eqref{isocompl}.} 
  Using this isomorphism, from a solution of the linearized Maurer-Cartan equation \eqref{eq:linMC} we can construct a solution of the Maurer-Cartan equation \eqref{eqns}.
	This leads to the following result, which we prove with a short direct computation.
	
	\begin{prop}\label{unobstructed}
		If $(\alpha,f)\in\Gamma\left(T^{*}\mathcal{F}_{L}\times\mathbb{R}\right)$ is a first order deformation such that $\pounds_{X}\alpha$ is foliated exact, then $(\alpha,f)$ is smoothly unobstructed.
	\end{prop}
	
	\begin{proof}
		Let $\pounds_{X}\alpha=d_{\mathcal{F}_{L}}h$ for $h\in C^{\infty}(L)$. We claim that the path 
		\begin{equation}\label{eq:prolong}
		s\mapsto(s\alpha,sfe^{sh})
		\end{equation}
		is a prolongation of $(\alpha,f)$ consisting of Lagrangian sections for all times $s$. Indeed, we have the following equivalences:\begin{align}\label{f}
		&d_{\mathcal{F}_{L}}(sfe^{sh})+sfe^{sh}\left(\gamma-\pounds_{X}s\alpha\right)=0\nonumber\\
		&\hspace{-0.6cm}\Leftrightarrow d_{\mathcal{F}_{L}}(sfe^{sh})+sfe^{sh}\left(\gamma-d_{\mathcal{F}_{L}}sh\right)=0\nonumber\\
		&\hspace{-0.6cm}\Leftrightarrow d_{\mathcal{F}_{L}}sf+sf\gamma=0.
		\end{align}
		In the last equivalence we use $i)$ of Lemma \ref{cohomology}, which says that the following map is an isomorphism of cochain complexes:
		\[
		\left(\Omega^{\bullet}(\mathcal{F}_{L}),d_{\mathcal{F}_{L}}^{\gamma-d_{\mathcal{F}_{L}}sh}\right)\rightarrow\left(\Omega^{\bullet}(\mathcal{F}_{L}),d_{\mathcal{F}_{L}}^{\gamma}\right):\beta\mapsto e^{-sh}\beta.
		\]
		The equality \eqref{f} is satisfied, since $(\alpha,f)$ is a first order deformation. So, by Theorem \ref{equations}, $(s\alpha,sfe^{sh})$ is indeed a Lagrangian section for each time $s$. Clearly, the path passes through the zero section at $s=0$ with velocity $(\alpha,f)$. This proves the claim.
	\end{proof}

	As a consistency check, we note that a first order deformation  $(\alpha,f)$ as in Prop. \ref{unobstructed} maps  to zero under the Kuranishi map:  we have $Kr([(\alpha,f)])=[(0, 2f\pounds_{X}\alpha)]$ by eq. \eqref{kur}. If  $\pounds_{X}\alpha=d_{\mathcal{F}_{L}}h$ for some $h\in C^{\infty}(L)$, then $d^{\gamma}_{\mathcal{F}_{L}}(f\cdot h)=d^{\gamma}_{\mathcal{F}_{L}}f\cdot h+f\cdot d_{\mathcal{F}_{L}}h =f\pounds_{X}\alpha$.

	\begin{remark}\label{rem:geominter}
		We give a geometric interpretation of Proposition \ref{unobstructed}.
		\begin{enumerate}[i)]
			\item For a closed foliated one-form $\alpha\in\Omega^{1}(\mathcal{F}_{L})$, exactness of $\pounds_{X}\alpha$ is equivalent with the existence of a closed one-form $\widetilde{\alpha}\in\Omega^{1}(L)$ that extends $\alpha$.
			Indeed, if $\widetilde{\alpha}\in\Omega^{1}(L)$ is a closed extension of $\alpha$ and $r:\Omega^{1}(L)\rightarrow\Omega^{1}\left(\mathcal{F}_{L}\right)$ is the restriction map, then 
			\[
			0=r\left(\iota_{X}d\widetilde{\alpha}\right)=r\left(\pounds_{X}\widetilde{\alpha}-d\iota_{X}\widetilde{\alpha}\right)=\pounds_{X}\alpha-d_{\mathcal{F}_{L}}\left(\iota_{X}\widetilde{\alpha}\right),
			\]
			which shows that $\pounds_{X}\alpha=d_{\mathcal{F}_{L}}\left(\iota_{X}\widetilde{\alpha}\right)$ is exact. Conversely, assume that $\pounds_{X}\alpha=d_{\mathcal{F}_{L}}h$ for some $h\in C^{\infty}(L)$. Let $\widetilde{\alpha}\in\Omega^{1}(L)$ be the unique extension\footnote{{Recall that $\mathcal{F}_{L}$ is a codimension-one foliation on $L$ and that $X\in\mathfrak{X}(L)$ is transverse to the leaves of $\mathcal{F}_{L}$. Hence, extensions of a foliated one-form $\alpha\in\Omega^{1}(\mathcal{F}_{L})$ are uniquely specified by their value on $X\in\mathfrak{X}(L)$.}} of $\alpha$ satisfying $\widetilde{\alpha}(X)=h$. Then $\widetilde{\alpha}$ is closed, since
			\[
			r\left(\iota_{X}d\widetilde{\alpha}\right)=r\left(\pounds_{X}\widetilde{\alpha}-d\iota_{X}\widetilde{\alpha}\right)=\pounds_{X}\alpha-d_{\mathcal{F}_{L}}h=0.
			\]
			\item The smooth path\footnote{This path can certainly not be obtained by applying Poisson diffeomorphisms to $L$ itself, since the latter preserve the Poisson submanifold $T^{*}\mathcal{F}_{L}$.} of Lagrangian deformations given by \eqref{eq:prolong} is obtained by applying certain Poisson diffeomorphisms of $T^{*}\mathcal{F}_{L}\times\mathbb{R}$ to the smooth path of Lagrangian sections $s\mapsto (0,sf)$. More precisely, as in item i), assume that $\pounds_{X}\alpha=d_{\mathcal{F}_{L}}h$, and let $\widetilde{\alpha}\in\Omega^{1}(L)$ be the closed one-form extending $\alpha$ determined by $\widetilde{\alpha}(X)=h$. As before, denote by $pr\colon T^{*}\mathcal{F}_{L}\times\mathbb{R}\to L$ and $p:T^{*}\mathcal{F}_{L}\rightarrow{L}$ the vector bundle projections. 
			Since $\widetilde{\alpha}$ is a closed one-form, it gives rise to a Poisson vector field on $T^{*}\mathcal{F}_{L}\times\mathbb{R}$, namely $$\widetilde{\Pi}^{\sharp}(pr^*\widetilde{\alpha})=(pr^*h)t\partial_{t}+\Pi_{can}^\sharp(p^*\widetilde{\alpha}).$$
			Notice that this vector field is tangent to the fibers of $pr$, and that the second summand is the constant vector field  on the fibers of $T^{*}\mathcal{F}_{L}$ with value $\alpha$.
			The flow at time $s$ of $\widetilde{\Pi}^{\sharp}(pr^*\widetilde{\alpha})$ maps $graph(0,sf)$ to $graph(s\alpha,sfe^{sh})$.
			
			In case $\alpha=d_{\mathcal{F}_{L}}g$ is exact, then we can interpret this construction in terms of the DGLA governing the deformation problem. Indeed, Remark \ref{Gaction} shows that the gauge action by the degree zero element $(g,0)$ takes the Maurer-Cartan element $(0,sf)$ to $\big(sd_{\mathcal{F}_{L}}g,sfe^{sX(g)}\big)$. This is consistent with the above, since $X(g)$ is a primitive of $\pounds_{X}\alpha$.
		\end{enumerate}
	\end{remark}
	
	\begin{cor}\label{H1}
		If $H^{1}(\mathcal{F}_{L})=0$, then all first order deformations $(\alpha,f)\in\Gamma\left(T^{*}\mathcal{F}_{L}\times\mathbb{R}\right)$ are smoothly unobstructed.
	\end{cor}
	\begin{proof}
		If $(\alpha,f)$ is a first order deformation, then $\alpha$ is closed. Since $H^{1}(\mathcal{F}_{L})=0$, it is exact. The same then holds for $\pounds_{X}\alpha$. So the result follows from Proposition \ref{unobstructed}.
	\end{proof}
	
	Corollary \ref{H1} shows in particular that obstructedness is a global issue, since the cohomology group $H^{1}(\mathcal{F}_{L})$ always vanishes locally.

	One may wonder if all first order deformations $(\alpha,f)$ that are smoothly unobstructed arise as in Prop. \ref{unobstructed}. The answer is negative, but it becomes positive if we restrict to first order deformations for which $f\in C^{\infty}(L)$ is nowhere vanishing.   We spell this out in the following remark and lemma.
 	
	\begin{remark}\label{rem:Kurexact}
		First order deformations of the form $(\alpha,0)$, hence $d_{\mathcal{F}_{L}}\alpha=0$, are smoothly unobstructed, but
		in general $\pounds_{X}\alpha$ is not foliated exact.  For instance, consider the log-symplectic manifold $(\mathbb{T}^{2}\times\mathbb{R}^{2},\Pi)$ and Lagrangian submanifold $L:=\mathbb{T}^{2}\times\{(0,0)\}$ as in Example \ref{ex:unob}, for which the foliation $\mathcal{F}_{L}$ is one-dimensional.
		Any $\alpha=g(\theta_1,\theta_2)d\theta_2\in \Omega^1(\mathcal{F}_{L})$ is foliated closed, but in general the integral of $\alpha$ along the fibers of $L\to S^1:(\theta_1,\theta_2)\to \theta_1$ is not independent of $\theta_1$, implying that $\pounds_{X}\alpha$ is not foliated exact. 
	\end{remark}
	
	\begin{lemma}\label{lem:Kurexact}
		Let $(\alpha,f)$ be a first order deformation such that  $Kr\big([(\alpha,f)]\big)=0$.
		Assume that  $f\in C^{\infty}(L)$ is nowhere vanishing. Then 
		$\pounds_{X}\alpha$ is foliated exact.
	\end{lemma}

	\begin{proof}
		The assumption $Kr\big([(\alpha,f)]\big)=0$ 
		is equivalent to $[f\pounds_{X}\alpha]=0$ in $H^{1}_{\gamma}(\mathcal{F}_{L})$ by \eqref{kur}, so it
		implies that there exists $g\in C^{\infty}(L)$ such that
		\[
		f\pounds_{X}\alpha=d_{\mathcal{F}_{L}}g+g\gamma.
		\]
		Since $f$ is nowhere zero, we can divide by $f$ and we obtain
		\begin{align}\label{com}
		\pounds_{X}\alpha&=\frac{1}{f}d_{\mathcal{F}_{L}}g+\frac{g}{f}\gamma\nonumber\\
		&=\frac{1}{f}d_{\mathcal{F}_{L}}g-\frac{g}{f^{2}}d_{\mathcal{F}_{L}}f\nonumber\\
		&=\frac{1}{f}d_{\mathcal{F}_{L}}g+gd_{\mathcal{F}_{L}}\left(\frac{1}{f}\right)\nonumber\\
		&=d_{\mathcal{F}_{L}}\left(\frac{g}{f}\right),
		\end{align}
		using in the second equality that $d^{\gamma}_{\mathcal{F}_{L}}f=0$. This shows that $\pounds_{X}\alpha$ is foliated exact.
	\end{proof}
	 	
{
\begin{remark}
We can phrase the above proof in more conceptual (cohomological) terms, as follows.
Since $f$ is $d^{\gamma}_{\mathcal{F}_{L}}$-closed and nowhere vanishing, it follows that $f^{-1}$ is
$d^{-\gamma}_{\mathcal{F}_{L}}$-closed. By Remark \ref{rem:wedge}	 the wedge product induces a bilinear map
$H^{0}_{-\gamma}(\mathcal{F}_{L})\times  H^{1}_{\gamma}(\mathcal{F}_{L})\to H^{1}(\mathcal{F}_{L})$. This map sends $[f^{-1}]\times [f\pounds_{X}\alpha]$ to $[\pounds_{X}\alpha]$, which has to be zero because  $[f\pounds_{X}\alpha]=0$ and the map is bilinear.
\end{remark}
}

	\subsubsection{\underline{Smoothly unobstructed deformations: the compact case}}
	We now show that for compact connected Lagrangians $(L,\mathcal{F}_{L})$, the condition $H^{1}_{\gamma}(\mathcal{F}_{L})=0$ from Corollary \ref{formal} in fact implies that the deformation problem is smoothly unobstructed. We actually prove  more: one only needs that the Kuranishi map \eqref{kur} is trivial.
	
	\begin{prop}\label{prop:krunob}
		Let $(L^{n},\mathcal{F}_{L})$ be a compact connected Lagrangian submanifold that is contained in the singular locus of a log-symplectic manifold $(M^{2n},Z,\Pi)$. A first order deformation $(\alpha,f)\in\Gamma\left(T^{*}\mathcal{F}_{L}\times\mathbb{R}\right)$ of $L$ is smoothly unobstructed if and only if $Kr\big([(\alpha,f)]\big)=0$.
	\end{prop}
	\begin{proof}
		We only have to prove the backward implication. Let $(\alpha,f)$ be a first order deformation of $L$ with $Kr\big([(\alpha,f)]\big)=0$. We know that either the leaves of $\mathcal{F}_{L}$ are dense, or $(L,\mathcal{F}_{L})$ is the foliation by fibers of a fiber bundle over $S^{1}$. 
		\begin{enumerate}[i)]
			\item First assume that the leaves of $\mathcal{F}_{L}$ are dense. 
			\begin{itemize}
				\item If $\gamma$ is not exact, then $H^{0}_{\gamma}(\mathcal{F}_{L})=\{0\}$ by {Theorem} \ref{H}. Since  $(\alpha,f)$ is a first order deformation, we have that $f\in H^{0}_{\gamma}(\mathcal{F}_{L})=\{0\}$. Therefore $(\alpha,f)=(\alpha,0)$ and a path of Lagrangian sections that prolongs $(\alpha,0)$ is simply given by $s\mapsto(s\alpha,0)$.
				\item Now assume that $\gamma=d_{\mathcal{F}_{L}}k$ is exact. Thanks to (the proof of) Lemma \ref{cohomology} $i)$, we know that $e^{k}f$ is constant on $L$. So either $f\equiv 0$, in which case we conclude that $(\alpha,f)$ is smoothly unobstructed as in the previous bullet point. Or $f$ is nowhere zero, in which case we can use Lemma \ref{lem:Kurexact}. There we showed that $\pounds_{X}\alpha$ is foliated exact, and Proposition \ref{unobstructed} then implies that the first order deformation $(\alpha,f)$ is smoothly unobstructed.		
			\end{itemize}
			\item Now assume that $\mathcal{F}_{L}$ is the fiber foliation of a fiber bundle $p:L\rightarrow S^{1}$. The closed foliated one-form $\pounds_{X}\alpha$ defines a section $\sigma_{\pounds_{X}\alpha}$ of the vector bundle $\mathcal{H}^{1}\rightarrow S^{1}$ via the correspondence \eqref{isom}. By Lemma \ref{function}, we can fix a smooth function $h\in C^{\infty}(L)$ satisfying
			\[
			\left(\left.\pounds_{X}\alpha\right)\right|_{p^{-1}(q)}=d\left(\left.h\right|_{p^{-1}(q)}\right)\hspace{1cm}\forall q\in \mathcal{Z}_{\pounds_{X}\alpha},
			\]
			where we denote   $\mathcal{Z}_{\pounds_{X}\alpha}:=\sigma_{\pounds_{X}\alpha}^{-1}(0)$.
			Mimicking the proof of Proposition \ref{unobstructed}, we claim that the path $s\mapsto(s\alpha,sfe^{sh})$ is a prolongation of $(\alpha,f)$ by Lagrangian sections. So we have to show that 
			\begin{equation}\label{ts}
			d_{\mathcal{F}_{L}}(sfe^{sh})+sfe^{sh}(\gamma-\pounds_{X}s\alpha)=0.
			\end{equation}
			
			To do so, we denote $\mathcal{Z}_{f}:=f^{-1}(0)\subset L$. Recall here that $f\in H^{0}_{\gamma}(\mathcal{F}_{L})$, so that $\mathcal{Z}_{f}$ is a union of fibers of $p:L\rightarrow S^{1}$ (cf. the proof of {Theorem} \ref{H}).
			Clearly, the equality \eqref{ts} holds on $\mathcal{Z}_{f}$. On the other hand, Lemma \ref{lem:Kurexact} implies that $\pounds_{X}\alpha$ is exact on $L\setminus\mathcal{Z}_{f}$. Therefore, $\pounds_{X}\alpha=d_{\mathcal{F}_{L}}h$ on $L\setminus\mathcal{Z}_{f}$, and the computation \eqref{f} in the proof of Prop. \ref{unobstructed} shows that \eqref{ts} holds on $L\setminus\mathcal{Z}_{f}$.
		\end{enumerate}
	\end{proof}

	\begin{remark}
		A crucial point in the proof of Prop. \ref{prop:krunob} is that $h$ is a smooth function defined on the whole of $L$. Its existence is guaranteed by Lemma \ref{function}, a statement about fiber bundles over $S^1$. Due to this, we do not expect the statement of Prop. \ref{prop:krunob} to hold if one removes the compactness assumption on $L$.
	\end{remark}

	We give an algorithmic overview of first order deformations and their obstructedness, for Lagrangians that are compact and connected.
	\begin{enumerate}[i)]
		\item Assume $(L,\mathcal{F}_{L})$ is the foliation by fibers of a fiber bundle $p:L\rightarrow S^{1}$. Fix a smooth function $g\in C^{\infty}(L)$ that is a primitive of $\gamma$ on $\mathcal{Z}_{\gamma}:=\sigma_{\gamma}^{-1}(0)$, as constructed in Lemma \ref{function}. Thanks to {Thm.} \ref{H} i) and its proof, we can characterize first order deformations $(\alpha,f)$ of $L$ by the requirements
		\[
		\begin{cases}
		d_{\mathcal{F}_{L}}\alpha=0\\
		e^{g}f\ \text{is constant on each}\ p\text{-fiber and vanishes on}\  S^1\setminus \mathcal{Z}_{\gamma}
		\end{cases}.
		\]
		By Prop. \ref{prop:krunob}, a first order deformation $(\alpha,f)$ of $L$ is smoothly unobstructed exactly when $Kr\big([(\alpha,f)]\big)=0$. We claim that the latter condition is equivalent to the following:
		\begin{equation}\label{eq:condstar}
		[\pounds_{X}\alpha]=0\in H^1(\mathcal{F}_{L}) \text{ on } L\setminus\mathcal{Z}_{f}.  
		\end{equation}
		Here $\mathcal{Z}_{f}$ denotes the zero locus of $f$, as in the proof of Prop. \ref{prop:krunob}. {In terms of the natural flat connection $\nabla$ of eq. \eqref{connection} on the vector bundle $\mathcal{H}^{1}\to S^1$,  eq. \eqref{eq:condstar} is simply saying that 
$[\alpha]$ is a flat section when restricted to $L\setminus\mathcal{Z}_{f}$.}				

		To see that the two conditions are equivalent, recall that 
		$Kr\big([(\alpha,f)]\big)=0$  
		implies the condition \eqref{eq:condstar}, by Lemma \ref{lem:Kurexact}.
		Conversely, assume that the condition \eqref{eq:condstar} holds. As in the proof of Prop. \ref{prop:krunob}, choose a smooth function $h\in C^{\infty}(L)$ such that $\pounds_{X}\alpha=d_{\mathcal{F}_{L}}h$ on $p^{-1}(\mathcal{Z}_{\pounds_{X}\alpha})$. In particular, this equality holds on $L\setminus\mathcal{Z}_{f}$. From this, we conclude that
		\[
		f\pounds_{X}\alpha=d^{\gamma}_{\mathcal{F}_{L}}(fh).
		\]
		Indeed, on $\mathcal{Z}_{f}$ this equation holds because both sides are zero; on the complement $L\setminus\mathcal{Z}_{f}$ it also holds because $d^{\gamma}_{\mathcal{F}_{L}}(fh)=d^{\gamma}_{\mathcal{F}_{L}}f\cdot h+f\cdot d_{\mathcal{F}_{L}}h=f\pounds_{X}\alpha$. 
		This shows that $[f\pounds_{X}\alpha]=0$ in $H^{1}_{\gamma}(\mathcal{F}_{L})$, 
		which by \eqref{kur} is equivalent to $Kr\big([(\alpha,f)]\big)=0$.
		\item In case $\mathcal{F}_{L}$ has dense leaves, then we distinguish between two types of first order deformations. The first type are the ones of the form $(\alpha,0)$ for closed $\alpha\in\Omega^{1}(\mathcal{F}_{L})$. Clearly, these are smoothly unobstructed.
		
		First order deformations of the second type, those with nonzero second component, can only occur if $\gamma$ is foliated exact, by {Thm.} \ref{H}. They are characterized as the $(\alpha,f)$ for which
		\[
		\begin{cases}
		d_{\mathcal{F}_{L}}\alpha=0\\
		e^{g}f\ \text{is a nonzero constant}
		\end{cases},
		\]
		where $g\in C^{\infty}(L)$ is a primitive of $\gamma$. Such a first order deformation $(\alpha,f)$ is smoothly unobstructed exactly when $[\pounds_{X}\alpha]=0$ in $H^{1}(\mathcal{F}_{L})$: the forward implication follows from Lemma \ref{lem:Kurexact}, and the backward implication from Prop. \ref{unobstructed}.

		Notice that we now showed that the criterion \eqref{eq:condstar} for unobstructedness in the fibration case also holds if $\mathcal{F}_{L}$ has dense leaves: the two types of infinitesimal deformations just described correspond with the extreme cases $L\setminus\mathcal{Z}_{f}=\emptyset$ and $L\setminus\mathcal{Z}_{f}=L$.
	\end{enumerate}
	
	In conclusion, we have proved the following.
	\begin{cor}\label{cor:Krsimple}
		A first order deformation $(\alpha,f)\in\Gamma(T^{*}\mathcal{F}_{L}\times\mathbb{R})$ of a compact, connected Lagrangian $L$ is smoothly unobstructed exactly when 
		\begin{equation}\label{lie}
		[\pounds_{X}\alpha]=0\in H^{1}(\mathcal{F}_{L})\ \text{on}\ L\setminus\mathcal{Z}_{f}.
		\end{equation}
		Here $\mathcal{Z}_{f}$ denotes the zero locus of $f$.
	\end{cor}
	Given a first order deformation $(\alpha,f)$, the  condition \eqref{lie} is equivalent with $\alpha$ extending to a closed one-form on $ L\setminus\mathcal{Z}_{f}$, by the argument of Remark \ref{rem:geominter} i).
	Therefore the condition \eqref{lie} is independent of the data $(X,\gamma)$ coming from the modular vector field. {We remark that, as seen just before Corollary \ref{cor:Krsimple}, the condition that $(\alpha,f)$ is a first order deformation involves $\gamma$ but not  $X$.}

	\begin{ex} \label{ex:T2obstr}
		Consider the manifold $\mathbb{T}^{2}\times\mathbb{R}^{2}$, regarded as a trivial vector bundle over $\mathbb{T}^{2}$. Denote its coordinates by $(\theta_{1},\theta_{2}, {\xi_1}, {\xi_2})$. Let $Z:=\mathbb{T}^{2}\times\mathbb{R}=\{\xi_1=0\}$ and  $L:=\mathbb{T}^{2}\times\{(0,0)\}$. 
		\begin{enumerate}[i)]
			\item Any orientable log-symplectic structure with singular locus $Z$ so that $L$ is Lagrangian with induced foliation $T\mathcal{F}_{L}=\text{Ker}(d\theta_{1})$, up to Poisson diffeomorphism, looks as follows nearby $L$:
			\[
			\Pi=V\wedge  {\xi_1}\partial_{\xi_1}+\partial_{\theta_{2}}\wedge\partial_{\xi_2},
			\]
			where	
			$$V=
			g_{X}(\theta_1)\partial_{\theta_{1}}+g_{\gamma}(\theta_1)\partial_{\xi_2}$$
			for some  function $g_{\gamma}\in C^{\infty}(S^1)$ and some  nowhere vanishing function $g_{X}\in C^{\infty}(S^1)$. Here we use Corollary \ref{normal}, Remark \ref{rem:freedom} and Corollary \ref{cor:isocoho} along with Remark \ref{rem:codim1PoisCoho}  ii).
			
			We have $\gamma=g_{\gamma}(\theta_1) d\theta_{2}$, and a function on $L$ satisfying the properties of Lemma \ref{function} is the constant function zero.
			Hence first order deformations are given by pairs $(\alpha,f)$,  subject to the condition that $f=f(\theta_1)$ and $f\cdot g_{\gamma}=0$. 
			
			To see when such a first order deformation is unobstructed, we apply Corollary \ref{cor:Krsimple}. In the case at hand, since the fibers of $p:L\rightarrow S^{1}$ are circles and thanks to Stokes' theorem, the condition \eqref{lie} can be rephrased as:
			$$\text{
				the function $q\mapsto \int_{p^{-1}(q)} \alpha$ is locally constant on $p(L\setminus\mathcal{Z}_{f})\subset S^1$.}$$

			For instance, in case $g_{\gamma}=0$ (as in Ex. \ref{ex:unob}), any pair $(\alpha,f)$ with $f=f(\theta_1)$ is a first other deformation. Such a pair is unobstructed exactly when, writing $\alpha=a(\theta_1,\theta_2)d\theta_2$,  the expression $$ \int_{\{\theta_1\}\times S^1} a(\theta_1,\theta_2)d\theta_2$$
			is constant on connected components of $p(L\setminus\mathcal{Z}_{f})$.

			\item Now let $\lambda\in\mathbb{R}\setminus\mathbb{Q}$ be a generic (i.e. not Liouville) irrational number. Any orientable log-symplectic structure with singular locus $Z$ so that $L$ is Lagrangian with induced foliation $T\mathcal{F}_{L}=\text{Ker}(d\theta_{1}-\lambda d\theta_{2})$ is Poisson diffeomorphic around $L$ with
			\begin{equation}\label{K}
			\Pi=(C\partial_{\theta_{1}}+K\partial_{\xi_{2}})\wedge\xi_{1}\partial_{\xi_{1}}+(\lambda\partial_{\theta_{1}}+\partial_{\theta_{2}})\wedge\partial_{\xi_{2}},
			\end{equation}
			for some $C,K\in\RR$ with $C$ nonzero.
			This follows from a similar reasoning as above, now using that 
			\[
			\mathfrak{X}(L)^{\mathcal{F}_{L}}/\Gamma(T\mathcal{F}_{L})\cong H^{0}(\mathcal{F}_{L})=\mathbb{R}\hspace{0.5cm}\text{and}\hspace{0.5cm}H^{1}(\mathcal{F}_{L})=\mathbb{R}[d\theta_{2}].
			\]
			Note that $\gamma=Kd\theta_{2}$ is exact if and only if $K=0$. Therefore, first order deformations are given by $(\alpha,0)$ if $K\neq 0$ and $(\alpha,c)$ if $K=0$, with $c\in\RR$. Clearly, the Lie derivative along $X=C\partial_{\theta_{1}}$ acts trivially in cohomology, since $H^{1}(\mathcal{F}_{L})=\RR [d\theta_{2}]$. Therefore, all first order deformations of $L$ are smoothly unobstructed, by Corollary \ref{cor:Krsimple}.

			The situation is different when $\lambda\in\RR\setminus\mathbb{Q}$ is a Liouville number. Disregarding trivially unobstructed first order deformations of the form $(\alpha,0)$, {Thm.} \ref{H} ii) implies that the ones with nonzero second component can only occur for log-symplectic structures that are isomorphic around $L$ to the following model:
			\[
			\Pi=C\partial_{\theta_{1}}\wedge\xi_{1}\partial_{\xi_{1}}+(\lambda\partial_{\theta_{1}}+\partial_{\theta_{2}})\wedge\partial_{\xi_{2}},
			\]
			where $C\in\mathbb{R}_{0}$. Notice that $H^{1}(\mathcal{F}_{L})$ is now infinite dimensional, and that the Lie derivative along $X=C\partial_{\theta_{1}}$ no longer acts trivially in cohomology, which is a direct consequence of (the proof of) Lemma \ref{dense}. This shows that there exist obstructed first order deformations.
		\end{enumerate}
	\end{ex}

	\subsection{Equivalences and rigidity of deformations}\label{subsec:eqrig}
	\leavevmode
	\vspace{0.1cm}
	
	We now consider two natural equivalence relations on the space of Lagrangian deformations: equivalence by Hamiltonian diffeomorphisms and equivalence by Poisson isotopies. We show that the action by Hamiltonian diffeomorphisms agrees with the gauge action of the DGLA that governs the deformation problem. We also discuss rigidity of Lagrangians, both for Hamiltonian and Poisson equivalence.
	
	\subsubsection{\underline{Hamiltonian isotopies}}
	We showed in \S\ref{corres} that the graph of  $(\alpha,f)\in\Gamma(T^{*}\mathcal{F}_{L}\times\mathbb{R})$ defines a Lagrangian submanifold of $(U,\widetilde{\Pi})$ exactly when $(\alpha,f)$ is a Maurer-Cartan element of the DGLA $\big(\Gamma(\wedge^{\bullet}(T^{*}\mathcal{F}_{L}\times\mathbb{R})),d,[\![\cdot,\cdot]\!]\big)$. So if we write for short
	\[
	\text{Def}_{U}(L):=\big\{(\alpha,f)\in\Gamma(U):\ \text{graph}(\alpha,f)\ \text{is Lagrangian inside}\ \big(U,\widetilde{\Pi}\big)\big\} 
	\]
	and
	\[
	\text{MC}_{U}\big(\Gamma(\wedge^{\bullet}(T^{*}\mathcal{F}_{L}\times\mathbb{R}))\big):=\left\{(\alpha,f)\in MC\big(\Gamma(\wedge^{\bullet}(T^{*}\mathcal{F}_{L}\times\mathbb{R}))\big):\ \text{graph}(\alpha,f)\subset U \right\},
	\]
	then we have a correspondence
	\begin{equation}\label{1:1}
	\text{Def}_{U}(L)\overset{1:1}\longleftrightarrow \text{MC}_{U}\big(\Gamma(\wedge^{\bullet}(T^{*}\mathcal{F}_{L}\times\mathbb{R}))\big).
	\end{equation}
	We now define equivalence relations on both sides of \eqref{1:1} and we show that they agree under this correspondence. We closely follow the exposition in \cite{equivalences}. There one considers equivalences of coisotropic submanifolds in symplectic geometry, but most of their results remain valid in the more general setting of fiberwise entire Poisson structures.
	
	\begin{defi}
		\begin{enumerate}[i)]
			\item Two Lagrangian sections $(\alpha_{0},f_{0})$ and $(\alpha_{1},f_{1})$ in $\text{Def}_{U}(L)$ are \emph{Hamiltonian equivalent} if they are interpolated by a smooth family $(\alpha_{s},f_{s})$ of Lagrangian sections in $\text{Def}_{U}(L)$ that is generated by a (locally defined) Hamiltonian isotopy. In other words, there exists a time-dependent Hamiltonian vector field $X_{H_{s}}$ on $U$ such that the associated isotopy $\phi_{s}$ maps $\text{graph}(\alpha_{0},f_{0})$ to $\text{graph}(\alpha_{s},f_{s})$, for all $s\in[0,1]$.
			\item Two Maurer-Cartan elements $(\alpha_{0},f_{0}),(\alpha_{1},f_{1})\in \text{MC}_{U}\big(\Gamma(\wedge^{\bullet}(T^{*}\mathcal{F}_{L}\times\mathbb{R}))\big)$ are \emph{gauge equivalent} if they are interpolated by a smooth family $\{(\alpha_{s},f_{s})\}_{s\in[0,1]}$ of sections whose graph lies inside $U$, and there exists a smooth family $\{g_{s}\}_{s\in[0,1]}$ of functions on $L$ such that
			\begin{align}\label{gauge}
			\frac{d}{ds}(\alpha_{s},f_{s})&=[\![(g_{s},0),(\alpha_{s},f_{s})]\!]-d(g_{s},0)\nonumber\\
			&=\big(d_{\mathcal{F}_{L}}g_{s},f_{s}\pounds_{X}g_{s}\big).
			\end{align}
		\end{enumerate}
	\end{defi}

	\begin{remark}\label{Gaction}
		By solving the flow equation \eqref{gauge}, we obtain an explicit description for the gauge action of the DGLA. Namely, a path of degree zero elements $(g_{s},0)$ acts on a Maurer-Cartan element $(\alpha_{0},f_{0})$, which yields a path of Maurer-Cartan elements $(\alpha_{s},f_{s})$ given by
		\begin{equation}\label{gaugeaction}
		(\alpha_{s},f_{s})=\left(\alpha_{0}+d_{\mathcal{F}_{L}}\left(\int_{0}^{s}g_{u}du\right),f_{0}\exp\left(\pounds_{X}\int_{0}^{s}g_{u}du\right)\right).
		\end{equation}
	\end{remark}
	
	We rewrite the gauge equivalence relation in more geometric terms.
	
	\begin{lemma}\label{gaugegeometric}
		Two Maurer-Cartan elements $(\alpha_{0},f_{0}),(\alpha_{1},f_{1})\in \text{MC}_{U}\big(\Gamma(\wedge^{\bullet}(T^{*}\mathcal{F}_{L}\times\mathbb{R}))\big)$ are gauge equivalent if and only if they are interpolated by a smooth family $\{(\alpha_{s},f_{s})\}_{s\in[0,1]}$ of sections whose graph lies inside $U$, and there exists a smooth family $\{g_{s}\}_{s\in[0,1]}$ of functions on $L$ such that
		\begin{equation}\label{ham}
		\frac{d}{ds}(\alpha_{s},f_{s})=\left.X_{pr^{*}g_{s}}\right|_{graph(\alpha_{s},f_{s})}.
		\end{equation}
		Here $pr:U\subset T^{*}\mathcal{F}_{L}\times\mathbb{R}\rightarrow L$ denotes the bundle projection, and we see $\eqref{ham}$ as an equality of sections of the vertical bundle restricted to $graph(\alpha_{s},f_{s})$. 
	\end{lemma}
	\begin{proof}
		We compute the Hamiltonian vector field $X_{pr^{*}g_{s}}$. As before, let $p:T^{*}\mathcal{F}_{L}\rightarrow L$ denote the bundle projection. We obtain
		\begin{align}\label{vertical}
		X_{pr^{*}g_{s}}&=\left((V_{vert}+V_{lift})\wedge t\partial_{t}+\Pi_{can}\right)^{\sharp}(dpr^{*}g_{s})\nonumber\\
		&=pr^{*}(\pounds_{X}g_{s})t\partial_{t}+\Pi_{can}^{\sharp}(p^{*}dg_{s}),
		\end{align}
		and therefore
		\[
		\left.X_{pr^{*}g_{s}}\right|_{graph(\alpha_{s},f_{s})}=pr^{*}(f_{s}\pounds_{X}g_{s})\partial_{t}+\Pi_{can}^{\sharp}(p^{*}dg_{s}).
		\]
		The section of $T^{*}\mathcal{F}_{L}\times\mathbb{R}$ corresponding with this vertical fiberwise constant vector field is $\big(d_{\mathcal{F}_{L}}g_{s},f_{s}\pounds_{X}g_{s}\big)\in\Gamma(T^{*}\mathcal{F}_{L}\times\mathbb{R})$, in agreement with \eqref{gauge}. This proves the lemma.
	\end{proof}
	
	We need some technical results that appeared in \cite{equivalences}. We state them here for convenience.
	
	\begin{lemma}\label{forward}
		Let $A\rightarrow M$ be a vector bundle with vertical bundle $V$. Let $X_{s}$ be one-parameter family of vector fields on $A$ with flow $\phi_{s}$, and let $\tau_{0}$ be a section of $A$.
		\begin{enumerate}[i)]
			\item If $\tau_{s}$ is a one-parameter family of sections of $A$ such that
			$
			\text{graph}(\tau_{s})=\phi_{s}(\text{graph}(\tau_{0}))
			$
			holds for all $s\in[0,1]$, then $\tau_{s}$ satisfies the equation
			\begin{equation*}
			\frac{d}{ds}\tau_{s}=P_{\tau_{s}}X_{s}\hspace{0.5cm}\forall s\in[0,1].
			\end{equation*}
			Here $P_{\tau_{s}}$ means vertical projection with respect to $TA|_{graph(\tau_{s})}=Tgraph(\tau_{s})\oplus V|_{graph(\tau_{s})}.$
			\item Conversely, assume that the integral curves of $X_{s}$ starting at points of $\text{graph}(\tau_{0})$ exist for all times $s\in[0,1]$, and suppose that $\tau_{s}$ is a one-parameter family of sections of $A$ satisfying
			\[
			\frac{d}{ds}\tau_{s}=P_{\tau_{s}}X_{s}\hspace{0.5cm}\forall s\in[0,1].
			\]
			Then the family of submanifolds $\text{graph}(\tau_{s})$ coincides with $\phi_{s}(\text{graph}(\tau_{0}))$ for all $s\in[0,1]$.
		\end{enumerate}
	\end{lemma}
	
	Making some minor modifications to the proofs of \cite[Proposition 3.18]{equivalences} and \cite[Proposition 3.19]{equivalences}, 
	 we can show that Hamiltonian equivalence coincides with gauge equivalence.
	
	\begin{prop}\label{prop:ham}
		The bijection between Lagrangian sections and Maurer-Cartan elements
		\[
		\text{Def}_{U}(L)\rightarrow \text{MC}_{U}\big(\Gamma(\wedge^{\bullet}(T^{*}\mathcal{F}_{L}\times\mathbb{R}))\big):(\alpha,f)\mapsto(\alpha,f)
		\]
		descends to a bijection between $\text{Def}_{U}(L)/{\sim_{\text{Ham}}}$ and $\text{MC}_{U}\big(\Gamma(\wedge^{\bullet}(T^{*}\mathcal{F}_{L}\times\mathbb{R}))\big)/{\sim_{\text{gauge}}}$.
	\end{prop}
	\begin{proof}
		First assume that $(\alpha_{0},f_{0}),(\alpha_{1},f_{1})\in\text{Def}_{U}(L)$ are Hamiltonian equivalent. Then they are interpolated by a smooth family of sections $(\alpha_{s},f_{s})\in\text{Def}_{U}(L)$ generated by the flow $\phi_{s}$ of a time-dependent Hamiltonian vector field $X_{H_{s}}\in \mathfrak{X}(U)$. Part $i)$ of Lemma \ref{forward} then implies that
		\begin{equation}\label{velocity}
		\frac{d}{ds}(\alpha_{s},f_{s})=P_{(\alpha_{s},f_{s})}X_{H_{s}}
		\end{equation}
		for all $s\in[0,1]$. Define $g_{s}:=H_{s}\circ(\alpha_{s},f_{s})\in C^{\infty}(L)$ and observe that $H_{s}-pr^{*}g_{s}$ vanishes along $\text{graph}(\alpha_{s},f_{s})$. Because $\text{graph}(\alpha_{s},f_{s})$ is coisotropic, this implies that the Hamiltonian vector field $X_{H_{s}-pr^{*}g_{s}}=X_{H_{s}}-X_{pr^{*}g_{s}}$ is tangent to $\text{graph}(\alpha_{s},f_{s})$. Consequently, the equality \eqref{velocity} becomes
		\[
		\frac{d}{ds}(\alpha_{s},f_{s})=P_{(\alpha_{s},f_{s})}X_{pr^{*}g_{s}}=\left.X_{pr^{*}g_{s}}\right|_{graph(\alpha_{s},f_{s})},
		\]
		where we also used that $X_{pr^{*}g_{s}}$ is vertical (which is clear from the expression \eqref{vertical}). By Lemma \ref{gaugegeometric}, we conclude that $(\alpha_{0},f_{0})$ and $(\alpha_{1},f_{1})$ are gauge equivalent.
		
		Conversely, assume that $(\alpha_{0},f_{0}),(\alpha_{1},f_{1})\in\text{MC}_{U}\big(\Gamma(\wedge^{\bullet}(T^{*}\mathcal{F}_{L}\times\mathbb{R}))\big)$ are gauge equivalent. By Lemma \ref{gaugegeometric}, this means that they are interpolated by a smooth family of sections $(\alpha_{s},f_{s})$ inside $U$, such that
		\[
		\frac{d}{ds}(\alpha_{s},f_{s})=\left.X_{pr^{*}g_{s}}\right|_{graph(\alpha_{s},f_{s})}=P_{(\alpha_{s},f_{s})}X_{pr^{*}g_{s}}\hspace{0.5cm}\forall s\in[0,1],
		\]
		for a smooth family of functions $g_{s}\in C^{\infty}(L)$. In particular, the integral curve of $X_{pr^{*}g_{s}}$ starting at a point $(\alpha_{0},f_{0})(p)\in\text{graph}(\alpha_{0},f_{0})$ is defined up to time $1$, and is given by $(\alpha_{s},f_{s})(p)$ for $s\in[0,1]$. Part $ii)$ of Lemma \ref{forward} gives $\phi_{s}(\text{graph}(\alpha_{0},f_{0}))=\text{graph}(\alpha_{s},f_{s})$ for all $s\in[0,1]$, where $\phi_{s}$ is the flow of $X_{pr^{*}g_{s}}$. This shows that $(\alpha_{0},f_{0})$ and $(\alpha_{1},f_{1})$ are Hamiltonian equivalent.
	\end{proof}
	
	\begin{remark}
		The above proof is almost identical to the one presented in \cite{equivalences}. The main difference is that in \cite{equivalences}, one needs to impose compactness on the coisotropic submanifold to obtain the implication ``gauge equivalence $\Rightarrow$ Hamiltonian equivalence'', as otherwise the flow lines of $X_{pr^{*}g_{s}}$ need not be defined for long enough time. Since in our setting Hamiltonian vector fields of basic functions are vertical, we don't need this additional assumption.
	\end{remark}
	
	As a consequence, we obtain that the formal tangent space at zero to the moduli space $\mathcal{M}_{U}^{Ham}(L):=\text{Def}_{U}(L)/{\sim_{Ham}}$ can be identified with the first cohomology group of the differential graded Lie algebra $\big(\Gamma\left(\wedge^{\bullet}\left(T^{*}\mathcal{F}_{L}\times\mathbb{R}\right)\right),d,[\![\cdot,\cdot]\!]\big)$:
	\begin{equation}\label{formalmoduli}
	T_{[0]}\mathcal{M}_{U}^{Ham}(L)=H^{1}(\mathcal{F}_{L})\oplus H^{0}_{\gamma}(\mathcal{F}_{L}).
	\end{equation}
	Indeed, if $(\alpha_{s},f_{s})$ is a path of Lagrangian deformations of $L$, then $\frac{d}{ds}|_{s=0}(\alpha_{s},f_{s})$ is closed with respect to the differential $d$ of the DGLA. Moreover, if the path $(\alpha_{s},f_{s})$ is generated by the flow of a time-dependent Hamiltonian vector field, then $(\alpha_{s},f_{s})$ is obtained by gauge transforming the zero section, as we just proved. The expression \eqref{gaugeaction} then shows that $\frac{d}{ds}|_{s=0}(\alpha_{s},f_{s})$ is of the form $(d_{\mathcal{F}_{L}}g,0)$ for $g\in C^{\infty}(L)$. This proves the assertion \eqref{formalmoduli}.

	\subsubsection{\underline{Smoothness of the moduli space by Hamiltonian isotopies}}\label{subsubsec:hammod}
	
	In general, the moduli space $\mathcal{M}_{U}^{Ham}(L)$ is by no means smooth, since the formal tangent spaces at different points can be drastically different. For instance, let us look again at Example \ref{ex:unob}, where we considered $(\mathbb{T}^{2}\times\mathbb{R}^{2},\theta_{1},\theta_{2},\xi_{1},\xi_{2})$ with log-symplectic structure
	\[
	\Pi=\partial_{\theta_{1}}\wedge\xi_{1}\partial_{\xi_{1}}+\partial_{\theta_{2}}\wedge\xi_{2}
	\]
	and Lagrangian $L=\mathbb{T}^{2}\times\{(0,0)\}$. The induced foliation on $L$ is the fiber foliation of $(L,\theta_{1},\theta_{2})\rightarrow(S^{1},\theta_{1})$. Since $\gamma=0$, we get for any nonzero constant $c\in\RR$ a Lagrangian section $(0,c)\in\Gamma(T^{*}\mathcal{F}_{L}\times\RR)$ whose graph lies outside the singular locus. Hence, by symplectic geometry, we have
	\[
	T_{[(0,c)]}\mathcal{M}_{U}^{Ham}(L)\cong H^{1}(graph(0,c))\cong H^{1}(L)=\RR^{2},
	\]
	which is finite dimensional.
	On the other hand, we have
	\[
	T_{[0]}\mathcal{M}_{U}^{Ham}(L)=H^{1}(\mathcal{F}_{L})\oplus H_{\gamma}^{0}(\mathcal{F}_{L})\cong H^{1}(\mathcal{F}_{L})\oplus H^{0}(\mathcal{F}_{L})\cong C^{\infty}(S^{1})\oplus C^{\infty}(S^{1}),
	\]
	which is infinite dimensional.
	
	\bigskip
	On the other hand, there are instances in which the moduli space is locally smooth.
	
	Suppose a Lagrangian submanifold $L^{n}$ contained in the singular locus $Z$
	has the property 
	that $\mathcal{C}^{1}$-small Lagrangian deformations of $L$ stay inside $Z$. 
	This means that the $\mathcal{C}^{1}$-small deformations are precisely the graphs of 
	$\mathcal{C}^{1}$-small elements of $\Omega^1_{cl}(\cF_L)$. Then $\mathcal{M}_{U}^{Ham}(L)$ is naturally isomorphic to an open neighborhood of the origin in $H^1(\cF_L)$, by Corollary \ref{moduliZ}.
	In particular, $\mathcal{M}_{U}^{Ham}(L)$ is smooth. We present two classes of examples.

	\begin{itemize}
		\item [i)]  
		A class of Lagrangians $L$ as above are those satisfying the assumptions of Corollary \ref{constrained}. In that case $\mathcal{M}_{U}^{Ham}(L)$ is infinite-dimensional. Indeed, recall that
		$H^{1}(\mathcal{F}_{L})\cong\Gamma(\mathcal{H}^{1})$; if this was finite-dimensional, then $\mathcal{H}^{1}$ would be of rank zero, which implies that $H^{1}(\mathcal{F}_{L})=0$. Then $\gamma$ would be exact, which is impossible under the  assumptions of Corollary \ref{constrained}.
		
		\item [ii)] Another class of Lagrangians $L$ as above are those that are $\mathcal{C}^1$-rigid under Poisson equivalences (see \S \ref{subsubsec:poisrig} later on), since Poisson diffeomorphisms of the ambient log-symplectic manifold necessarily preserve $Z$. In that case $\mathcal{M}_{U}^{Ham}(L)$ is finite-dimensional by Lemma \ref{finitedim}, assuming $L$ is compact and connected. We exhibit concrete examples of such $L$ in Example \ref{ex:C^1rigid}. Notice that Proposition \ref{rigidity} as stated does not quite 
		provide examples, since it makes a statement only about $\mathcal{C}^{\infty}$-small deformations.
	\end{itemize}

	\subsubsection{\underline{Rigidity and Hamiltonian isotopies}}
	
	At this point, we would like to address some rigidity phenomena. A Lagrangian $L$ is called \emph{infinitesimally rigid} under Hamiltonian equivalence if the formal tangent space $T_{[0]}\mathcal{M}_{U}^{Ham}(L)$ is zero. We call a Lagrangian $L$ \emph{rigid} under Hamiltonian equivalence if small deformations of $L$ are Hamiltonian equivalent with $L$. It turns out however that Hamiltonian equivalence is too restrictive for rigidity purposes: there are  no Lagrangians that are infinitesimally rigid. Indeed, if the formal tangent space $T_{[0]}\mathcal{M}_{U}^{Ham}(L)=H^{1}(\mathcal{F}_{L})\oplus H^{0}_{\gamma}(\mathcal{F}_{L})$ is zero, then the triviality of the first summand  implies that $\gamma$ is foliated exact. But then $H^{0}(\mathcal{F}_{L})=H^{0}_{\gamma}(\mathcal{F}_{L})=\{0\}$ by Proposition \ref{cohomology} $i)$, which is impossible. This is a motivation to look at a more flexible notion of equivalence.
	
	\subsubsection{\underline{Poisson isotopies}}
	
	We will use flows of Poisson vector fields instead of Hamiltonian vector fields to obtain a less restrictive equivalence relation on the space of Lagrangian deformations of $L$.
	
	\begin{defi}
		We call two Lagrangian sections $(\alpha_{0},f_{0})$ and $(\alpha_{1},f_{1})$ in $\text{Def}_{U}(L)$ \emph{Poisson equivalent} if they are interpolated by a smooth family $(\alpha_{s},f_{s})$ of Lagrangian sections in $\text{Def}_{U}(L)$ that is generated by a (locally defined) Poisson isotopy. In other words, there exists a time-dependent Poisson vector field $Y_{s}$ on $U$ such that the associated isotopy $\phi_{s}$ maps $\text{graph}(\alpha_{0},f_{0})$ to $\text{graph}(\alpha_{s},f_{s})$, for all $s\in[0,1]$.
	\end{defi}
	
	We denote the moduli space $\text{Def}_{U}(L)/{\sim_{Poiss}}$ of Lagrangian deformations under Poisson equivalence by $\mathcal{M}_{U}^{Poiss}(L)$. In order to study rigidity under Poisson equivalence, we want to compute the formal tangent space $T_{[0]}\mathcal{M}_{U}^{Poiss}(L)$, as done in \eqref{formalmoduli} for Hamiltonian equivalence. We now quotient first order deformations of $L$ by elements of the form $\frac{d}{ds}|_{s=0}(\alpha_{s},f_{s})$, where $(\alpha_{s},f_{s})$ is generated by the flow of a time-dependent Poisson vector field $Y_{s}\in\mathfrak{X}(U)$. Lemma \ref{forward} $i)$ implies that
	\begin{equation}\label{PY0}
	\left.\frac{d}{ds}\right|_{s=0}(\alpha_{s},f_{s})=P(Y_{0}),
	\end{equation}
	where $P:\mathfrak{X}(U)\rightarrow\Gamma(T^{*}\mathcal{F}_{L}\times\mathbb{R})$ is the restriction to $L$ composed with the vertical projection induced by the splitting $\big(T(T^{*}\mathcal{F}_{L}\times\mathbb{R})\big)|_{L}=TL\oplus (T^{*}\mathcal{F}_{L}\times\mathbb{R})$. So we have to take a closer look at (vertical components of) Poisson vector fields on {$U\subset T^{*}\mathcal{F}_{L} \times\mathbb{R}$}.

	\begin{lemma}\label{poiscoh}
		Given the Poisson structure $\widetilde{\Pi}=V\wedge t\partial_{t}+\Pi_{can}$ on $U\subset T^{*}\mathcal{F}_{L}\times\mathbb{R}$, the following map is an isomorphism:
		\[
		H^{1}(L)\oplus H^{0}(L)\rightarrow H^{1}_{\widetilde{\Pi}}(U):\big([\xi],g\big)\mapsto \left[\widetilde{\Pi}^{\sharp}(pr^{*}\xi)+(pr^{*}g)V\right],
		\]
		where $pr:U\rightarrow L$ is the projection.
	\end{lemma}
We remark that the existence of the isomorphism follows from known facts:
$L$ is a deformation retract of   $U$ and of $U\cap T^{*}\mathcal{F}_{L}$,
and both cohomology groups appearing above are isomorphic to the first $b$-cohomology group of the pair $(U,U\cap T^{*}\mathcal{F}_{L})$, by \cite{Melrose}  and \cite[Prop. 1]{defslog} respectively.
	
	\begin{proof}
		Clearly, the map is well-defined. For injectivity, assume  $\widetilde{\Pi}^{\sharp}(pr^{*}\xi)+(pr^{*}g)V=\widetilde{\Pi}^{\sharp}(dh)$ for some $h\in C^{\infty}(U)$. Restricting to {$W:=U\cap \{t=0\}$}, this implies that $\Pi_{can}^{\sharp}(p^{*}\xi)+(p^{*}g)V$ is tangent to the symplectic leaves, where $p:{W}\rightarrow L$ is the projection. Since $V$ is transverse to the leaves, we get that $p^{*}g=0$, and therefore $g=0$. This means that  $\widetilde{\Pi}^{\sharp}(pr^{*}\xi)=\widetilde{\Pi}^{\sharp}(dh)$, and since $\widetilde{\Pi}$ is invertible away from {$W\subset U$}, we get that $pr^{*}\xi=dh$ on {$U\setminus W$}. By continuity, $pr^{*}\xi=dh$ on all of $U$, so that $\xi=d(i_{L}^{*}h)$ is exact.
		
		To prove surjectivity, we use some $b$-symplectic geometry. The $b$-symplectic form $\omega$ on $U$ obtained by inverting $\widetilde{\Pi}$ reads \cite[Proposition 4.1.2]{Osorno}
		\[
		\omega=-\widetilde{\Pi}^{-1}=q^{*}\theta\wedge\frac{dt}{t}+q^{*}\eta,
		\]
		where $q:U\rightarrow {W}$ is the projection and $(\theta,\eta)\in\Omega^{1}({W})\times\Omega^{2}({W})$ is the cosymplectic structure corresponding with the pair $(\Pi_{can},V)$. If $Y\in\mathfrak{X}(U)$ is a Poisson vector field, then $Y$ is tangent to ${W}$, so we can evaluate
		\begin{equation}\label{omega}
		\omega^{\flat}(Y)=q^{*}\langle \theta,Y|_{{W}}\rangle\frac{dt}{t}+\left[\left(\langle q^{*}\theta,Y\rangle-q^{*}\langle \theta,Y|_{{W}}\rangle\right)\frac{dt}{t}-\left\langle\frac{dt}{t},Y\right\rangle q^{*}\theta+\iota_{Y}q^{*}\eta\right],
		\end{equation}
		which is a closed $b$-one form on $U$. Note indeed that the summand between square brackets is a smooth de Rham form since $q^{*}\langle \theta,Y|_{{W}}\rangle-\langle q^{*}\theta,Y\rangle$ vanishes along the hypersurface {${W}=U\cap\{t=0\}$} and $Y$ is tangent to it. By the Mazzeo-Melrose isomorphism \cite{miranda2}, \cite{defslog}
		\[
		{}^{b}H^{1}(U)\rightarrow H^{1}(U)\oplus H^{0}({W}):\left[q^{*}(h)\frac{dt}{t}+\beta\right]\mapsto\left([\beta],h\right),
		\]
		we know that the one-form
		\[
		\beta:=\left(\langle q^{*}\theta,Y\rangle-q^{*}\langle \theta,Y|_{{W}}\rangle\right)\frac{dt}{t}-\left\langle\frac{dt}{t},Y\right\rangle q^{*}\theta+\iota_{Y}q^{*}\eta
		\]
		appearing in \eqref{omega} is closed, and that $h:=\langle \theta,Y|_{{W}}\rangle$ is locally constant. We now have
		\begin{align}\label{Y}
		Y&=-\widetilde{\Pi}^{\sharp}(\omega^{\flat}(Y))\nonumber\\
		&=q^{*}\langle \theta,Y|_{{W}}\rangle V+\widetilde{\Pi}^{\sharp}\left(\left(q^{*}\langle \theta,Y|_{{W}}\rangle-\langle q^{*}\theta,Y\rangle\right)\frac{dt}{t}+\left\langle\frac{dt}{t},Y\right\rangle q^{*}\theta-\iota_{Y}q^{*}\eta\right)\nonumber\\
		&=(q^{*}h)V+\widetilde{\Pi}^{\sharp}(-\beta)
		\end{align}
		We make sure that the neighborhood $U$ is such that the map $i_{L}\circ pr:U\rightarrow U$ induces the identity map in cohomology. This means that
		$q^{*}h=pr^{*}(i_{L}^{*}q^{*}h)$ and
		$\beta-pr^{*}(i_{L}^{*}\beta)$ is exact. So if we put $\xi:=-i_{L}^{*}\beta$ and $g:=i_{L}^{*}q^{*}h$, then it follows from \eqref{Y} that
		\[
		[Y]=\left[\widetilde{\Pi}^{\sharp}(pr^{*}\xi)+(pr^{*}g)V\right]\in H^{1}_{\widetilde{\Pi}}(U).\qedhere
		\]
	\end{proof}
	
	\begin{prop}\label{prop:poisequiv}
		The formal tangent space $T_{[0]}\mathcal{M}_{U}^{Poiss}(L)$ is given by
		\begin{equation}\label{formaltangent}
		T_{[0]}\mathcal{M}_{U}^{Poiss}(L)=\frac{\Omega^{1}_{cl}(\mathcal{F}_{L})}{\text{Im}\left(r:\Omega^{1}_{cl}(L)\rightarrow\Omega^{1}_{cl}(\mathcal{F}_{L})\right)+H^{0}(L)\cdot \gamma}\oplus H^{0}_{\gamma}(\mathcal{F}_{L}),
		\end{equation}
		where the map $r$ is restriction of closed one-forms on $L$ to the leaves of $\mathcal{F}_{L}$.
	\end{prop}
	\begin{proof}
		Throughout, for all vector bundles appearing, we denote by $P$ the map that restricts vector fields to the zero section, and then takes their vertical component. Because of \eqref{PY0}, we have to show that the denominator appearing in \eqref{formaltangent} is equal to
		\[
		\left\{P(Y_{0}): Y_{s}\in\mathfrak{X}(U)\ \text{time-dependent Poisson vector field}\right\}.
			\]
Notice that the above set lies in $\Omega^{1}(\mathcal{F}_{L})$, since all		
Poisson vector fields on $U$ are tangent to $W:=U\cap \{t=0\}$.	
		For one inclusion, let $Y_{0}$ be a Poisson vector field on $U$. Using 
		the fact that $Y_{0}$ is tangent to $W$ and
		Lemma \ref{poiscoh}, we have
		\begin{align}\label{PY}
		P(Y_{0})&=P\left(Y_{0}|_{{W}}\right)\nonumber\\
		&=P\left(\Pi_{can}^{\sharp}(p^{*}\xi)+(p^{*}g)V+\Pi_{can}^{\sharp}(dh)\right)
		\end{align}
		for some $\xi\in\Omega^{1}_{cl}(L), g\in H^{0}(L)$ and $h\in C^{\infty}({W})$. {Here $p:W\rightarrow L$ is the projection.} Now note that
		\[
		P\left(\Pi_{can}^{\sharp}(dh)\right)=P\left(\Pi_{can}^{\sharp}(p^{*}di_{L}^{*}h)\right).
		\]
		Indeed, since $L$ is coisotropic and $h-p^{*}i_{L}^{*}h$ vanishes along $L$, we have that $\Pi_{can}^{\sharp} (d(h-p^{*}i_{L}^{*}h))$ is tangent to $L$. So \eqref{PY} becomes
		\begin{align*}
		P(Y_{0})&=P\left(\Pi_{can}^{\sharp}(p^{*}\xi)+(p^{*}g)V+\Pi_{can}^{\sharp}(p^{*}di_{L}^{*}h)\right)\\
		&=r(\xi+di_{L}^{*}h)+g\gamma,
		\end{align*}
		where we used the correspondence \eqref{cor} to obtain the last equality. This proves one inclusion.
		
		For the reverse inclusion, given $\xi\in\Omega^{1}_{cl}(L)$ and $g\in H^{0}(L)$, we get a Poisson vector field 
		\[
		\widetilde{\Pi}^{\sharp}(pr^{*}\xi)+(pr^{*}g)V\in\mathfrak{X}(U),
		\]
		and its vertical component along $L$ is
		\[
		P\left(\widetilde{\Pi}^{\sharp}(pr^{*}\xi)+(pr^{*}g)V\right)=P\left(\Pi_{can}^{\sharp}(p^{*}\xi)+(p^{*}g)V\right)=r(\xi)+g\gamma.\qedhere
		\]
	\end{proof}
	
	\subsubsection{\underline{Smoothness of the moduli space by Poisson isotopies}}\label{subsubsec:poismod}
	
	The moduli space $\mathcal{M}_{U}^{Poiss}(L)$ is not smooth in general, since its formal tangent space can change drastically from point to point. For instance, let us consider the same example as in \S\ref{subsubsec:hammod}, i.e. $(\mathbb{T}^{2}\times\mathbb{R}^{2},\theta_{1},\theta_{2},\xi_{1},\xi_{2})$ with log-symplectic structure
	\[
	\Pi=\partial_{\theta_{1}}\wedge\xi_{1}\partial_{\xi_{1}}+\partial_{\theta_{2}}\wedge\xi_{2}
	\]
	and Lagrangian $L=\mathbb{T}^{2}\times\{(0,0)\}$. Consider again a Lagrangian section $(0,c)\in\Gamma(T^{*}\mathcal{F}_{L}\times\RR)$ for nonzero $c\in\RR$; its graph lies outside the singular locus. By symplectic geometry, $[(0,c)]$ is an isolated point in the moduli space $\mathcal{M}_{U}^{Poiss}(L)$ and therefore
	\[
	T_{[(0,c)]}\mathcal{M}_{U}^{Poiss}(L)=0.
	\] 
	On the other hand, we have
	\begin{align*}
	T_{[0]}\mathcal{M}_{U}^{Poiss}(L)&=\frac{\Omega^{1}_{cl}(\mathcal{F}_{L})}{\text{Im}\left(r:\Omega^{1}_{cl}(L)\rightarrow\Omega^{1}_{cl}(\mathcal{F}_{L})\right)}\oplus H^{0}(\mathcal{F}_{L})\\
	&\cong\frac{C^{\infty}(\mathbb{T}^{2})}{\left\{f\in C^{\infty}(\mathbb{T}^{2}):\frac{\partial}{\partial\theta_{1}}\left(\int_{S^{1}}fd\theta_{2}\right)=0\right\}}\oplus C^{\infty}(S^{1}),
	\end{align*}
	which is infinite dimensional. Here we used Remark \ref{rem:geominter} i) to compute the first summand.
	
	\subsubsection{\underline{Rigidity and Poisson isotopies}}\label{subsubsec:poisrig}
	We now address rigidity of Lagrangians under the equivalence relation by Poisson isotopies. As in the case of Hamiltonian equivalence, we call a Lagrangian $L$ \emph{infinitesimally rigid} under Poisson equivalence if the formal tangent space $T_{[0]}\mathcal{M}_{U}^{Poiss}(L)$ is zero. A Lagrangian $L$ is called \emph{rigid} under Poisson equivalence if small deformations of $L$ are Poisson equivalent with $L$. Rigidity is a very restrictive property: since Poisson diffeomorphisms fix the singular locus of the log-symplectic structure, a Lagrangian $L$ can only be rigid if small deformations of it stay inside the singular locus.

	We will restrict ourselves to Lagrangians $L$ that are compact and connected. It turns out that asking for infinitesimal rigidity under Poisson equivalence is only a little weaker than asking for infinitesimal rigidity under Hamiltonian equivalence, as the next lemma shows.
	\begin{lemma}\label{finitedim}
		Let $L$ be a Lagrangian that is compact, connected and infinitesimally rigid under Poisson equivalence. Then $H^{1}(\mathcal{F}_{L})$ is finite dimensional.
	\end{lemma} 
	\begin{proof}
		Since $L$ is compact, we know that $H^{1}(L)$ is finite dimensional. Choose a basis $\{[\beta_{1}],\ldots,[\beta_{k}]\}$ of $H^{1}(L)$. If $\alpha\in\Omega^{1}_{cl}(\mathcal{F}_{L})$ is a closed foliated one-form, then infinitesimal rigidity implies that $\alpha=r(\widetilde{\alpha})+c\gamma$ for some $\widetilde{\alpha}\in\Omega^{1}_{cl}(L)$ and $c\in\mathbb{R}$, {see Prop. \ref{prop:poisequiv}}. Since $\widetilde{\alpha}$ can be written as $\widetilde{\alpha}=c_{1}\beta_{1}+\cdots+c_{k}\beta_{k}+dh$ for some $c_{1},\ldots,c_{k}\in\mathbb{R}$ and $h\in C^{\infty}(L)$, we get
		\[
		\alpha=c_{1}r(\beta_{1})+\cdots+c_{k}r(\beta_{k})+d_{\mathcal{F}_{L}}h+c\gamma.
		\]
		Therefore $H^{1}(\mathcal{F}_{L})$ is spanned by $\{[r(\beta_{1})],\ldots,[r(\beta_{k})],[\gamma]\}$, hence finite dimensional.
	\end{proof}

	This implies that Lagrangians $L$ for which $\mathcal{F}_{L}$ is the foliation by fibers of a fiber bundle over $S^{1}$ are never rigid, not even infinitesimally.
	\begin{cor}\label{notrigid}
		If $L$ is a compact Lagrangian for which $\mathcal{F}_{L}$ is the foliation by fibers of a fiber bundle $p:L\rightarrow S^{1}$, then $L$ is not infinitesimally rigid under Poisson equivalence.
	\end{cor}
	\begin{proof}
		Assume to the contrary that $L$ is infinitesimally rigid. By Lemma \ref{finitedim}, we know that $H^{1}(\mathcal{F}_{L})\cong\Gamma(\mathcal{H}^{1})$ is finite dimensional. So $\mathcal{H}^{1}$ has to be of rank zero, which implies that $H^{1}(\mathcal{F}_{L})=0$. Consequently, $\gamma$ is exact, and then {Theorem} \ref{H} $i)$ guarantees that $H^{0}_{\gamma}(\mathcal{F}_{L})$ is nonzero. This contradicts that the infinitesimal moduli space \eqref{formaltangent} is zero. So $L$ cannot be infinitesimally rigid.
	\end{proof}
	\begin{remark}
		Alternatively, one can obtain Corollary \ref{notrigid} by using the flat connection $\nabla$ on $\mathcal{H}^{1}$, which was defined in \eqref{connection}. Assuming that $L$ is infinitesimally rigid, fix an open $U\subset S^{1}$ and a frame $\{\sigma_{\eta_{1}},\ldots,\sigma_{\eta_{m}}\}$ for $\mathcal{H}^{1}|_{U}$ consisting of flat sections. If $\alpha\in\Omega^{1}_{cl}(\mathcal{F}_{L})$, then infinitesimal rigidity implies that $\alpha=r(\widetilde{\alpha})+c\gamma$ for some $\widetilde{\alpha}\in\Omega^{1}_{cl}(L)$ and $c\in\mathbb{R}$. Note that the section $\sigma_{r(\widetilde{\alpha})}\in\Gamma(\mathcal{H}^{1})$ is flat, since for all $Y\in\mathfrak{X}(S^{1})$ we have
		\[
		\nabla_{Y}\sigma_{r(\widetilde{\alpha})}=\sigma_{r\left(\pounds_{\overline{Y}}\widetilde{\alpha}\right)}=\sigma_{d_{\mathcal{F}_{L}}\iota_{\overline{Y}}\widetilde{\alpha}}=0,
		\]
		where we used Cartan's magic formula. It follows that
		\[
		\sigma_{\alpha}|_{U}=c_{1}\sigma_{\eta_{1}}+\cdots+c_{m}\sigma_{\eta_{m}}+c\sigma_{\gamma}|_{U}
		\]
		for constants $c_{1},\ldots,c_{k},c\in\mathbb{R}$. This means that necessarily $H^{1}(\mathcal{F}_{L})=0$, and we obtain a contradiction as in the proof of Corollary \ref{notrigid}.
	\end{remark}

	So fibrations over $S^{1}$ don't give examples of rigid Lagrangians. However, if the foliation $\mathcal{F}_{L}$ on $L$ has dense leaves, then we do obtain an interesting rigidity statement: infinitesimal rigidity implies rigidity with respect to the Fr\'echet $\mathcal{C}^{\infty}$-topology.

	\begin{prop}\label{rigidity}
		Let $L$ be a compact, connected Lagrangian whose induced foliation $\mathcal{F}_{L}$ has dense leaves. Assume that $L$ is infinitesimally rigid under Poisson equivalence. Then there exists a neighborhood $\mathcal{V}\subset\left(\Gamma(T^{*}\mathcal{F}_{L}\times\mathbb{R}),\mathcal{C}^{\infty}\right)$ of $0$ such that if $\text{Graph}(\alpha,f)$ is Lagrangian for $(\alpha,f)\in\mathcal{V}$, then $(\alpha,f)$ is Poisson equivalent with the zero section of $T^{*}\mathcal{F}_{L}\times\mathbb{R}$.
	\end{prop}
	\begin{proof}
		Infinitesimal rigidity implies that $H^{0}_{\gamma}(\mathcal{F}_{L})=0$, so $\gamma$ is not foliated exact by $ii)$ of {Theorem} \ref{H}. Moreover, $H^{1}(\mathcal{F}_{L})$ is finite dimensional by Lemma \ref{finitedim}. By Proposition \ref{denserestr}, we obtain a neighborhood $\mathcal{V}\subset\left(\Gamma(T^{*}\mathcal{F}_{L}\times\mathbb{R}),\mathcal{C}^{\infty}\right)$ of $0$ such that if $\text{Graph}(\alpha,f)$ is Lagrangian for $(\alpha,f)\in\mathcal{V}$, then $f\equiv0$.
		To show that $\mathcal{V}$ satisfies the criteria, we distinguish between two cases.
		\newline
		\vspace{-0.3cm}
		\newline
		\indent \underline{Case 1:} $\gamma$ extends to a closed one-form on $L$.
		The assumption of infinitesimal rigidity then implies that $\Omega^{1}_{cl}(\mathcal{F}_{L})=\text{Im}\big(r:\Omega^{1}_{cl}(L)\rightarrow\Omega^{1}_{cl}(\mathcal{F}_{L})\big)$. So if $(\alpha,f)=(\alpha,0)\in\mathcal{V}$ is such that the graph of $(\alpha,0)\in\Gamma(T^{*}\mathcal{F}_{L}\times\mathbb{R})$ is Lagrangian, then we have $\alpha=r(\widetilde{\alpha})$ for some $\widetilde{\alpha}\in\Omega^{1}_{cl}(L)$. The time 1-flow of the Poisson vector field $\widetilde{\Pi}^{\sharp}(pr^{*}\widetilde{\alpha})$ then takes $L$ to $\text{Graph}(\alpha,0)$. 
		\newline
		\vspace{-0.3cm}
		\newline
		\indent \underline{Case 2:} $\gamma$ does not extend to a closed one-form on $L$. In this case, infinitesimal rigidity implies that $\big(\Omega^{1}_{cl}(\mathcal{F}_{L}),\mathcal{C}^{\infty}\big)$ splits into an algebraic direct sum 
		\begin{equation}\label{sum}
		\Omega^{1}_{cl}(\mathcal{F}_{L})=\text{Im}\big(r:\Omega^{1}_{cl}(L)\rightarrow\Omega^{1}_{cl}(\mathcal{F}_{L})\big)\oplus\mathbb{R}\gamma.
		\end{equation}
		Since $r$ is $\mathcal{C}^{\infty}$-continuous, linear and $\text{Im}(r)\subset\Omega^{1}_{cl}(\mathcal{F}_{L})$ is of finite codimension, we get that $\text{Im}(r)\subset\big(\Omega^{1}_{cl}(\mathcal{F}_{L}),\mathcal{C}^{\infty}\big)$ is closed. This implies that \eqref{sum} is in fact a topological direct sum: $\mathbb{R}\gamma$ is an algebraic complement to a maximal closed subspace, and therefore a topological complement \cite[Theorem 4.9.5]{topvec}. So the projection onto the second summand of \eqref{sum} is continuous, and therefore we get a continuous map
		\[
		p_{2}:(\Omega^{1}_{cl}(\mathcal{F}_{L}),\mathcal{C}^{\infty}\big)\rightarrow\mathbb{R}:r(\widetilde{\alpha})+c\gamma\mapsto c.
		\]
		This implies that, shrinking the neighborhood $\mathcal{V}$ constructed above if necessary, we can assume that $p_{2}(\pounds_{X}\alpha)<1$ whenever $(\alpha,f)=(\alpha,0)\in\mathcal{V}$ is a Lagrangian section. 
		
		Now suppose that $(\alpha,f)=(\alpha,0)\in\mathcal{V}$ is such that the graph of $(\alpha,0)\in\Gamma(T^{*}\mathcal{F}_{L}\times\mathbb{R})$ is Lagrangian. We decompose $\alpha$ and $\pounds_{X}\alpha$ in the direct sum \eqref{sum}:
		\begin{equation}\label{alphalx}
		\begin{cases}
		\alpha=r(\xi)+C\gamma\\
		\pounds_{X}\alpha=r(\eta)+K\gamma
		\end{cases}
		\end{equation}
		for $\xi,\eta\in\Omega^{1}_{cl}(L)$ and $C,K\in\mathbb{R}$ with $K<1$. We define smooth families $\xi_{s}\in\Omega^{1}_{cl}(L)$ and $C_{s}\in\mathbb{R}$ for $s\in[0,1]$ by the formulas
		\begin{equation}\label{eq:ksis}
		\xi_{s}:=\xi+\frac{C}{1-sK}s\eta,\hspace{1cm}
		C_{s}:=\frac{C}{1-sK}.
		\end{equation}
		Note that the denominator $1-sK$ occurring in these expressions is never zero for $s\in[0,1]$ since $K<1$. We claim that the isotopy $\phi_{s}$ generated by the time-dependent Poisson vector field $\widetilde{\Pi}^{\sharp}(pr^{*}\xi_{s})+C_{s}V$  takes the zero section of $T^{*}\mathcal{F}_{L}\times\mathbb{R}$ to $graph(\alpha,0)$, or more precisely, that $\phi_{s}(L)=graph(s\alpha,0)$ for $s\in[0,1]$. To prove this, by \cite[Lemma 3.15]{equivalences} it is enough to check that 
		\begin{equation}\label{pathvelocity}
		\frac{d}{ds}(s\alpha,0)=P_{(s\alpha,0)}\left(\widetilde{\Pi}^{\sharp}(pr^{*}\xi_{s})+C_{s}V\right),
		\end{equation}
		where $P_{(s\alpha,0)}$ denotes the vertical projection  induced by the direct sum decomposition of $T(T^{*}\mathcal{F}_{L}\times\mathbb{R})|_{graph(s\alpha,0)}$ into $Tgraph(s\alpha,0)$ and the vertical bundle along $graph(s\alpha,0)$. Computing the right hand side of \eqref{pathvelocity} gives
		\begin{align*}
		P_{(s\alpha,0)}\left(\widetilde{\Pi}^{\sharp}(pr^{*}\xi_{s})+C_{s}V\right)&=P_{s\alpha}\left(\Pi_{can}^{\sharp}(p^{*}\xi_{s})+C_{s}V\right)\\
		&=r(\xi_{s})+C_{s}P_{s\alpha}(V_{vert}+V_{lift})\\
		&=r(\xi_{s})+C_{s}(\gamma-\pounds_{X}(s\alpha))\\
		&=r\left(\xi+\frac{C}{1-sK}s\eta\right)+\frac{C}{1-sK}\left(\gamma-sr(\eta)-sK\gamma\right)\\
		&=r(\xi)+\frac{Cs}{1-sK}r(\eta)+\frac{C(1-sK)}{1-sK}\gamma-\frac{Cs}{1-sK}r(\eta)\\
		&=r(\xi)+C\gamma\\
		&=\alpha.
		\end{align*}
		Here we used the correspondence \eqref{cor} in the second equality, Lemma \ref{projection} below in the third equality and the expressions \eqref{alphalx},{\eqref{eq:ksis}} in the fourth equality. This finishes the proof. 
	\end{proof}
	
	{By definition of the $\mathcal{C}^{\infty}$-topology, one can rephrase the above proposition as follows: infinitesimal rigidity of $L$ implies the existence of some $k\in\mathbb{N}$ such that $L$ is $\mathcal{C}^{k}$-rigid.}
	
	\begin{lemma}\label{projection}
		Let $\alpha\in\Gamma(T^{*}\mathcal{F}_{L})$, and denote by $P_{\alpha}$ the vertical projection induced by the splitting of $T(T^{*}\mathcal{F}_{L})|_{graph(\alpha)}$ into $Tgraph(\alpha)$ and the vertical bundle along $graph(\alpha)$. We then have
		\[
		P_{\alpha}(V_{lift})=-\pounds_{X}\alpha.
		\]
	\end{lemma}
	
	\begin{proof}
		Denote by $\phi^{-\alpha}$ the translation map
		\[
		\phi^{-\alpha}:T^{*}\mathcal{F}_{L}\rightarrow T^{*}\mathcal{F}_{L}:(p,\xi)\mapsto(p,\xi-\alpha(p)),
		\]
		and let $P:=P_{0}$ be the vertical projection along the zero section. We then have a commutative diagram
		\[
		\begin{tikzcd}[column sep=large, row sep=large]
		T(T^{*}\mathcal{F}_{L})|_{graph(\alpha)}\arrow{r}{(\phi^{-\alpha})_{*}}\arrow{dr}{P_{\alpha}}&T(T^{*}\mathcal{F}_{L})|_{L}\arrow{d}{P}\\
		&\Gamma(T^{*}\mathcal{F}_{L})
		\end{tikzcd},
		\]
		so the lemma follows from the equality \eqref{toshow}.
	\end{proof}
	
	\begin{ex}\label{ex:C^1rigid}
		Let $L=(\mathbb{T}^{2},\theta_{1},\theta_{2})$ with Kronecker foliation $T\mathcal{F}_{L}=\text{Ker}(d\theta_{1}-\lambda d\theta_{2})$ for generic (i.e. not Liouville) $\lambda\in\mathbb{R}\setminus\mathbb{Q}$. Let $\xi$ be the fiber coordinate on $T^{*}\mathcal{F}_{L}$ corresponding with the frame $\{d\theta_{2}\}$. As in eq. \eqref{K}, we take a log-symplectic structure {of the form}
		\[
		\Big(T^{*}\mathcal{F}_{L}\times\mathbb{R},\widetilde{\Pi}:=(C\partial_{\theta_{1}}+K\partial_{\xi})\wedge t\partial_{t}+\left(\lambda\partial_{\theta_{1}}+\partial_{\theta_{2}}\right)\wedge\partial_{\xi}\Big),
		\]
		where {now both $C,K\in\RR$ are nonzero}. As for generic $\lambda\in\mathbb{R}\setminus\mathbb{Q}$, we have $H^{1}(\mathcal{F}_{L})=\mathbb{R}[d\theta_{2}]$, it is clear that every element of $\Omega^{1}_{cl}(\mathcal{F}_{L})$ extends to a closed one-form on $L$. Moreover, since $\gamma=Kd\theta_{2}$ is not exact, we have that $H^{0}_{\gamma}(\mathcal{F}_{L})=0$ by {Theorem} \ref{H} $ii)$. So $L$ is infinitesimally rigid:
		\[
		T_{[0]}\mathcal{M}_{U}^{Poiss}(L)=\frac{\Omega^{1}_{cl}(\mathcal{F}_{L})}{\text{Im}\left(r:\Omega^{1}_{cl}(L)\rightarrow\Omega^{1}_{cl}(\mathcal{F}_{L})\right)+\mathbb{R} \gamma}\oplus H^{0}_{\gamma}(\mathcal{F}_{L})=0,
		\]
		and therefore $L$ is $\mathcal{C}^{\infty}$-rigid, by Proposition \ref{rigidity}. 
		
		In this particular example, we in fact know a bit more. We already noted in Remark \ref{generic} that $\mathcal{C}^{1}$-small deformations of $L$ stay inside the singular locus, i.e. they are of the form $(\alpha,f)=(\alpha,0)\in\Gamma(T^{*}\mathcal{F}_{L}\times\mathbb{R})$ for $\alpha\in\Omega_{cl}^{1}(\mathcal{F}_{L})$. Along with the fact that foliated closed one-forms extend to closed one-forms on $L$, this implies that the Lagrangian $L$ is $\mathcal{C}^{1}$-rigid under Poisson equivalence. For if $\widetilde{\alpha}\in\Omega^{1}_{cl}(L)$ is a closed extension of $\alpha$, then the flow of the Poisson vector field $\widetilde{\Pi}^{\sharp}(pr^{*}\widetilde{\alpha})$ takes $L$ to $graph(\alpha,0)$.
		\newline
		\indent
		If instead we take $\lambda\in\mathbb{R}\setminus\mathbb{Q}$ to be a Liouville number, then $L$ is not infinitesimally rigid by Lemma \ref{finitedim}, since in that case $H^{1}(\mathcal{F}_{L})$ is infinite dimensional. 	
	\end{ex}

	\section{Appendix}
	
This short appendix summarizes some facts about Liouville numbers and Fr\'echet spaces.

	\subsection{Liouville numbers}
	\leavevmode
	\vspace{0.1cm}
	
	We collect some facts about Liouville numbers that are used in  \S\ref{subs}.
	
	\begin{defi}\label{liouville}
		A Liouville number is a real number $\alpha\in\mathbb{R}$ with the property that, for all integers $p\geq 1$, there exist integers $m_{p},n_{p}\in\mathbb{Z}$ such that $n_{p}>1$ and
		\[
		0<\left|\alpha-\frac{m_{p}}{n_{p}}\right|<\frac{1}{n_{p}^{p}}.
		\]
	\end{defi}
	
	Liouville numbers are irrational (even transcendental, see \cite[Theorem 4.5]{molin}). 
	
	\begin{remark}
		For any sequence $(m_{p},n_{p})_{p\in\mathbb{N}}$ as in Definition \ref{liouville}, the set of denominators $\{n_{p}:p\in\mathbb{N}\}$ is unbounded. Indeed, assume to the contrary that this set is bounded by some constant $M$. Since $n_{p}>1$, the sequence $(m_{p}/n_{p})_{p\in\mathbb{N}}$ converges to $\alpha$. As there are only finitely many fractions $a/b$ such that $1<b\leq M$ and $a/b$ lies within distance $1$ of $\alpha$, the sequence $(m_{p}/n_{p})_{p\in\mathbb{N}}$  must have a constant subsequence. This subsequence must also converge to $\alpha$, which implies that $\alpha\in\mathbb{Q}$. This contradiction shows that $\{n_{p}:p\in\mathbb{N}\}$ is unbounded.
	\end{remark}
	
	The next statement is used in the proof of Lemma \ref{dense}. It appears without proof in \cite{bertelson}.
	
	\begin{lemma}\label{liou}
		If $\alpha$ is a Liouville number, then for each integer $p\geq 1$, there exists a pair of integers $(m_{p},n_{p})\in\mathbb{Z}^{2}$ such that
		\[
		|m_{p}+\alpha n_{p}|\leq\frac{1}{(|m_{p}|+|n_{p}|)^{p}}.
		\]
	\end{lemma}
	\begin{proof}
		Since $\alpha$ is Liouville, we can fix a sequence $(M_{p},N_{p})$ for integers $p\geq 1$, satisfying
		\begin{equation}\label{eq:2}
		0<\left|\alpha-\frac{M_{p}}{N_{p}}\right|<\frac{1}{N_{p}^{p}},\hspace{1cm}N_{p}\geq 2.
		\end{equation}
		The sequence $(M_{p}/N_{p})_{p\in\mathbb{N}}$ is convergent hence bounded, so there exists an integer $k\geq 1$ such that
		\begin{equation}\label{bound}
		|M_{p}|\leq 2^{k}N_{p},\hspace{1cm}\forall p\geq 1.
		\end{equation}
		Notice that
		\begin{equation}\label{estimate}
		\left|\alpha-\frac{M_{(k+2)p}}{N_{(k+2)p}}\right|<\frac{1}{N_{(k+2)p}^{(k+2)p}}=\frac{1}{N_{(k+2)p}^{p}\cdot N_{(k+2)p}^{(k+1)p}}\leq \frac{1}{N_{(k+2)p}^{p}\cdot 2^{(k+1)p}},
		\end{equation}
		{using in the last inequality that $N_{(k+2)p}\geq 2$ (see \eqref{eq:2})}.
		Since the function $x\mapsto x^{p}$ is convex on the domain $(0,\infty)$, we have
		\[
		\left(\frac{|N_{(k+2)p}|+|M_{(k+2)p}|}{2}\right)^{p}\leq \frac{|N_{(k+2)p}|^{p}+|M_{(k+2)p}|^{p}}{2},
		\]
		and therefore
		\begin{align}\label{convex}
		\left(|N_{(k+2)p}|+|M_{(k+2)p}|\right)^{p}&\leq 2^{p-1}\left(|N_{(k+2)p}|^{p}+|M_{(k+2)p}|^{p}\right)\nonumber\\
		&\leq 2^{p}\max\left(|N_{(k+2)p}|^{p},|M_{(k+2)p}|^{p}\right)\nonumber\\
		&\leq 2^{p}\cdot 2^{kp}|N_{(k+2)p}|^{p}\nonumber\\
		&=2^{(k+1)p}|N_{(k+2)p}|^{p},
		\end{align}
		using \eqref{bound} in the third inequality. Combining the inequality \eqref{estimate} with \eqref{convex} gives
		\[
		\left|\alpha-\frac{M_{(k+2)p}}{N_{(k+2)p}}\right|<\frac{1}{\left(|N_{(k+2)p}|+|M_{(k+2)p}|\right)^{p}}.
		\]
		Replacing $M_{p}$ by $-M_{p}$, this implies that
		\[
		\left|M_{(k+2)p}+\alpha N_{(k+2)p}\right|<\frac{N_{(k+2)p}}{\left(|N_{(k+2)p}|+|M_{(k+2)p}|\right)^{p}}\leq \frac{1}{\left(|N_{(k+2)p}|+|M_{(k+2)p}|\right)^{p-1}}.
		\]
		So if we set $(m_{p},n_{p}):=\left(M_{(k+2)(p+1)},N_{(k+2)(p+1)}\right)$, then we have
		\[
		|m_{p}+\alpha n_{p}|<\frac{1}{(|m_{p}|+|n_{p}|)^{p}}.\qedhere
		\]
	\end{proof}
	
	\begin{remark}\label{assumptions}
		The proof of Lemma \ref{liou} shows that we can make the additional assumptions $n_{p}\geq p$ and $(m_{p},n_{p})\neq (m_{q},n_{q})$ for $p\neq q$. Indeed, since the set of denominators $\{N_{p}:p\in\mathbb{N}\}$ of the sequence $(M_{p},N_{p})_{p\in\mathbb{N}}$ is unbounded, we can ensure that $N_{p}\geq p$. For if $N_{p}<p$, then we know that there exists $p'>p$ such that the element $(M_{p'},N_{p'})$ satisfies $N_{p'}\geq p$. We then have
		\[
		\left|\alpha-\frac{M_{p'}}{N_{p'}}\right|<\frac{1}{N_{p'}^{p'}}<\frac{1}{N_{p'}^{p}},
		\]
		so we can just replace $(M_{p},N_{p})$ by $(M_{p'},N_{p'})$. It then follows that \[n_{p}=N_{(k+2)(p+1)}\geq (k+2)(p+1)\geq p.\]
		Similarly, we can make sure that $N_{p}\neq N_{q}$ for $q\neq p$, so that also $(m_{p},n_{p})\neq (m_{q},n_{q})$.
	\end{remark}
	
	\subsection{Fr\'echet spaces}
	\leavevmode
	\vspace{0.1cm}
	
	We recall some basic facts about Fr\'echet spaces, which are used in \S\ref{subs} and  \S\ref{subsec:eqrig}. For more details, see for instance \cite{hamilton}.
	\begin{defi}
		A Fr\'echet space is a topological vector space $X$ that satisfies the following three properties:
		\begin{enumerate}[i)]
			\item $X$ is Hausdorff.
			\item The topology on $X$ is induced by a countable family of seminorms $\{\|\cdot\|_{k}\}_{k\geq 0}$.
			\item $X$ is complete.
		\end{enumerate}
	\end{defi}
	By item ii), a base of neighborhoods of $x\in X$ is given by subsets of the form
	\[
	\mathcal{B}_{r}^{k_{1}}(x)\cap\cdots\cap\mathcal{B}_{r}^{k_{n}}(x)
	\] 
	for $n\in\mathbb{N}$ and $r>0$, where $\mathcal{B}_{r}^{k_{j}}(x)$ denotes the open ball
	\[
	\mathcal{B}_{r}^{k_{j}}(x)=\{y\in X:\|y-x\|_{k_{j}}<r\}.
	\]
	A sequence $x_{n}$ converges to $x$ if and only if $\|x_{n}-x\|_{k}$ converges to zero for each $k\geq 0$.
	\begin{ex}
		If $L$ is compact, the space of sections of any vector bundle over $L$ becomes a Fr\'echet space when endowed with the $\mathcal{C}^{\infty}$-topology generated by $\mathcal{C}^{k}$-norms $\|\cdot\|_{k}$. We recall the construction of such norms in the situation that is of interest to us. Let $(L,\mathcal{F}_{L})$ be a compact manifold endowed with a codimension-one foliation; we will define $\mathcal{C}^{k}$-norms on the space $\Omega^{\bullet}(\mathcal{F}_{L})=\Gamma(\wedge^{\bullet}(T^{*}\mathcal{F}_{L}))$ of foliated forms of fixed degree. Fix a finite cover $\{U_{1},\ldots,U_{m}\}$ of $L$ consisting of foliated charts with coordinates $(x_{1},\ldots,x_{n-1},x_{n})$, such that plaques of $\mathcal{F}_{L}$ are level sets of $x_{n}$. Choose open subsets $V_{i}$ for $i=1,\ldots,m$ that still cover $L$ and have compact closures satisfying $\overline{V_{i}}\subset U_{i}$. The $k$-norm of a foliated form $\eta\in\Omega^{l}(\mathcal{F}_{L})$ with coordinate representation
		\[
		\eta|_{U_{j}}=\sum_{1\leq i_{1}<\cdots<i_{l}\leq n-1}g^{j}_{i_{1}\cdots i_{l}}dx_{i_{1}}\wedge\cdots\wedge dx_{i_{l}}
		\]
		is then
		\[
		\|\eta\|_{k}=\sum_{1\leq j\leq m}\sum_{1\leq i_{1}<\cdots<i_{l}\leq n-1}\sum_{|\alpha|\leq k}\sup_{p\in \overline{V_{j}}}\left|\frac{\partial^{\alpha}g^{j}_{i_{1}\cdots i_{l}}}{\partial x^{\alpha}}(p)\right|.
		\]
	\end{ex}
	
	Recall also that a closed subspace of a Fr\'echet space is itself a Fr\'echet space. Finally, it is useful to note that, if $X$ and $Y$ are vector spaces whose topologies are generated by families of seminorms $\{\|\cdot\|_{k}\}$ and $\{\|\cdot\|'_{k}\}$ respectively, then a linear map $L:X\rightarrow Y$ is continuous if and only if for every $k\in\mathbb{N}$, there exist $n_{1},\ldots,n_{l}\in\mathbb{N}$ and $C\in\mathbb{R}$ such that
	\[
	\|L(x)\|'_{k}\leq C\sum_{j=1}^{l}\|x\|_{n_{j}}.
	\]

	\bibliographystyle{abbrv} 

\end{document}